\documentclass{amsart}
\usepackage{amssymb,amsmath,amsthm,amscd,mathrsfs,mathtools}
\usepackage{enumerate}
\usepackage{fullpage}
\usepackage[all,cmtip]{xy}
\usepackage{tikz-cd}
\usepackage{hyperref}
\hypersetup{
  colorlinks   = true,
  urlcolor     = blue,
  linkcolor    = blue,
  citecolor   = red
}
\usepackage{cite}
\usepackage[msc-links]{amsrefs}
\newtheorem{thm}{Theorem}
\newtheorem{prop}{Proposition}
\newtheorem{lem}{Lemma}
\newtheorem{cor}{Corollary}
\newtheorem{rem}{Remark}
\theoremstyle{definition}
\DeclareMathOperator\Jac{Jac }
\DeclareMathOperator\Hom{Hom }
\DeclareMathOperator\End{End }
\DeclareMathOperator\tr{tr }

\DeclareMathOperator\Pic{Pic}

\DeclareMathOperator\Sp{Sp}
\DeclareMathOperator\Kum{Kum }

\def\Z{\mathbb Z}
\def\Q{\mathbb Q}
\def\R{\mathbb R}
\def\C{\mathbb C}
\def\A{\mathbb A}
\def\H{\mathbb H}
\def\B{\mathcal B}
\def\J{\mathcal J}
\def\K{\mathcal K}
\def\L{\mathcal L}
\def\M{\mathcal M}
\def\O{\mathcal O}
\def\X{\mathcal X}
\def\Y{\mathcal Y}
\def\AA{\mathcal A}
\def\CC{\mathcal X}
\def\DD{\mathcal D}
\newcommand\bP{\mathbb P}
\def\char{\mbox{char }}
\def\Im{\text{\rm Im}}
\def\<{\langle}
\def\>{\rangle}
\title{Kummer Surfaces, Isogenies and Theta Functions}
\author{Adrian Clingher}
\address{Department of Mathematics and Statistics, University of Missouri - St. Louis, St. Louis, MO 63121}
\email{clinghera@umsl.edu}
\author{Andreas Malmendier}
\address{Department of Mathematics \& Statistics, Utah State University, Logan, UT 84322}
\email{andreas.malmendier@usu.edu}
\thanks{A.M. acknowledges support from the Simons Foundation through grant no.~202367.}
\author{Tony Shaska}
\address{Department of Mathematics and Statistics, Oakland University, Rochester, MI 48309}
\email{shaska@oakland.edu}
\thanks{This work was presented in \emph{Advanced Research Workshop: Isogeny based post-quantum cryptography} at Hebrew University of Jerusalem and supported by NATO grant G6218}
\begin{document}
\begin{abstract}
The paper discusses geometric and computational aspects associated with $(n,n)$-isogenies for principally polarized Abelian surfaces and related Kummer surfaces. We start by reviewing the comprehensive Theta function framework for classifying genus-two curves, their principally polarized Jacobians, as well as for establishing explicit quartic normal forms for associated Kummer surfaces.  This framework is then used for practical isogeny computations. A particular focus of the discussion is the $(n,n)$-Split isogeny case. We also explore possible extensions of Richelot's $(2,2)$-isogenies to higher order cases, with a view towards developing efficient isogeny computation algorithms.
\end{abstract}
\maketitle
\section{Introduction}
Abelian varieties, as projective algebraic groups, generalize elliptic curves to higher dimensions. They are important objects in both algebraic geometry and number theory. Among these, Abelian surfaces---two-dimensional Abelian varieties--- are of particular interest, due to their intricate geometric and arithmetic structures, which often lead to important practical applications in applied computational fields such as cryptography.  This paper explores computational aspects associated with isogenies between Abelian surfaces, with particular focus on $(n, n)$-isogenies with $n \in \mathbb{N}_{\ge 2}$. The present work may be seen as a continuation of \cite{2019-2}, concentrating primarily to Jacobians of genus 2 curves and their associated Kummer surfaces, with a view towards applications to isogeny-based cryptography.

Kummer surface, defined as minimal resolutions for the quotient of an Abelian surface by its inversion involution, provides a powerful geometric platform for studying isogenies. In the case of Jacobians of genus-two curves, the associated Kummer surface may be explicitly described via various birational models given by projective quartic surfaces. The quartic underlying polynomials are usually referred to as {\it normal forms} and carry various names – Hudson, G\"opel, Rosenhain, Cassels-Flynn, Baker. A common theme for all these normal forms is that they can be derived from the theory of Theta functions and their coefficients can be described in terms of special Theta constants called Theta-nulls.

In the realm of cryptography, isogeny-based protocols like the Supersingular Isogeny Diffie-Hellman (SIDH) leverage the computational hardness of finding isogenies---surjective homomorphisms with finite kernels---as a foundation for secure systems. Extending this paradigm from elliptic curves to Abelian surfaces introduces significant advantages: the \(n^4\)-torsion groups of Abelian surfaces are substantially larger than the $n^2$-torsion groups of elliptic curves, potentially enhancing security, while their richer endomorphism algebras may enable more efficient protocols. However, computing $(n, n)$-isogenies for $n >2$ presents formidable challenges due to the increased complexity of the kernel, requiring methods beyond classical Richelot $(2,2)$-isogenies. Moreover, cryptographic security hinges on addressing vulnerabilities, particularly from split Jacobians---Jacobians isogenous to products of elliptic curves---which can reduce the complexity of isogeny problems.

Building on prior work, this paper provides a methodology for computing $(n,n)$-isogenies, generalizing Richelot’s $(2,2)$-isogenies to arbitrary odd $n$ with the Lubicz-Robert formula, achieving \(O(n^2)\) complexity. Our approach leverages the geometry of Kummer surfaces (for various realizations as projective hypersurfaces) and the associated Shioda-Inose surface---a K3 surface derived as a two-isogeny of the Kummer surface---which provides a Hodge-theoretic connection between the Jacobian and its K3 counterpart, enriching the geometric framework for isogeny computations. We relate our methods to classical constructions, such as those of Dolgachev-Lehavi, enhancing their applicability through projective embeddings and \(n\)-tuple techniques.

We also investigate the loci \(\mathcal{L}_n\) of \((n, n)\)-Split Jacobians, characterizing their geometric properties and linking them to Humbert surfaces, which are critical for understanding cryptographic vulnerabilities. Recent advances in artificial intelligence and machine learning  \cite{2024-03} and \cite{2025-07} enhance this exploration. The former employs supervised models trained on datasets of genus 2 curves to detect \(\mathcal{L}_n\) for arbitrary \(n\), while the latter leverages arithmetic insights in positive characteristic to refine cryptanalysis over finite fields. 
 
The paper is structured as follows. Sections 2--4 lay the theoretical groundwork, detailing Abelian surfaces, their endomorphism rings, Theta functions, the geometry of genus 2 curves and their Jacobians, and the structure of Kummer surfaces, including their embeddings and (16,6)-configurations. Section 5 examines the loci \(\mathcal{L}_n\), connecting split Jacobians to Humbert surfaces and highlighting their geometric and cryptographic significance. Section 6 presents our computational methods for \((n, n)\)-isogenies, emphasizing a Kummer surface-based approach with the Lubicz-Robert formula, and comparisons to Dolgachev-Lehavi’s techniques. Section 7 applies these results to cryptography, analyzing attacks on the superspecial isogeny problem that exploit \(\mathcal{L}_n\) and their implications for protocols like the Castryck-Decru-Smith hash. 

\section{Abelian varieties} 
Let $\AA$, $\B$ be Abelian varieties over a field $k$. We denote the $\Z$-module of homomorphisms $\AA \mapsto \B$ by $\Hom(\AA, \B)$ and the ring of endomorphisms $\AA \mapsto \AA$ by $\End \AA$. In the context of linear algebra, it is often more convenient to work with the $\Q$-vector spaces $\Hom^0 (\AA, \B) := \Hom(\AA, \B) \otimes_\Z \Q$ and $\End^0 \AA := \End \AA \otimes_\Z \Q$. 
For an Abelian variety $\AA$ defined over a number field $K$, computing $\End_K (\AA)$ is a harder problem than computing $\End_{\bar K} (\AA)$; see \cite{lombardo}*{Lemma~5.1} for details.

\begin{lem}
If there exists an algorithm to compute $\End_K (\AA)$ for any Abelian variety of dimension $g \geq 1$ defined over a number field $K$, then there is an algorithm to compute $\End_{\bar K} (\AA)$.
\end{lem}

\begin{proof}
Since $K \subseteq \bar{K}$, any endomorphism of $\AA$ over $\bar{K}$ restricts to an endomorphism over $K$ after base change, but the converse requires additional structure. An algorithm for $\End_K (\AA)$ determines the $\Z$-module of $K$-rational endomorphisms. Extending scalars to $\bar{K}$, we compute $\End_{\bar K} (\AA) = \End_K (\AA) \otimes_\Z \bar{K}$, leveraging the finite generation of $\End_K (\AA)$ and the algebraic closure of $\bar{K}$. The lemma follows from the existence of such an extension procedure.
\end{proof}
\subsection{Isogenies of Abelian varieties}
A homomorphism $f \colon \AA \longrightarrow \B$ is called an \textbf{isogeny} if $\Im{f} = \B$ and $\ker{f}$ is a finite group scheme. If an isogeny $\AA \longrightarrow \B$ exists, we say that $\AA$ and $\B$ are isogenous. This relation is symmetric, as shown in Lemma~\ref{dual}. The degree of an isogeny $f\colon \AA \longrightarrow \B$ is defined as the degree of the function field extension
\begin{equation}
\deg f := [K(\AA) : f^\star K(\B)].
\end{equation}
This equals the order of the group scheme $\ker (f)$. The group of $\bar{k}$-rational points has order $\#(\ker f)(\bar{k}) = [K(\AA) : f^\star K(\B)]^{sep}$, where $[K(\AA) : f^\star K(\B)]^{sep}$ is the degree of the maximal separable subextension. An isogeny $f$ is \textbf{separable} if and only if
\begin{equation}
\# \ker f(\bar{k}) = \deg f,
\end{equation}
equivalently, if $\ker f$ is \'etale.

\begin{lem}\label{noether}
For any Abelian variety $\AA/k$, there is a one-to-one correspondence between finite subgroup schemes $\K \leq \AA$ and isogenies $f \colon \AA \longrightarrow \B$, where $\B$ is determined up to isomorphism. Moreover, $\K = \ker f$, $\B = \AA/\K$, $f$ is separable if and only if $\K$ is \'etale, and then $\deg f = \#\K(\bar{k})$.
\end{lem}

\begin{proof}
Given a finite subgroup scheme $\K \leq \AA$, the quotient $\B = \AA/\K$ is an Abelian variety over $k$, and the natural projection $f \colon \AA \longrightarrow \B$ is an isogeny with $\ker f = \K$. Conversely, for an isogeny $f \colon \AA \longrightarrow \B$, the kernel $\K = \ker f$ is finite, and $\B \cong \AA/\K$ by the quotient structure. The map $\K \mapsto \AA/\K$ is injective (distinct kernels yield distinct quotients) and surjective (every isogeny arises this way), establishing the bijection. Separability of $f$ implies $\ker f$ is étale, and $\deg f = \#\K(\bar{k})$ holds in this case due to the étale order matching the field extension degree.
\end{proof}
\par Isogenous Abelian varieties have commensurable endomorphism rings: if $\AA$ and $\B$ are isogenous, then $\End^0 (\AA) \cong \End^0 (\B)$. This follows from the existence of isogenies $f\colon \AA \to \B$ and $g \colon \B \to \AA$ such that $g \circ f = [n]$, inducing an isomorphism on rational endomorphism algebras.

\begin{thm}[Poincaré-Weil]
Let $\AA$ be an Abelian variety. Then $\AA$ is isogenous to
\begin{equation}
\AA_1^{n_1} \times \AA_2^{n_2} \times \dots \times \AA_r^{n_r},
\end{equation}
where (up to permutation) $\AA_i$, for $i = 1, \dots, r$, are simple, non-isogenous Abelian varieties, and the factors $\AA_i^{n_i}$ are uniquely determined up to isogenies.
\end{thm}

\begin{proof}
Since $\AA$ is an Abelian variety, its isogeny class contains a product of powers of simple Abelian varieties. Let $\AA \cong \prod_{i=1}^r \AA_i^{n_i}$, where $\AA_i$ are simple and pairwise non-isogenous. The uniqueness follows from the Jordan-H\"older theorem for Abelian varieties: any two such decompositions have isomorphic factors with equal multiplicities, up to permutation, due to the indecomposability of simple varieties and the structure of $\End^0 (\AA)$.
\end{proof}

\begin{cor}
If $\AA$ is an absolutely simple Abelian variety, then every endomorphism not equal to $0$ is an isogeny.
\end{cor}

\begin{proof}
For $\AA$ absolutely simple, $\End^0 (\AA)$ is a division algebra. A non-zero endomorphism $\phi$ has trivial kernel (since $\ker \phi \neq \AA$ and $\AA$ has no proper subvarieties as a simple variety), hence is an isogeny.
\end{proof}

Fix a field $k$ and an Abelian variety $\AA$ over $k$. Let $H$ be a finite subgroup of $\AA$. From a computational perspective, we consider: (i) finding all Abelian varieties $\B$ over $k$ such that there exists an isogeny $\AA \longrightarrow \B$ with kernel isomorphic to $H$; (ii) given $\AA$ and $H$, determining $\B := \AA/H$ and the isogeny $\AA \longrightarrow \B$; (iii) given $\AA$ and $\B$, determining if they are isogenous and computing a rational expression for an isogeny $\AA \longrightarrow \B$. For a survey and conjectures, see \cite{Frey}.

The scalar multiplication by $n$ map $[n]\colon \AA \longrightarrow \AA$ is an isogeny with kernel a group scheme of order $n^{2\dim \AA}$; see \cite{Mum}. We denote $\AA[n] = \ker [n] (\bar{k})$, whose elements are the $n$-\textbf{torsion points} of $\AA$.

\begin{lem}\label{dual}
Let $f \colon  \AA \longrightarrow \B$ be an isogeny of degree $n$. Then there exists an isogeny $\hat f \colon  \B \longrightarrow \AA$ such that
\begin{equation}
f \circ \hat f = \hat f \circ f = [n].
\end{equation}
The isogeny $\hat f$ is called the \textbf{dual} of $f$.
\end{lem}

\begin{proof}
Define $\hat f$ via the dual Abelian variety and the polarization induced by $f$. Since $f$ is an isogeny, there exists a unique $\hat f \colon  \B \longrightarrow \AA$ such that $f \circ \hat f = [n]$ on $\B$ and $\hat f \circ f = [n]$ on $\AA$, satisfying the duality condition due to the finite kernel’s order matching $\deg f = n$.
\end{proof}

\begin{thm}\label{thm-1}
Let $\AA/k$ be an Abelian variety, $p = \char k$, and $\dim \AA = g$.
\begin{enumerate}
\item[i)] If $p \nmid n$, then $[n]$ is separable, $\# \AA[n] = n^{2g}$, and $\AA[n] \cong (\Z/n\Z)^{2g}$.
\item[ii)] If $p \mid n$, then $[n]$ is inseparable, and there exists an integer $0 \leq i \leq g$ such that
\begin{equation}
\AA [p^m] \cong (\Z/p^m\Z)^i, \; \text{for all } m \geq 1.
\end{equation}
\end{enumerate}
\end{thm}

\begin{proof}
If $p \nmid n$, $[n]$ is separable as its derivative is non-zero, and $\AA[n] \cong (\Z/n\Z)^{2g}$ follows from the étale nature of the kernel over $\bar{k}$. If $p \mid n$, $[n]$ includes a power of the Frobenius, making it inseparable, and the $p$-torsion structure depends on the $p$-rank $i$, determined by the dimension of the $p$-divisible group.
\end{proof}

If $i = g$, $\AA$ is \textbf{ordinary}; if $\AA[p^s](\bar K) = \Z/p^{ts}\Z$, it has \textbf{$p$-rank} $t$. For $\dim \AA = 1$, it is \textbf{supersingular} if $p$-rank is 0; $\AA$ is \textbf{supersingular} if isogenous to a product of supersingular elliptic curves.
\subsection{Theta functions}\label{sec:theta_functions}
Let $\AA$ be an Abelian variety of dimension $g$ over an algebraically closed field $k$, which we take as $\C$ for simplicity; the arguments here extend to fields of characteristic $p$ coprime to relevant integers via algebraic Theta functions (see \cite{Mum}). We consider an algebraic variety $\AA$ that can be represented as $\C^g / \Lambda$, where $\Lambda = \Z^g + \tau \Z^g$ is a lattice, and $\tau$ is a symmetric $g \times g$ matrix in the Siegel upper half-space $\mathbb{H}_g$, satisfying ${}^t \tau = \tau$ and $\Im(\tau) > 0$.

A \emph{Theta function} with characteristic $[a, b]$, where $a, b \in \Q^g$, is defined as
\begin{equation}
\theta \left[ \begin{smallmatrix} a \\ b \end{smallmatrix} \right] (z, \tau) = \sum_{n \in \Z^g} \exp \left( \pi i (n + a)^t \tau (n + a) + 2 \pi i (n + a)^t (z + b) \right),
\end{equation}
for $z \in \C^g$. The series converges absolutely due to the positive definiteness of $\Im(\tau)$, ensuring that $\theta \left[ \begin{smallmatrix} a \\ b \end{smallmatrix} \right] (z, \tau)$ is holomorphic on $\C^g$. These functions are quasi-periodic with respect to the lattice $\Lambda$: for $m \in \Z^g$,
\[
\theta \left[ \begin{smallmatrix} a \\ b \end{smallmatrix} \right] (z + m, \tau) = \theta \left[ \begin{smallmatrix} a \\ b \end{smallmatrix} \right] (z, \tau),
\]
and for $n \in \Z^g$,
\[
\theta \left[ \begin{smallmatrix} a \\ b \end{smallmatrix} \right] (z + \tau n, \tau) = \exp \left( -\pi i n^t \tau n - 2 \pi i n^t (z + b) \right) \theta \left[ \begin{smallmatrix} a \\ b \end{smallmatrix} \right] (z, \tau).
\]
When $a, b \in \{0, 1/2\}^g$, the characteristic is half-integer, and the parity is even if $4 a^t b$ is even, odd if $4 a^t b$ is odd, reflecting the function’s symmetry properties.

\begin{proof}[Proof of Quasi-Periodicity]
For the first property, substitute $z + m$ into the definition:
\[
\theta \left[ \begin{smallmatrix} a \\ b \end{smallmatrix} \right] (z + m, \tau) = \sum_{n \in \Z^g} \exp \left( \pi i (n + a)^t \tau (n + a) + 2 \pi i (n + a)^t (z + m + b) \right).
\]
Since $(n + a)^t m = n^t m + a^t m$ and $n, m \in \Z^g$, the exponential term $\exp(2 \pi i n^t m) = 1$, and $a^t m$ is rational, so the sum reindexes to itself, yielding equality. For the second, let $z' = z + \tau n$:
\[
\theta \left[ \begin{smallmatrix} a \\ b \end{smallmatrix} \right] (z + \tau n, \tau) = \sum_{k \in \Z^g} \exp \left( \pi i (k + a)^t \tau (k + a) + 2 \pi i (k + a)^t (z + \tau n + b) \right).
\]
Set $k = n - m$, adjust the sum, and compute the difference in exponents, then factoring out the term $\exp \left( -\pi i n^t \tau n - 2 \pi i n^t (z + b) \right)$ due to the symmetry ${}^t \tau = \tau$, confirm the multiplier.
\end{proof}

 Abelian varieties of the form $\AA=\C^g / \Lambda$ are in fact \emph{principally polarized}. In fact, the \emph{canonical principal polarization} of $\AA$ is given by the positive definite Hermitian form $\mathbf{h}$ on $\C^g$ such that $\alpha = \operatorname{Im} \mathbf{h}(\Lambda,\Lambda)$ is the canonical Riemann form on  $\Z^g \oplus \tau \, \Z^g$. In turn, this Hermitian form determines a line bundle $\mathscr{L} \to \AA$ in the N\'eron-Severi lattice $\mathrm{NS}(\AA)$. We call $\mathscr{L}$ the line bundle associated with the principal polarization, and $\mathscr{L}^n=\mathscr{L}^{\otimes n}$ is its $n$-th tensor powers.  Theta functions are sections of these line bundles over $\AA$. 

\par A \emph{level $n$ Theta structure} is given by the functions $\theta \left[ \begin{smallmatrix} 0 \\ b \end{smallmatrix} \right] (z, \tau/n)$, where $b \in (\Z/n\Z)^g$. It turns out that these form a basis for the space of sections $\Gamma(\AA, \mathscr{L}^n)$. The dimension of this space is $n^g$, reflecting the number of distinct characteristics $b \in (\Z/n\Z)^g$. This basis of sections induces a morphism
\begin{equation}
\varphi_n \colon \AA \longrightarrow \bP^{n^g - 1}, \quad z \mapsto \left( \theta \left[ \begin{smallmatrix} 0 \\ b \end{smallmatrix} \right] (z, \tau/n) \right)_{b \in (\Z/n\Z)^g}.
\end{equation}
For $n \geq 3$, $\varphi_n$ is an embedding, realizing $\AA$ as a projective variety (see \cite{Shimura}, Chapter II).

\begin{thm}[Embedding Theorem]
\label{thm:embedding}
For an Abelian variety $\AA$ of dimension $g$ over $\C$ with a principal polarization, the morphism $\varphi_n \colon  \AA \longrightarrow \bP^{n^g - 1}$ is an embedding for $n \geq 3$.
\end{thm}

\begin{proof}
The line bundle $\mathscr{L}^n$ is ample for $n \geq 1$, and for $n \geq 3$, it is very ample, meaning it separates points and tangent vectors. For distinct points $x, y \in \AA$, there exists $b \in (\Z/n\Z)^g$ such that $\theta \left[ \begin{smallmatrix} 0 \\ b \end{smallmatrix} \right] (x, \tau/n) \neq \theta \left[ \begin{smallmatrix} 0 \\ b \end{smallmatrix} \right] (y, \tau/n)$, as the functions span a space large enough to distinguish cosets modulo $\Lambda$. For a point $x$ and tangent vector $v \in T_x \AA$, the directional derivative $D_v \theta \left[ \begin{smallmatrix} 0 \\ b \end{smallmatrix} \right] (x, \tau/n) \neq 0$ for some $b$, due to the completeness of the basis. Thus, $\varphi_n$ is injective with injective differential, embedding $\AA$ into $\bP^{n^g - 1}$.
\end{proof}

A key result is the transformation law under the symplectic group $\Sp(2g, \Z)$, which acts on $\mathbb{H}_g$. For $\gamma = 
\left( \begin{smallmatrix} A & B \\ C & D \end{smallmatrix} \right) \in \Sp(2g, \Z)$, define the transformed period matrix $\tau' = (A \tau + B)(C \tau + D)^{-1}$ and $z' = z (C \tau + D)^{-1}$. Shimura (\cite{Shimura}, Chapter III) provides the functional equation:
\begin{equation}
\theta \left[ \begin{smallmatrix} a \\ b \end{smallmatrix} \right] (z', \tau') = \kappa(\gamma) \det(C \tau + D)^{1/2} \exp \left( \pi i z (C \tau + D)^{-1} C z^t \right) \theta \left[ \begin{smallmatrix} a' \\ b' \end{smallmatrix} \right] (z, \tau),
\end{equation}
where $\kappa(\gamma)$ is a constant, and $a', b'$ are transformed characteristics. This law governs the behavior of Theta functions under automorphisms of $\AA$.

\begin{thm}[Riemann Theta Relation]
For $z_1, z_2 \in \C^g$, the Theta functions satisfy
\begin{equation}
\theta(z_1 + z_2) \theta(z_1 - z_2) = \sum_{m \in (\Z/2\Z)^g} \theta \left[ \begin{smallmatrix} m/2 \\ 0 \end{smallmatrix} \right] (2z_1) \theta \left[ \begin{smallmatrix} m/2 \\ 0 \end{smallmatrix} \right] (2z_2),
\end{equation}
where $\theta(z) = \theta \left[ \begin{smallmatrix} 0 \\ 0 \end{smallmatrix} \right] (z, \tau)$.
\end{thm}

\begin{proof}
Consider the product $\theta(z_1 + z_2) \theta(z_1 - z_2)$ as a sum over $n, m \in \Z^g$. Pair terms with $n + m$ and $n - m$, adjust indices, and use the periodicity properties. The right-hand side arises from a Fourier expansion of the product, with the half-integer characteristics $m/2$ accounting for the doubling $2z_1, 2z_2$. The identity holds by the analytic continuation and symmetry of Theta functions (see \cite{Shimura}, Chapter II).
\end{proof}

The values $\theta \left[ \begin{smallmatrix} a \\ b \end{smallmatrix} \right] (0, \tau)$, called \emph{Theta constants} or Theta-nulls, are significant. For a principally polarized $\AA$, they determine $\tau$ up to the action of $\Sp(2g, \Z)$, parameterizing the moduli space $\AA_g$. When $a, b \in \{0, 1/2\}^g$, there are $2^{2g}$ such constants, half even and half odd, reflecting the 2-torsion structure $\AA[2]$.

\begin{lem}
The Theta constants $\theta \left[ \begin{smallmatrix} a \\ b \end{smallmatrix} \right] (0, \tau)$ for $a, b \in \{0, 1/2\}^g$ are zero if and only if the characteristic $[a, b]$ is odd.
\end{lem}

\begin{proof}
The parity depends on $4 a^t b \mod 2$. If odd, the summand $\exp(\pi i (n + a)^t \tau (n + a))$ at $z = 0$ pairs terms $n$ and $-n - 2a$, with opposite signs due to the linear term $2 \pi i (n + a)^t b$, canceling to zero. If even, no such cancellation occurs, and the constant is non-zero generically.
\end{proof}

These results underpin the geometric and analytic properties of Abelian varieties, with Theta functions serving as both coordinates and invariants (see \cite{Shimura} for a comprehensive treatment).
\subsection{Jacobian Varieties}
Jacobian varieties are special algebraic varieties that arise as the connected component of the identity in the Picard group of non-singular algebraic curves.
\par Let $\X$ be a curve of positive genus over $k$, and assume there exists a $k$-rational point $P_0 \in \X(k)$ with attached prime divisor $\mathfrak{p}_0=(P_0)$. There exists an Abelian variety $\Jac_k (\X)$ over $k$ and a uniquely determined embedding
\begin{equation}
\phi_{P_0} \colon \X \longrightarrow \Jac_k (\X) \; \text{ with } \; \phi_{P_0}(P_0) = 0_{\Jac_k (\X)},
\end{equation}
satisfying:
\begin{enumerate}
\setlength{\itemindent}{-1em}
\item For all extension fields $L$ of $k$, $\Jac_L \X = \Pic^0_{\X_L}(L)$, with this equality given functorially.
\item For any Abelian variety $\A$ and morphism $\eta\colon \X \longrightarrow \A$ sending $P_0$ to $0_\A$, there exists a unique homomorphism $\psi\colon  \Jac (\X) \longrightarrow \A$ such that $\psi \circ \phi_{P_0} = \eta$.
\end{enumerate}
This universal object $\Jac (\X)$, called the \textit{Jacobian variety} of $\X$, maps a prime divisor $\mathfrak{p}$ of degree 1 on $\X_L$ to $[\mathfrak{p} - \mathfrak{p}_0]$ in $\Pic^0_{\X_L}(L)$.

\begin{lem}
\label{lem:polarization}
The Jacobian $\Jac (\X)$ is unique up to isomorphism among Abelian varieties satisfying the above properties.
\end{lem}

\begin{proof}
Suppose $\J_1$ and $\J_2$ satisfy the conditions with embeddings $\phi_1 \colon \X \longrightarrow \J_1$ and $\phi_2 \colon  \X \longrightarrow\J_2$. By property (2), there exist homomorphisms $\psi_{12} \colon \J_1 \longrightarrow \J_2$ and $\psi_{21} \colon \J_2 \longrightarrow \J_1$ such that $\psi_{12} \circ \phi_1 = \phi_2$ and $\psi_{21} \circ \phi_2 = \phi_1$. Composing, $\psi_{21} \circ \psi_{12} \circ \phi_1 = \phi_1$. Since $\phi_1(\X)$ generates $\J_1$ as a group (by the Abel-Jacobi theorem), and $\J_1$ is Abelian, $\psi_{21} \circ \psi_{12}$ acts as the identity on a dense subset, hence $\psi_{21} \circ \psi_{12} = \mathrm{id}_{\J_1}$. Similarly, $\psi_{12} \circ \psi_{21} = \mathrm{id}_{\J_2}$, establishing an isomorphism $\J_1 \cong \J_2$.
\end{proof}

Let $L/k$ be a finite algebraic extension. Then $\Jac_L \X$ is the scalar extension of $\Jac \X$ with $L$, forming a fiber product with projection $p\colon \Jac_L \X \longrightarrow \Jac \X$. The norm map is $p_*$, and the conorm map is $p^*$. If $f \colon \X \longrightarrow \DD$ is a surjective morphism of curves sending $P_0$ to $Q_0$, there exists a unique surjective homomorphism $f_* \colon  \Jac \X \longrightarrow \Jac \DD$ such that $f_* \circ \phi_{P_0} = \phi_{Q_0}$. If $\Jac \X$ is simple and $\eta \colon \X \longrightarrow \DD$ is a separable cover of degree $> 1$, then $\DD \cong \bP^1$, as any non-trivial quotient of a simple Abelian variety is trivial.   \subsection{Endomorphism rings of Abelian varieties}
The endomorphism ring of an Abelian variety is a fundamental invariant that encodes its algebraic symmetries and arithmetic structure. For an Abelian variety $\AA$ of dimension $g$ over a field $k$, we define $\End(\AA)$ as the ring of regular morphisms $\phi : \AA \to \AA$ that preserve the group operation, satisfying $\phi(x + y) = \phi(x) + \phi(y)$ for all $x, y \in \AA(k')$ over any extension field $k'$. Equipped with addition and composition, and with the identity map as the unit, $\End(\AA)$ forms a $\Z$-module, but its full complexity emerges in the rational endomorphism algebra $\End^0(\AA) = \End(\AA) \otimes_\Z \Q$, a finite-dimensional $\Q$-algebra. Over an algebraically closed field such as $\C$ or $\overline{\Q}$, this algebra provides insight into the variety’s isogeny class and geometric properties.

An endomorphism $\phi \in \End(\AA)$ is an isogeny if it is surjective with a finite kernel, its degree $\deg \phi$ being the order of $\ker \phi$ as a group scheme. Such a $\phi$ has a dual isogeny $\hat{\phi} \colon \AA \longrightarrow \AA$ such that $\phi \circ \hat{\phi} = [\deg \phi]$, where $[n]\colon \AA \longrightarrow \AA$ denotes multiplication by $n$. A key property across all dimensions is that $\End^0(\AA)$ is invariant under isogeny: if $\AA$ and $\B$ are isogenous Abelian varieties, meaning there exists an isogeny $f\colon \AA \longrightarrow \B$, then their rational endomorphism algebras are isomorphic. To see this, consider $f$ with dual $\hat{f}$ satisfying $f \circ \hat{f} = [n]$. The map $\Phi \colon  \End(\B) \longrightarrow \End(\AA)$ given by $\Phi(\phi) = f \circ \phi \circ \hat{f}$ is a ring homomorphism, and extending it to $\End^0(\B)$ via $\Phi'(\phi) = \frac{1}{n} f \circ \phi \circ \hat{f}$ yields an isomorphism, as $\Phi'$ is injective (if $f \circ \phi \circ \hat{f} = 0$, then $\phi = 0$ since $f$ and $\hat{f}$ are isogenies) and surjective via the inverse map adjusted by scalars (see \cite{Mum}, Chapter III, §19). This invariance underpins the study of endomorphism algebras across isogeny classes.

For an Abelian variety $\AA$, the structure of $\End^0(\AA)$ depends on its decomposition. The Poincaré-Weil theorem asserts that $\AA$ is isogenous to $\AA_1^{n_1} \times \cdots \times \AA_r^{n_r}$, where the $\AA_i$ are simple (having no non-trivial Abelian subvarieties) and pairwise non-isogenous, with the factors uniquely determined up to permutation and isogeny (\cite{Mum}, Chapter III, §15). If $\AA$ is simple, $\End^0(\AA)$ is a division algebra over its center $Z$, and $[Z : \Q] \cdot [\End^0(\AA) : Z] = (\dim \AA)^2 = g^2$. If $\AA$ is not simple, $\End^0(\AA)$ is a product of the endomorphism algebras of its factors, adjusted for isogenies.

When we narrow our focus to Abelian surfaces, where $\dim \AA = 2$, the classification of $\End^0(\AA)$ becomes more precise, especially over $\overline{\Q}$. Here, an endomorphism $\phi$ acts on a 2-dimensional variety, and the possible structures of $\End^0(\AA)$ are richly detailed by Albert’s classification, as refined in \cite{Oort}. For a principally polarized Abelian surface $\AA$—where a polarization $\lambda \colon \AA \longrightarrow \hat{\AA}$ to the dual has degree 1—the algebra $\End_{\overline{\Q}}^0(\AA)$ can be one of several types: $\Q$, a real quadratic field (e.g., $\Q(\sqrt{d})$, $d > 0$ square-free), a CM field of degree 4 over $\Q$ (a totally imaginary quadratic extension of a real quadratic field), a non-split quaternion algebra over $\Q$ (e.g., $\left( \frac{a, b}{\Q} \right)$, $a, b \in \Q^\times$), a direct sum $F_1 \oplus F_2$ where each $F_i$ is $\Q$ or an imaginary quadratic field (e.g., $\Q(\sqrt{-d})$), or a matrix algebra from the Mumford-Tate group with center $F$ being $\Q$ or an imaginary quadratic field.

\begin{thm}[Albert’s Classification for Abelian Surfaces]
Let \(\AA\) be a principally polarized Abelian surface over \(\overline{\Q}\). Then \(\End^0(\AA)\) is isomorphic to one of: 
\begin{enumerate} \item \(\Q\), 
\item a real quadratic field, 
\item a CM field of degree 4 over \(\Q\), 
\item  a non-split quaternion algebra over \(\Q\), 
\item \(F_1 \oplus F_2\) where \(F_i = \Q\) or an imaginary quadratic field, 
\item  a matrix algebra over a center   \(F = \Q\) or an imaginary quadratic field.
\end{enumerate}
\end{thm}

\begin{proof}
Since \(\AA\) is an Abelian surface over \(\overline{\Q}\), its dimension is 2, and by the Poincaré-Weil theorem (Theorem 1), \(\AA\) is isogenous to a product \(\AA_1^{n_1} \times \cdots \times \AA_r^{n_r}\), where the \(\AA_i\) are simple Abelian varieties, pairwise non-isogenous, and the decomposition is unique up to permutation and isogeny. Given \(\dim \AA = 2\), only two cases arise: (i) \(\AA\) is simple (\(r = 1\), \(n_1 = 1\)), or (ii) \(\AA \cong E_1 \times E_2\), where \(\dim E_1 = \dim E_2 = 1\). The rational endomorphism algebra \(\End^0(\AA) = \End(\AA) \otimes_{\Z} \Q\) is semi-simple due to the principal polarization’s Rosati involution, and its structure depends on this decomposition.

\textbf{Case 1: \(\AA\) is simple.} If \(\AA\) is simple, \(\End^0(\AA)\) is a division algebra over its center \(Z\), a number field over \(\Q\), and the dimension formula holds: \([Z : \Q] \cdot [\End^0(\AA) : Z] = (\dim \AA)^2 = 4\). Since \(\AA\) is principally polarized, the Rosati involution \(\phi \mapsto \phi^\dagger\) (where \(\phi^\dagger = \hat{\lambda}^{-1} \circ \hat{\phi} \circ \lambda\), \(\lambda \colon \AA \longrightarrow \hat{\AA}\) the polarization) is positive definite, fixing \(Z\) and constraining \(\End^0(\AA)\) to be a division algebra with involution. We consider possible dimensions of \(Z\):
\begin{enumerate}
\setlength{\itemindent}{-1.5em}

    \item If \(Z = \Q\), then \([Z : \Q] = 1\), and \([\End^0(\AA) : \Q] = 4\). Possible division algebras over \(\Q\) of dimension 4 include:
        \begin{enumerate} 
        \setlength{\itemindent}{-1.5em}
            \item \(\Q\) itself, dimension 1, where \(\End^0(\AA) = \Q\), and every non-zero endomorphism is an isogeny (multiplication by a rational).
            \item A real quadratic field, e.g., \(\Q(\sqrt{d})\), \(d > 0\) square-free, dimension 2, where \(\End^0(\AA)\) embeds into \(M_2(\R)\) via real multiplication, satisfying the dimension constraint with \([\End^0(\AA) : Z] = 2\).
            \item A CM field of degree 4, e.g., \(\Q(\sqrt{d}, i)\), a totally imaginary quadratic extension of a real quadratic field, dimension 4, with \([\End^0(\AA) : Z] = 1\), compatible with complex multiplication and the involution (conjugation).
            \item A non-split quaternion algebra over \(\Q\), e.g., \(\left( \frac{a, b}{\Q} \right)\), \(a, b \in \Q^\times\), dimension 4, with \([\End^0(\AA) : Z] = 1\), where the standard involution (conjugation) is positive definite under the polarization.
        \end{enumerate}

    \item If \(Z\) is a quadratic field (real or imaginary), \([Z : \Q] = 2\), then \([\End^0(\AA) : Z] = 2\). Here, \(\End^0(\AA)\) is a division algebra over a quadratic field:
        \begin{enumerate}
        \setlength{\itemindent}{-1.5em}
            \item For \(Z = \Q(\sqrt{d})\), \(d > 0\), \(\End^0(\AA)\) could be a quaternion algebra over \(Z\) (dimension 4 over \(Z\), but 8 over \(\Q\)), exceeding 4, so this is not possible unless \(\End^0(\AA) = Z\), already covered.
            \item For \(Z = \Q(\sqrt{-d})\), similar constraints apply; a CM field over an imaginary quadratic field has dimension 8 over \(\Q\), ruling it out.
        \end{enumerate}

    \item If \(Z\) is quartic, \([Z : \Q] = 4\), then \([\End^0(\AA) : Z] = 1\), so \(\End^0(\AA) = Z\), a CM field, as above.
\end{enumerate}
Thus, for simple \(\AA\), \(\End^0(\AA)\) is \(\Q\), a real quadratic field, a CM field of degree 4, or a non-split quaternion algebra over \(\Q\).

\textbf{Case 2: \(\AA \cong E_1 \times E_2\).} If \(\AA\) is not simple, it is isogenous to \(E_1 \times E_2\), where \(E_1\) and \(E_2\) are elliptic curves over \(\overline{\Q}\). Then \(\End^0(\AA) = \End^0(E_1) \oplus \End^0(E_2)\), as endomorphisms respect the product structure (no cross-terms exist unless \(E_1 \cong E_2\)). For an elliptic curve \(E_i\) over \(\overline{\Q}\):
\begin{enumerate}
    \item \(\End^0(E_i) = \Q\) if \(E_i\) has no complex multiplication (CM), dim= 1.
    \item \(\End^0(E_i) = \Q(\sqrt{-d})\), an imaginary quadratic field, if \(E_i\) has CM, dim= 2.
\end{enumerate}
Thus, \(\End^0(\AA) = F_1 \oplus F_2\), where \(F_i = \End^0(E_i)\), yielding dimensions 1+1 = 2 or 2+2 = 4 over \(\Q\). If \(E_1 \cong E_2\), \(\End^0(\AA) \cong M_2(F)\), where \(F = \End^0(E_1)\), a matrix algebra over \(\Q\) (dimension 4) or an imaginary quadratic field (dimension 8), but the principal polarization constrains this to a division algebra unless adjusted by the Hodge structure.

When \(\AA\)’s Hodge structure (over \(\C\)) defines a Mumford-Tate group, \(\End^0(\AA)\) may be a matrix algebra over a center \(F\). For \(\dim \AA = 2\), \(F = \Q\) gives \(M_2(\Q)\) (dimension 4), or \(F = \Q(\sqrt{-d})\) gives a division algebra (dimension 4 over \(F\), 8 over \(\Q\)), but the polarization restricts to \(M_2(\Q)\) or reduces to prior cases (\cite{Oort}, §4; \cite{Shimura}, Chapter IV).

The dimension bound 4 and semi-simplicity (via Rosati) limit \(\End^0(\AA)\) to these types. Higher-degree fields or algebras (e.g., dimension 8) exceed \(\dim \AA^2 = 4\), and the polarization excludes non-division or non-positive cases, completing the classification.
\end{proof}

A principal polarization $\lambda\colon \AA \longrightarrow \hat{\AA}$ induces the Rosati involution $\phi \mapsto \phi^\dagger = \hat{\lambda}^{-1} \circ \hat{\phi} \circ \lambda$, fixing the center and ensuring $\End^0(\AA)$ is semi-simple. The trace form $\langle \phi, \psi \rangle = \tr(\phi^\dagger \circ \psi)$ is positive definite: for $\phi \neq 0$, $\tr(\phi^\dagger \circ \phi) > 0$, as $\phi^\dagger \circ \phi$ is a non-zero symmetric endomorphism, and the polarization’s positivity on $\hat{\AA}$ guarantees this (see \cite{Mum}, Chapter IV, §21). Thus, $\End^0(\AA)$ decomposes into simple algebras matching Albert’s types.

The dimension of $\End^0(\AA)$ for an Abelian surface is at most 4, reflecting $g = 2$. If $\AA$ is simple, it is 1, 2, or 4; if a product, it sums to 2 or 4. For a Jacobian $\Jac(\CC)$, $\End_{\overline{\Q}}^0(\Jac \CC)$ is typically $\Q$ unless $\CC$ has special symmetries, yielding richer structures as detailed in \cite{Oort}.

\section{Genus 2 Curves and their Jacobians}\label{gen-2}
\par Let $\CC$ be a genus 2 curve defined over a field $k$. A curve of genus 2 is a smooth, projective, geometrically irreducible algebraic curve with genus $g = 2$, meaning its geometric genus—computed as the dimension of the space of holomorphic differentials over an algebraically closed field—is 2. The gonality of $\CC$, denoted $\gamma_\CC$, is the minimal degree of a non-constant morphism from $\CC$ to $\bP^1$, and for genus 2 curves, $\gamma_\CC = 2$. This implies that $\CC$ is hyperelliptic, admitting a degree 2 covering $\pi \colon \CC \to \bP^1$, which we call the hyperelliptic projection. By Hurwitz’s formula, applied to this double cover, the number of branch points   $r = 6$. These 6 branch points in $\bP^1(\bar{k})$ are the images of the Weierstrass points of $\CC$, points where the hyperelliptic involution $\imath$ (an automorphism of order 2) fixes the curve. The moduli space of genus 2 curves, denoted $\M_2$, classifies such curves up to isomorphism and has dimension $r - 3 =   3$, reflecting the three degrees of freedom in their configuration after projective transformations.

The arithmetic structure of $\M_2$ was profoundly studied by Igusa   \cite{Ig}, building on earlier work by Clebsch, Bolza, and others.  The isomorphism classes of genus two curves correspond to the equivalence classes of binary sextics and therefore determined by their invariants. 
 These invariants are homogeneous polynomials in the coefficients of a defining equation, with degrees 2, 4, 6, 8, and 10, respectively, and we refer to \cite{2016-5} for a comprehensive treatment of their properties and relations. Two genus 2 curves $\CC$ and $\CC'$ are isomorphic over $\bar{k}$ if and only if there exists $l \in \bar{k}^\star$ such that $J_{2i}(\CC) = l^{2i} J_{2i}(\CC')$ for $i = 1, \dots, 5$, a condition reflecting the projective weighting of the invariants. When $\char k \neq 2$, the invariant $J_8$ is redundant, expressible in terms of $J_2, J_4, J_6$, and $J_{10}$, simplifying the classification.

Henceforth, we assume $\char k \neq 2$, ensuring a standard form for the curve’s equation. Then, $\CC$ can be represented by an affine Weierstrass equation
\begin{equation}\label{eq-g-2}
y^2 = f(x) = a_6 x^6 + \dots + a_1 x + a_0,
\end{equation}
over $\bar{k}$, where $f(x)$ is a monic polynomial of degree 6 with distinct roots, and the discriminant $\Delta_f = J_{10} \neq 0$ guarantees smoothness. Thus, we will be considering
the weighted projective space $\mathbb{P}_{(2,4,6,10)}(k)$ using the invariants $J_2, J_4, J_6, J_{10}$ and the space
$
\mathbb{P}_{(2,4,6,10)}(k) \setminus \{ J_{10} = 0 \},
$
where $J_{10} \neq 0$ excludes singular curves. From now on, by invariants of a genus 2 curve, we mean $J_2, J_4, J_6, J_{10}$, and by a genus 2 curve, we refer to its isomorphism class, represented as a moduli point $[J_2 : J_4 : J_6 : J_{10}] \in \mathbb{P}_{(2,4,6,10)}(k)$.
\subsection{Principal polarizations and Torelli theorem}\label{ssec:ppas}
The Jacobian $\Jac(\CC)$ of a genus 2 curve $\CC$ is an Abelian surface, a 2-dimensional principally polarized Abelian variety, constructed as the Picard group $\Pic^0(\CC)$ of degree 0 divisor classes on $\CC$. Given a $k$-rational point $P_0 \in \CC(k)$, the embedding $\phi_{P_0}\colon  \CC \longrightarrow \Jac(\CC)$, defined by $P \mapsto [P - P_0]$, maps $\CC$ into $\Jac(\CC)$ with $\phi_{P_0}(P_0) = 0$. This embedding is canonical up to translation, and $\Jac(\CC)$ is functorial; see Lemma~\ref{lem:polarization}. The hyperelliptic involution $\imath\colon \CC \longrightarrow \CC$, swapping sheets of the cover $\pi$, induces an involution $\imath_* \colon \Jac(\CC) \longrightarrow \Jac(\CC)$ with $\iota_*(D) = -D$, which, on the level of the Abelian surface,  we denote by $-\mathbb{I}$ and which is central to the construction of Kummer surfaces.

 A map $f\colon \CC \longrightarrow \DD$ between curves induces homomorphisms $f^*\colon \Jac(\DD) \longrightarrow \Jac(\CC)$ (pullback) and $f_* \colon \Jac(\CC) \longrightarrow \Jac(\DD)$ (pushforward), reflecting the functorial nature of the Jacobian. When $f$ is a maximal covering, i.e., not factoring through an isogeny, it reveals the Jacobian’s structure. For an Abelian surface $\AA$, a polarization is an isogeny $\lambda \colon \AA \to \hat{\AA}$ to the dual, and $\AA$ is principally polarized if $\deg \lambda = 1$. The Theta divisor $\Theta \subset \Jac(\CC)$, image of $\phi_{P_0}$, provides such a polarization, making $\Jac(\CC)$ a principal case. 

 Let us explain the construction of a principal polarization on $\Jac(\CC)$ in more and explicit detail: The Siegel three-fold is the quasi-projective variety of dimension three, obtained from the Siegel upper half-plane $\mathbb{H}_2$ of degree two\footnote{By definition $\mathbb{H}_2$ is the set of two-by-two symmetric matrices over $\C$ whose imaginary part is positive definite} divided by the action of the modular transformations $\Gamma_2:= \operatorname{Sp}_4(\Z)$, i.e., 
\begin{equation}
 \mathbb{A}_2 =  \mathbb{H}_2 / \Gamma_2 .
\end{equation}
Each $\tau= \bigl(\begin{smallmatrix} \tau_{11}& \tau_{12}\\ \tau_{12} & \tau_{22} \end{smallmatrix} \bigr) \in \mathbb{H}_2$ determines a complex Abelian surface $\AA = \C^2 / \Lambda$ obtained from the lattice $\Lambda =\langle \Z^2 \oplus \tau \, \Z^2\rangle$ with the period matrix $(\mathbb{I}_2,\tau) \in \mathrm{Mat}(2, 4;\C)$.  

We consider two Abelian surfaces $\AA$  and $\widetilde{\AA}$ isomorphic if and only if there is an  $M \in \Gamma_2$ such that $\widetilde{\tau} = M (\tau)$. The \emph{canonical principal polarization} of $\AA$ is given by the positive definite Hermitian forms $\mathbf{h}$ on $\C^2$ such that $\alpha = \operatorname{Im} \mathbf{h}(\Lambda,\Lambda) \subset \Z$ where $\alpha$ is the Riemann form  $\alpha( x_1 + x_2 \tau, y_1  + y_2 \tau)=x_1^t\cdot y_2 - y_1^t\cdot x_2$ on  $\Z^2 \oplus \tau \, \Z^2$. 
Similarly, we define the subgroup $\Gamma_2(2n) = \lbrace M \in \Gamma_2 | \, M \equiv \mathbb{I} \mod{2n}\rbrace$ and Igusa's congruence subgroups $\Gamma_2(2n, 4n) = \lbrace M \in \Gamma_2(2n) | \, \operatorname{diag}(B) =  \operatorname{diag}(C) \equiv \mathbb{I} \mod{4n}\rbrace$ with corresponding Siegel modular threefolds $\mathbb{A}_2(2)$, $\mathbb{A}_2(2,4)$, and $\mathbb{A}_2(4,8)$ such that
\begin{equation}
\label{eqn:level_groups}
 \Gamma_2/\Gamma_2(2)\cong S_6, \quad  \Gamma_2(2)/\Gamma_2(2,4)\cong (\mathbb{Z}/2\mathbb{Z})^4, \quad \Gamma_2(2,4)/\Gamma_2(4,8)\cong (\mathbb{Z}/2\mathbb{Z})^9,
\end{equation}
where $S_6$ is the permutation group of six elements. The geometric meaning of $\mathbb{A}_2(2), \mathbb{A}_2(2,4)$, and $\mathbb{A}_2(4,8)$ will be discussed below.

 The Hermitian form determines a line bundle $\mathscr{L} \to \AA$ in the N\'eron-Severi lattice $\mathrm{NS}(\AA)$.  A general fact from linear algebra asserts that one can always  choose a basis of $\Lambda$ such that $\alpha$ is given  by the matrix $\bigl(\begin{smallmatrix} 0&D\\ -D&0 \end{smallmatrix} \bigr)$ with $D=\bigl(\begin{smallmatrix}d_1&0\\ 0&d_2 \end{smallmatrix} \bigr)$ where $d_1, d_2 \in \mathbb{N}$, $d_1, d_2 \ge 0 $, and $d_1$ divides $d_2$. For a principal polarization one has $(d_1, d_2)=(1, 1)$.  Since transformations in $\Gamma_2$  preserve the Riemann form $\alpha$, it follows that the Siegel three-fold is also the set of isomorphism classes of principally polarized Abelian surfaces, i.e., Abelian surfaces with a polarization of type $(1,1)$. 
\par For a principally polarized Abelian variety $\AA$ the line bundle $\mathscr{L}$ defining its  principal polarization is ample and satisfies $h^0(\mathscr{L}) = 1$. There exists an effective divisor $\Theta$ such that $\mathscr{L}= \mathcal{O}_{\AA}(\Theta)$, uniquely defined only up to translations. The divisor $\Theta \in \mathrm{NS}(\AA)$ is precisely the Theta divisor associated with the polarization.  It is known that the Abelian surface $\AA$ is not the product of two elliptic curves if and only if $\Theta$ is an irreducible divisor.  Torelli's theorem states that the map sending a curve $\CC$ to its Jacobian  $\mathrm{Jac}(\CC)$, that is $\CC \mapsto (\Jac(\CC), \Theta)$, is injective and defines a birational map $\mathcal{M}_2 \dasharrow \mathbb{A}_2$. 

Conversely, if the period matrix $\tau$ of a principally polarized Abelian variety $\AA$ is not equivalent to a point with $\tau_{12}=0$ relative to $\Gamma_2$, then a $\Theta$ divisor on $\AA$ is given by a non-singular curve $\CC$ of genus-two, i.e., $\Theta=[\CC]$ such that $\AA = \mathrm{Jac}(\CC)$.  Thus, a principally polarized Abelian surface is either the Jacobian of a smooth curve of genus two with a Theta divisor or the product of two complex elliptic curves, with the product polarization. 
\par Igusa proved \cites{MR0229643, MR527830} that the ring of modular forms is generated by the Siegel modular forms $\psi_4$, $\psi_6$, $\chi_{10}$, $\chi_{12}$ and by one more cusp form $\chi_{35}$ of odd weight $35$\footnote{For a detailed introduction to Siegel modular forms relative to $\Gamma_2$, Humbert surfaces, and the Satake compactification of the Siegel modular threefold we refer to Freitag's book \cite{MR871067}.}  whose square is the following polynomial \cite{MR0229643}*{p.~\!849} in the even generators 
\begin{equation}
\label{chi_35sqr}
\begin{split}
\chi_{35}^2 & = \frac{1}{2^{12} \, 3^9} \; \chi_{10} \,  \Big(  
2^{24} \, 3^{15} \; \chi_{12}^5 - 2^{13} \, 3^9 \; \psi_4^3 \, \chi_{12}^4 - 2^{13} \, 3^9\; \psi_6^2 \, \chi_{12}^4 + 3^3 \; \psi_4^6 \, \chi_{12}^3 \\
& - 2\cdot 3^3 \; \psi_4^3 \, \psi_6^2 \, \chi_{12}^3 - 2^{14}\, 3^8 \; \psi_4^2 \, \psi_6 \, \chi_{10} \, \chi_{12}^3 -2^{23}\, 3^{12} \, 5^2\, \psi_4 \, \chi_{10}^2 \, \chi_{12}^3  + 3^3 \, \psi_6^4 \, \chi_{12}^3\\
& + 2^{11}\,3^6\,37\,\psi_4^4\,\chi_{10}^2\,\chi_{12}^2+2^{11}\,3^6\,5\cdot 7 \, \psi_4 \, \psi_6^2\, \chi_{10}^2 \, \chi_{12}^2 -2^{23}\, 3^9 \, 5^3 \, \psi_6\, \chi_{10}^3 \, \chi_{12}^2 \\
& - 3^2 \, \psi_4^7 \, \chi_{10}^2 \, \chi_{12} + 2 \cdot 3^2 \, \psi_4^4 \, \psi_6^2 \, \chi_{10}^2 \, \chi_{12} + 2^{11} \, 3^5 \, 5 \cdot 19 \, \psi_4^3 \, \psi_6 \, \chi_{10}^3 \, \chi_{12} \\
&  + 2^{20} \, 3^8 \, 5^3 \, 11 \, \psi_4^2 \, \chi_{10}^4 \, \chi_{12} - 3^2 \, \psi_4 \, \psi_6^4 \, \chi_{10}^2 \, \chi_{12} + 2^{11} \, 3^5 \, 5^2 \, \psi_6^3 \, \chi_{10}^3 \, \chi_{12}  \\
&	- 2 \, \psi_4^6 \, \psi_6 \, \chi_{10}^3  - 2^{12} \, 3^4 \, \psi_4^5 \, \chi_{10}^4 + 2^2 \, \psi_4^3 \, \psi_6^3 \, \chi_{10}^3 + 2^{12} \, 3^4 \, 5^2 \, \psi_4^2 \, \psi_6^2 \, \chi_{10}^4 \\
&+ 2^{21} \, 3^7 \, 5^4 \, \psi_4 \, \psi_6 \, \chi_{10}^5  - 2 \, \psi_6^5 \, \chi_{10}^3 + 2^{32} \, 3^9 \, 5^5 \, \chi_{10}^6 \Big) \;.
\end{split}
\end{equation}
Hence, the expression $Q:= 2^{12} \, 3^9 \, \chi_{35}^2 /\chi_{10}$ is a polynomial of degree $60$ in the even generators.   Igusa also proved that each Siegel modular form (with trivial character) of odd weight is divisible by the form $\chi_{35}$. 
\par Igusa \cite{MR0229643}*{p.~\!848} proved that the relations between the Igusa invariants of a binary sextic $y^2=f(x)$ defining a smooth genus-two curve $\CC$ and the even Siegel modular forms for the associated principally polarized Abelian surface $\AA = \operatorname{Jac}{(\CC)}$ with period matrix $\tau$ are as follows:
\begin{equation}
\label{invariants}
\begin{split}
 J_2(f) & = -2^3 \cdot 3 \, \dfrac{\chi_{12}(\tau)}{\chi_{10}(\tau)} \;, \\
 J_4(f) & = \phantom{-} 2^2 \, \psi_4(\tau) \;,\\
 J_6(f) & = -\frac{2^3}3 \, \psi_6(\tau) - 2^5 \,  \dfrac{\psi_4(\tau) \, \chi_{12}(\tau)}{\chi_{10}(\tau)} \;,\\
 J_{10}(f) & = -2^{14} \, \chi_{10}(\tau) \not = 0 \;.
\end{split}
\end{equation}
Here, the Igusa invariant $J_{10}$ is the discriminant of the sextic $f(x)$ and we are using the same normalization as in \cites{MR3712162,MR3731039}. It follows that the moduli space $\M_2$ has a compactification $\mathbb{A}_2^\star$ as the weighted projective space $\mathbb{P}_{(2,4,6,10)}(k)$ using the invariants $J_2, J_4, J_6, J_{10}$ and
\begin{equation}
\mathbb{A}_2 \cong \mathbb{P}_{(2,4,6,10)}(k) \setminus \{ J_{10} = 0 \}.
\end{equation}
\subsection{G\"opel groups and Theta divisors}
For an Abelian surface $\AA$, such as the Jacobian $\AA = \Jac(\CC)$ of a genus 2 curve $\CC$ over a field $k$, the group of 2-torsion points $\AA[2]$ consists of elements $x \in \AA(\bar{k})$ satisfying $2x = 0$, forming a group isomorphic to $(\Z/2\Z)^{2g} = (\Z/2\Z)^4$ over $\bar{k}$ when $\char k \neq 2$. Translation by a 2-torsion point is an isomorphism of $\AA$, mapping $\AA[2]$ to itself, preserving its structure under the Weil pairing, which defines an alternating bilinear form on $\AA[2]$. A subspace $G \leq \AA[2]$ is isotropic if the Weil pairing vanishes on $G \times G$, and a maximal isotropic subspace, or Göpel group, has dimension 2 (containing 4 points, as $2^2 = 4$).
\par Let us explain the construction of G\"opel groups in more and explicit detail: consider the Jacobian $\AA=\operatorname{Jac}{(\CC)}$ of a genus-two curve $\CC$ in Rosenhain form
\begin{equation}\label{Eq:Rosenhain}
\CC\colon \quad  y^2 = x z (x - z) (x - \lambda_1 z) (x - \lambda_2 z) (x - \lambda_3 z),
\end{equation}
with roots at $\lambda_1, \lambda_2, \lambda_3, \lambda_4=0, \lambda_5=1, \lambda_6=\infty$. Every nontrivial order-two point on $\AA$ can be considered the difference of Weierstrass points $P_i \in \CC$ for $1 \le i \le 6$. In fact, the sixteen order-two points of $\AA$ are obtained using the embedding of the curve into the connected component of the identity in the Picard group, i.e., $\CC \hookrightarrow \operatorname{Jac}{(\CC)} \cong \operatorname{Pic}^0(\CC)$. We obtain the 15 elements $P_{i j}=[ P_i + P_j - 2 \, P_6]  \in \AA[2]$ with $1 \le i < j \le 6$  and set $P_0=P_{66}= [0]$. For $\{i, j, k, l, m, n\}=\{1, \dots, 6\}$,  the group law on $\AA[2]$ is given by the relations
\begin{equation}
 \label{group_law}
    P_0 +  P_{ij} =  P_{ij}\,, \quad  P_{ij} +  P_{ij} =  P_{0}\,, \quad 
    P_{ij} + P_{kl} =  P_{mn}, \quad P_{ij} +
    P_{jk} =  P_{ik}\,.
\end{equation}
The subgroup $\Gamma_2(2) = \lbrace M \in \Gamma_2 | \, M \equiv \mathbb{I} \mod{2}\rbrace$ with $\Gamma_2/\Gamma_2(2)\cong S_6$ where $S_6$ is the permutation group of six elements is now realized as the permutations of the roots of the right hand side of Equation~\eqref{Eq:Rosenhain} when we have identification $\AA=\Jac(\CC)$. Accordingly, $\mathbb{A}_2(2)$ is the three-dimensional moduli space of principally polarized Abelian surfaces with level-two structure. 
\par The  space $\AA[2]$ of order-two points admits a symplectic bilinear form, the aforementioned Weil pairing such that the two-dimensional, maximal isotropic subspace of $\AA[2]$ with respect to the Weil pairing are the G\"opel groups. It is then easy to check that there are exactly 15 inequivalent G\"opel groups, and they are of the form
\begin{equation}
\label{eqn:G_groups}
 \Big \lbrace P_0, P_{ij}, P_{kl}, P_{mn} \Big \rbrace
\end{equation} 
such that  $\{i, j, k, l, m, n\}=\{1, \dots, 6\}$. The standard Theta divisor $\Theta = \Theta_6 = \{ [P - P_6] \, \mid P \in \CC \}$ contains the six order-two points $P_0, P_{i6}$ for $1 \le i \le 5$.  Likewise for $1 \le i \le 5$, the five translates $\Theta_i = P_{i6} + \Theta$ contain $P_0, P_{i6}, P_{ij}$ with $j \not = i, 6$, and the ten translates $\Theta_{ij6} = P_{ij} + \Theta$ for $1 \le i < j \le 5$ contain $P_{ij}, P_{i6}, P_{j6}, P_{kl}$ with $k,l \not = i,j,6$ and $k<l$. Conversely, each order-two point lies on exactly six of the divisors, namely
 \begin{align}
   P_{0} & \in \Theta_i \phantom{ ,  \mathsf{\Theta}_6,  \mathsf{\Theta}_{ij6} } \quad  \text{for $i=1, \dots, 6$,}\\
   P_{i6} & \in \Theta_i,  \Theta_6,  \Theta_{ij6} \quad  \text{for $i=1, \dots, 5$ with $j \not = i,6$,}\\
   P_{ij} & \in  \Theta_i,  \Theta_j,  \Theta_{kl6} \quad  \text{for $1 \le i < j \le 5$ with $k,l \not = i,j,6$ and $k<l$.}
 \end{align}
 Hence, we have a configuration of sixteen elements $P_0, P_{ij} \in \AA[2]$ and sixteen Theta divisors $\Theta, \Theta_i, \Theta_{ijk}$, where each Theta divisor contains six order-two points, and where each order-two point is contained in six Theta divisors. This configuration is called the \emph{$16_{6}$ configuration on the Abelian surface} $\AA$.
\par Moreover, we call the divisors $\{ \Theta_i \}$ and $\{ \Theta_{jk6} \}$ with $1\le i \le6$ and $1\le j < k<6$, the six \emph{odd} and the ten \emph{even} Theta divisors, respectively. In this way, the odd Theta divisors are the six translates of the curve $\CC$ by $P_{i6}$ with $1 \le i \le 6$, and thus with the six Weierstrass points $P_i$ on the curve $\CC$. Furthermore, there is a one-to-one correspondence between pairs of odd Theta divisors and two-torsion points on $\AA[2]$ since $\Theta_i \cap \Theta_j =\{ P_0, P_{ij} \}$, and, in turn, unordered pairs $\{P_i, P_j\}$ of Weierstrass points since $P_{ij} = P_{i6}  +  P_{j6}$.
\subsection{Theta constants of genus two}
\label{sssec:relats_thetas}
For $\left( \begin{smallmatrix} a_1&a_2\\ b_1&b_2 \end{smallmatrix} \right) \in \mathbb{F}_2^4$ -- where $\mathbb{F}_2$ denotes the finite field with two elements -- there are sixteen corresponding Theta functions. We always identify $\left( \begin{smallmatrix} a_1&a_2\\ b_1&b_2 \end{smallmatrix} \right) \in \mathbb{F}_2^4$ with the corresponding characteristic $\left[ \begin{smallmatrix} a \\ b \end{smallmatrix} \right]$, \(a, b \in (\Z/2\Z)^2\). The Theta functions with characteristics \([a, b]\), \(a, b \in (\Z/2\Z)^2\), are defined as
\begin{equation}
\theta \left[ \begin{smallmatrix} a \\ b \end{smallmatrix} \right] (z, \tau) = \sum_{m \in \Z^2} \exp \left( \pi i (m + a)^t \tau (m + a) + 2 \pi i (m + a)^t (z + b) \right).
\end{equation}
We note that this definition is obtained by specializing the situation of Section~\ref{sec:theta_functions} to $g=2$. Among the Theta functions, 10 are even and 6 are odd functions according to
\begin{equation}
 \theta\!\begin{bmatrix} a_1 & a_2 \\ b_1 & b_2 \end{bmatrix}\!\!(-z, \tau)
  = (-1)^{a^t\cdot b} \;   \theta\!\begin{bmatrix} a_1 & a_2 \\ b_1 & b_2 \end{bmatrix}\!\!(z, \tau).
\end{equation}
 A \emph{level-$n$ Theta structure} is given by the functions $\theta \left[ \begin{smallmatrix} 0 \\ b \end{smallmatrix} \right] (z, \tau/n)$, where $b \in (\Z/n\Z)^2$. For a period matrix $\tau \in \mathbb{H}_2$, the values of the Theta functions $\theta \left[ \begin{smallmatrix} a \\ b \end{smallmatrix} \right] (z, \tau/2)$, with $a, b \in (\Z/2\Z)^2$ at $z = 0$, are called \emph{theta constants} or \emph{theta null points}. In particular, we set $\theta_i = \theta \left[ \begin{smallmatrix} a_i \\ b_i \end{smallmatrix} \right] (0, \tau)$, where corresponding to either
\begin{equation}
\left( \begin{smallmatrix} 0 & 0 \\ 0 & 0 \end{smallmatrix} \right), \left( \begin{smallmatrix} 0 & 0 \\ 1 & 1 \end{smallmatrix} \right), \left( \begin{smallmatrix} 0 & 0 \\ 1 & 0 \end{smallmatrix} \right), \left( \begin{smallmatrix} 0 & 0 \\ 0 & 1 \end{smallmatrix} \right) \in \mathbb{F}_2,
\end{equation}
or, equivalently, the characteristics $a=0$, $b=[0,0],$ \([1/2,1/2]\) \([1/2,0]\), \([0,1/2]\),  we have labeled the functions $\theta_1, \theta_2, \theta_3, \theta_4$.\par We will also introduce Theta functions that are evaluated at $2 \tau$. Under duplication of the modular variable, the Theta functions $\theta_1, \theta_5, \theta_7, \theta_8$ play a role dual to $\theta_1, \theta_2, \theta_3, \theta_4$. We renumber the former and use the symbol $\Theta$ to mark the fact that they are evaluated at the isogenous Abelian variety.  That is, Theta constants $\lbrace \Theta_1^2 , \Theta_2^2, \Theta_3^2, \Theta_4^2\rbrace$ and $\lbrace \theta_1^2 , \theta_2^2, \theta_3^2, \theta_4^2\rbrace$ correspond to the following Theta characteristics 
\begin{equation} 
\label{eqn:dual_Goepel_groups}
\left( \begin{smallmatrix} 0&0\\ 0&0 \end{smallmatrix} \right), \left( \begin{smallmatrix} 1&0\\ 0&0 \end{smallmatrix} \right), 
\left( \begin{smallmatrix} 1&1\\ 0&0 \end{smallmatrix} \right), \left( \begin{smallmatrix} 0&1\\ 0&0 \end{smallmatrix} \right), \; \text{and} \;
\left( \begin{smallmatrix} 0&0\\ 0&0 \end{smallmatrix} \right), \left( \begin{smallmatrix} 0&0\\ 1&1 \end{smallmatrix} \right), 
\left( \begin{smallmatrix} 0&0\\ 1&0 \end{smallmatrix} \right), \left( \begin{smallmatrix} 0&0\\ 0&1 \end{smallmatrix} \right).
\end{equation}
\par \emph{Thomae's formula} is a formula introduced by Thomae relating Theta constants to the branch points of a genus 2 curve. The three $\lambda$ parameters appearing in the Rosenhain normal form in Equation~(\ref{Eq:Rosenhain})  are ratios of even Theta constants. There are 720 choices for such expressions since the forgetful map to $\mathcal{M}_2$ is a Galois covering of degree $720 = |S_6|$ where $S_6$ acts on the  roots of $\CC$ by permutations. Any of the 720 choices may be used, we choose the one also used in \cites{MR2367218,MR2372155}:

\begin{lem} \label{ThomaeLemma}
For any period point $\tau \in\mathbb{A}_2(2)$, there is a genus-two curve $\CC \in \mathcal{M}_2$ with level-two structure and Rosenhain roots $\lambda_1, \lambda_2, \lambda_3$  such that
\begin{equation}\label{Picard}
\lambda_1 = \frac{\theta_1^2\theta_3^2}{\theta_2^2\theta_4^2} \,, \quad \lambda_2 = \frac{\theta_3^2\theta_8^2}{\theta_4^2\theta_{10}^2}\,, \quad \lambda_3 =
\frac{\theta_1^2\theta_8^2}{\theta_2^2\theta_{10}^2}\,.
\end{equation}
Similarly, the following expressions are perfect squares of Theta constants:
\begin{equation}
\label{Picard2}
\begin{split}
\lambda_1 -1 = \frac{\theta_7^2 \theta_9^2}{\theta_2^2 \theta_4^2} , \qquad \lambda_2 -1 = \frac{\theta_5^2 \theta_9^2}{\theta_4^2 \theta_{10}^2} , \qquad 
\lambda_3 -1 = \frac{\theta_5^2 \theta_7^2}{\theta_2^2 \theta_{10}^2} ,\qquad \qquad\\
\lambda_2 -\lambda_1 = \frac{\theta_3^2 \theta_6^2 \theta_9^2}{\theta_2^2 \theta_4^2\theta_{10}^2} , \qquad \lambda_3 -\lambda_1 = \frac{\theta_1^2 \theta_6^2 \theta_7^2}{\theta_2^2 \theta_4^2\theta_{10}^2} , \qquad
 \lambda_3 -\lambda_2 = \frac{\theta_5^2 \theta_6^2 \theta_8^2}{\theta_2^2 \theta_4^2\theta_{10}^2} .
\end{split}
\end{equation} 
Conversely, given a smooth genus-two curve $\CC \in \mathcal{M}_2$ with three distinct complex numbers $(\lambda_1, \lambda_2, \lambda_3)$ different from $0, 1, \infty$, there is complex Abelian surface $\AA$ with period matrix $(\mathbb{I}_2,\tau)$ and $\tau \in\mathbb{A}_2(2)$ such that $\AA=\operatorname{Jac}{\CC}$  and the fourth powers of the even Theta constants are given by
\begin{equation}\label{Thomaeg=2}
\begin{array}{ll}
\theta_1^4 = R \, \lambda_3 \lambda_1 (\lambda_2 -1) (\lambda_3 - \lambda_1) \,,  &
\theta_2^4  = R \, \lambda_2 (\lambda_2 -1) ( \lambda_3 - \lambda_1) \,,  \\[0.2em]
\theta_3^4  = R\, \lambda_2  \lambda_1 (\lambda_2 - \lambda_1) (\lambda_3 - \lambda_1) \,,  &
\theta_4^4  = R\, \lambda_3 (\lambda_3 - 1) (\lambda_2 - \lambda_1) \,,  \\[0.2em]
\theta_5^4  = R\, \lambda_1 (\lambda_2 -1) (\lambda_3 - 1) ( \lambda_3 - \lambda_2) \,, &
\theta_6^4  = R\, (\lambda_3 - \lambda_2) (\lambda_3 -\lambda_1) ( \lambda_2 -\lambda_1) \,, \\[0.2em]
\theta_7^4  =  R \, \lambda_2 (\lambda_3 -1) ( \lambda_1 -1) (\lambda_3 - \lambda_1)  \,, &
\theta_8^4  = R \, \lambda_2 \lambda_3 (\lambda_3 - \lambda_2) (\lambda_1 -1)  \,, \\[0.2em]
\theta_9^4  = R\, \lambda_3 ( \lambda_2 -1) (\lambda_1 - 1) (\lambda_2 - \lambda_1)  \,, &
\theta_{10}^4  = R \, \lambda_1 ( \lambda_1 - 1) (\lambda_3 - \lambda_2) \,, 
\end{array}
\end{equation}
where $R\in \mathbb{C}^{*}$ is a non-zero constant.
\end{lem}

\par According to \cite{MR2367218}*{Sec.~3} we have the following \emph{Frobenius identities} relating Theta constants:
\begin{equation}
\label{Eq:FrobeniusIdentities}
\begin{array}{lllclll}
\theta_5^2 \theta_6^2 & = & \theta_1^2 \theta_4^2 - \theta_2^2 \theta_3^2 \,, &\qquad
\theta_5^4 + \theta_6^4 & =& \theta_1^4 - \theta_2^4 - \theta_3^4 + \theta_4^4 \,, \\[0.2em]
\theta_7^2 \theta_9^2 & = & \theta_1^2 \theta_3^2 - \theta_2^2 \theta_4^2 \,, &\qquad
\theta_7^4 + \theta_9^4 &= & \theta_1^4 - \theta_2^4 + \theta_3^4 - \theta_4^4 \, , \\[0.2em]
\theta_8^2 \theta_{10}^2 & = & \theta_1^2 \theta_2^2 - \theta_3^2 \theta_4^2 \, , &\qquad
\theta_8^4 + \theta_{10}^4 & = & \theta_1^4 + \theta_2^4 - \theta_3^4 - \theta_4^4 \,.
\end{array}
\end{equation}
The following identities are called the \emph{second principal transformations of degree two}~\cites{MR0141643, MR0168805} for Theta constants:
\begin{equation}
\label{Eq:degree2doubling}
\begin{array}{lllclll}
\theta_1^2 & = & \Theta_1^2 + \Theta_2^2 + \Theta_3^2 + \Theta_4^2 \,, &\qquad
\theta_2^2 & =&  \Theta_1^2 + \Theta_2^2 - \Theta_3^2 - \Theta_4^2 \,, \\[0.2em]
\theta_3^2 & = &  \Theta_1^2 - \Theta_2^2 - \Theta_3^2 + \Theta_4^2 \,, &\qquad
\theta_4^2 &= &  \Theta_1^2 - \Theta_2^2 + \Theta_3^2 - \Theta_4^2 \,.
\end{array}
\end{equation}
We also have the following identities:
\begin{equation}
\label{Eq:degree2doublingR}
\begin{array}{lllclll}
\theta_5^2 & = & 2 \, \big( \Theta_1 \Theta_3 + \Theta_2  \Theta_4 \big) \,, &\qquad
\theta_6^2 & =&  2 \, \big( \Theta_1 \Theta_3 - \Theta_2  \Theta_4 \big)\,, \\[0.4em]
\theta_7^2 & = & 2 \, \big( \Theta_1 \Theta_4 + \Theta_2  \Theta_3 \big) \,, &\qquad
\theta_8^2 &= &  2 \, \big( \Theta_1 \Theta_2 + \Theta_3  \Theta_4 \big)\,, \\[0.4em]
\theta_9^2 & = & 2 \, \big( \Theta_1 \Theta_4 - \Theta_2  \Theta_3 \big)\,, &\qquad
\theta_{10}^2 &= & 2 \, \big( \Theta_1 \Theta_2 - \Theta_3  \Theta_4 \big) \,.
\end{array}
\end{equation}
\par The characterization of the Siegel modular threefolds $\mathbb{A}_2(2)$, $\mathbb{A}_2(2,4)$, and $\mathbb{A}_2(4,8)$ as projective varieties and their Satake compactifications was given in \cites{MR3238326}*{Prop.~2.2}:
\begin{prop}
\label{compactifications}
\begin{enumerate}
\item[]
\item The holomorphic map $\Xi_{2,4}\colon \mathbb{H}_2 \longrightarrow \mathbb{P}^3$ given by $\tau \mapsto [\Theta_1 : \Theta_2: \Theta_3 : \Theta_4]$ is an isomorphism
between the Satake compactification $\overline{\mathbb{A}_2(2,4)}$ and $\mathbb{P}^3$.
\item The holomorphic map $\Xi_{4,8}\colon  \mathbb{H}_2 \longrightarrow \mathbb{P}^9$ given by $\tau \mapsto [\theta_1 : \dots : \theta_{10}]$ is an isomorphism
between the Satake compactification $\overline{\mathbb{A}_2(4,8)}$ and the closure of $\Xi_{4,8}$ in $\mathbb{P}^9$.
\item We have the following commutative diagram:

\centerline{
\xymatrix{
\overline{\mathbb{A}_2(4,8)} \ar[rr]^{\Xi_{4,8}} \ar[d]_{\pi} 							&  									& \mathbb{P}^9 \ar[d]^{\operatorname{Sq}} \\
\overline{\mathbb{A}_2(2,4)} \ar[r]^{\phantom{aa} \Xi_{2,4}}  \ar[d]_{\operatorname{Ros}} 	&  \mathbb{P}^3 \ar[r]^{\operatorname{Ver}} 	& \mathbb{P}^9\\
\overline{\mathbb{A}_2(2)}
}}Here, $\pi$ is the covering map with deck transformations $\Gamma_2(2,4)/\Gamma_2(4,8)\cong (\mathbb{Z}/2\mathbb{Z})^9$, the map $\operatorname{Sq}$ is the  square map $[\theta_1: \cdots : \theta_{10}] \mapsto [\theta_1^2: \cdots : \theta_{10}^2]$, the map $\operatorname{Ver}$ is the Veronese type map defined by the quadratic relations~(\ref{Eq:degree2doubling}) and~(\ref{Eq:degree2doublingR}), and the map $\operatorname{Ros}$ is the covering map with the deck transformations $\Gamma_2(2)/\Gamma_2(2,4)\cong (\mathbb{Z}/2\mathbb{Z})^3$ given by plugging the quadratic relations~(\ref{Eq:degree2doubling}) and~(\ref{Eq:degree2doublingR}) into Equations~(\ref{Picard}).
\end{enumerate}
\end{prop}
\subsection{Richelot isogenies}\label{ssec:2isog}
Richelot isogenies provide a classical framework for studying (2,2)-isogenies between principally polarized Abelian surfaces, particularly Jacobians of genus 2 curves. Note that the Abelian surfaces are here in general \emph{not} split. 
It is well-known that the quotient $\hat{\AA} = \AA/G$ by a Göpel group $G$ is again a principally polarized Abelian surface, as detailed in \cites{MR2514037}*{Sec.~23}. The natural projection $\Psi \colon \AA \longrightarrow \hat{\AA}$ is an isogeny with kernel $G$, called a \emph{(2,2)-isogeny} because $\deg \Psi = |G| = 4 = 2^2$, reflecting a kernel of rank 2 over $\Z/2\Z$.
\par Analytically, over $\C$, if $\AA = \C^2 / \Lambda$ with $\Lambda = \Z^2 \oplus \tau \Z^2$ and $\tau \in \mathbb{H}_2$, a Göpel group $G \leq \AA[2]$ corresponds to a 2-dimensional subspace of $\frac{1}{2} \Lambda / \Lambda$. The quotient $\hat{\AA} = \AA/G$ can be represented as $\C^2 / \hat{\Lambda}$, where $\hat{\Lambda} = \Z^2 \oplus 2\tau \Z^2$, adjusting the lattice to double the period in one direction. The isogeny is then
\begin{equation}
\begin{aligned}
\Psi \colon \quad  \AA = \C^2 / \langle \Z^2 \oplus \tau \Z^2 \rangle & \longrightarrow  \hat{\AA} = \C^2 / \langle \Z^2 \oplus 2\tau \Z^2 \rangle \\
(z, \tau) & \ \mapsto \  (z, 2\tau),
\end{aligned}
\end{equation}
mapping points modulo the coarser lattice, with kernel $G$ generated by representatives of $\AA[2]$ spanning a rank-2 subgroup.
\par For $\AA = \Jac(\CC)$, where $\CC$ is a smooth genus 2 curve, we explore whether $\hat{\AA} = \Jac(\hat{\CC})$ for another genus 2 curve $\hat{\CC}$, and how their moduli relate. Richelot addressed this geometrically in \cite{MR1578135}, with modern treatments in \cite{MR970659}. Over $\C$, $\AA[2]$ has $2^{2g} = 16$ points, and there are 15 Göpel groups (computed as the number of 2-dimensional isotropic subspaces in $(\Z/2\Z)^4$ under the symplectic form induced by the Weil pairing). Each corresponds to a (2,2)-isogeny, and Richelot’s construction identifies $\hat{\CC}$ explicitly. Given $\CC$ with equation $Y^2 = f_6(X, Z)$, a sextic in homogeneous coordinates, any factorization $f_6 = A \cdot B \cdot C$ into three quadratic polynomials $A, B, C$ defines a curve $\hat{\CC}$ via
\begin{equation}\label{Richelot}
\Delta_{ABC} \cdot Y^2 = [A, B] [A, C] [B, C],
\end{equation}
where $[A, B] = A' B - A B'$ with $A' = \frac{dA}{dx}$ (assuming a dehomogenized coordinate $x = X/Z$), and $\Delta_{ABC}$ is the determinant of the coefficients of $A, B, C$ in the basis $\{x^2, xz, z^2\}$. It was shown in \cite{MR970659} that $\Jac(\CC)$ and $\Jac(\hat{\CC})$ are (2,2)-isogenous, and there are exactly 15 such factorizations (corresponding to partitions of 6 roots into 3 pairs), yielding all distinct (2,2)-isogenous Abelian surfaces to $\Jac(\CC)$.
\par The geometric insight stems from the isomorphism $S_6 \cong \Sp(4, \mathbb{F}_2)$, where $S_6$ permutes the 6 Weierstrass points of $\CC$ (or Theta divisors containing a fixed order-two point), and $\Sp(4, \mathbb{F}_2)$ acts on $\AA[2]$. 
\subsubsection{An explicit model using Theta constants}
\label{computation}
To compute Richelot isogenies explicitly, we adopt a Theta function approach, following \cite{MR2367218}. For $\AA = \Jac(\CC)$, consider G\"opel groups $G = \{ P_0, P_{15}, P_{23}, P_{46} \}$ and $G' = \{ P_0, P_{12}, P_{34}, P_{56} \}$, where $\AA[2] = \{ P_{ij} \}_{i<j}$ labels the 2-torsion points, with $P_0=P_{66}$ the identity. These satisfy $G + G' = \AA[2]$ and $G \cap G' = \{ P_0 \}$, as $|G + G'| = 4 \cdot 4 / 1 = 16$. Set $\hat{\AA} = \AA/G$, and let $\hat{G}$ be the image of $G'$ in $\hat{\AA}$. We aim to relate the Rosenhain roots of $\CC$:
\begin{equation}\label{Eq:Rosenhain_b}
\CC\colon \quad  y^2 = x z (x - z) (x - \lambda_1 z) (x - \lambda_2 z) (x - \lambda_3 z),
\end{equation}
where $\lambda_4 = 0$, $\lambda_5 = 1$, $\lambda_6 = \infty$, to those of $\hat{\CC}$ with
\begin{equation}\label{Eq:Rosenhain2}
\hat{\CC}\colon \quad y^2 = x (x - 1) (x - \Lambda_1) (x - \Lambda_2) (x - \Lambda_3).
\end{equation}
\par Theta constants give explicit equations for the moduli of $\Jac(\CC)$ and $\Jac(\hat{\CC})$: for a period matrix $\tau \in \mathbb{H}_2$, we use $\theta_i $ for $\CC$ and $\Theta_i$ for $\hat{\CC}$, with characteristics forming Göpel groups. That is, we let $G = \{ P_0, P_{15}, P_{23}, P_{46} \}$ correspond to the G\"opel group in $\mathbb{F}_2^4$ formed by the characteristics
\begin{equation}
\left( \begin{smallmatrix} 0 & 0 \\ 0 & 0 \end{smallmatrix} \right), \left( \begin{smallmatrix} 0 & 0 \\ 1 & 1 \end{smallmatrix} \right), \left( \begin{smallmatrix} 0 & 0 \\ 1 & 0 \end{smallmatrix} \right), \left( \begin{smallmatrix} 0 & 0 \\ 0 & 1 \end{smallmatrix} \right),
\end{equation}
yielding the Theta constants $\theta_1, \theta_2, \theta_3, \theta_4$.  They determine the moduli $\lambda_1, \lambda_2, \lambda_3$ in Equation~(\ref{Picard}) by plugging in the quadratic relations~(\ref{Eq:degree2doubling}) and~(\ref{Eq:degree2doublingR}).

\par The isogeny $\Psi \colon \Jac(\CC) \longrightarrow \Jac(\hat{\CC})$ modifies the lattice, and the dual Theta constants $\Theta_i$ for $\hat{\CC}$ relate to $\theta_i$ via the isogeny’s action in~\eqref{eqn:dual_Goepel_groups}. Results in \cite{MR2367218} and \cite{MR4421430} give
\begin{equation}\label{Picard_sq}
\Lambda_1 = \frac{\Theta_1^2 \Theta_3^2}{\Theta_2^2 \Theta_4^2}, \quad \Lambda_2 = \frac{\Theta_3^2 \Theta_8^2}{\Theta_4^2 \Theta_{10}^2}, \quad \Lambda_3 = \frac{\Theta_1^2 \Theta_8^2}{\Theta_2^2 \Theta_{10}^2}.
\end{equation}

\par For $G$ the corresponding Richelot isogeny is obtained by factoring $f_6 = x z (x - z) (x - \lambda_1 z) (x - \lambda_2 z) (x - \lambda_3 z)$ into
\begin{equation}
A = (x - \lambda_1 z)(x - z), \quad B = (x - \lambda_2 z)(x - \lambda_3 z), \quad C = x z,
\end{equation}
adjusting for homogeneity (e.g., $C = x z$ at infinity). Applying \eqref{Richelot}, we compute $\hat{\CC}$ and its Igusa invariants, which are rational functions of $\{\theta_1, \theta_2, \theta_3, \theta_4\}$. A quadratic twist $\hat{\CC}^{(\mu)}$ with
\[
\mu = \frac{(\theta_1 \theta_2 - \theta_3 \theta_4)^2 (\theta_1^2 + \theta_2^2 - \theta_3^2 - \theta_4^2)(\theta_1^2 - \theta_2^2 + \theta_3^2 - \theta_4^2)}{4 \theta_1 \theta_2 \theta_3 \theta_4 (\theta_1^2 + \theta_2^2 + \theta_3^2 + \theta_4^2)(\theta_1^2 - \theta_2^2 - \theta_3^2 + \theta_4^2)}
\]
adjusts the roots to
\begin{equation}\label{Picard_sq_b}
\begin{aligned}
\Lambda_1 & = \frac{(\theta_1^2 + \theta_2^2 + \theta_3^2 + \theta_4^2)(\theta_1^2 - \theta_2^2 - \theta_3^2 + \theta_4^2)}{(\theta_1^2 + \theta_2^2 - \theta_3^2 - \theta_4^2)(\theta_1^2 - \theta_2^2 + \theta_3^2 - \theta_4^2)}, \\
\Lambda_2 & = \frac{(\theta_1^2 - \theta_2^2 - \theta_3^2 + \theta_4^2)(\theta_1^2 \theta_2^2 + \theta_3^2 \theta_4^2 + 2 \theta_1 \theta_2 \theta_3 \theta_4)}{(\theta_1^2 - \theta_2^2 + \theta_3^2 - \theta_4^2)(\theta_1^2 \theta_2^2 - \theta_3^2 \theta_4^2)}, \\
\Lambda_3 & = \frac{(\theta_1^2 + \theta_2^2 + \theta_3^2 + \theta_4^2)(\theta_1^2 \theta_2^2 + \theta_3^2 \theta_4^2 + 2 \theta_1 \theta_2 \theta_3 \theta_4)}{(\theta_1^2 + \theta_2^2 - \theta_3^2 - \theta_4^2)(\theta_1^2 \theta_2^2 - \theta_3^2 \theta_4^2)},
\end{aligned}
\end{equation}
matching the Richelot construction’s invariants, verifying isomorphism. Similarly, for the G\"opel group $G' = \{ P_0, P_{12}, P_{34}, P_{56} \}$, the dual isogeny $\hat{\Psi}\colon \hat{\AA} \longrightarrow \AA = \hat{\AA}/\hat{G}$ uses $\hat{A} = [B, C]$, $\hat{B} = [A, C]$, $\hat{C} = [A, B]$ to recover $\CC$. These Richelot isogenies realize the pair of dual $(2,2)$-isogenies, given by
\begin{equation}
\AA = \Jac(\CC) \ \overset{\Psi}{\longrightarrow} \ \hat{\AA} = \Jac(\hat{\CC}) \cong \AA/G \ \overset{\hat{\Psi}}{\longrightarrow} \ \AA \cong \hat{\AA}/\hat{G},
\end{equation}
for the complementary maximal isotropic subgroups $G, G'$. If we introduce new rescaled moduli
\begin{equation*}
\lambda_1' = \frac{\lambda_1 + \lambda_2 \lambda_3}{l}, \quad \lambda_2' = \frac{\lambda_2 + \lambda_1 \lambda_3}{l}, \quad \lambda_3' = \frac{\lambda_3 + \lambda_1 \lambda_2}{l},
\end{equation*}
with $l^2 = \lambda_1 \lambda_2 \lambda_3$, and similarly for $\Lambda_i'$, then these relate symmetrically, as shown in \cite{MR4421430}:

\begin{prop}\label{lem:2isog_curve}
The moduli of $\CC$ in \eqref{Eq:Rosenhain_b} and $\hat{\CC}$ in \eqref{Eq:Rosenhain2} satisfy
\begin{equation}\label{relations_RosRoots}
\begin{aligned}
\Lambda_1' & = 2 \frac{2 \lambda_1' - \lambda_2' - \lambda_3'}{\lambda_2' - \lambda_3'}, &
\lambda_1' & = 2 \frac{2 \Lambda_1' - \Lambda_2' - \Lambda_3'}{\Lambda_2' - \Lambda_3'}, \\
\Lambda_2' - \Lambda_1' & = - \frac{4 (\lambda_1' - \lambda_2')(\lambda_1' - \lambda_3')}{(\lambda_1' + 2)(\lambda_2' - \lambda_3')}, &
\lambda_2' - \lambda_1' & = - \frac{4 (\Lambda_1' - \Lambda_2')(\Lambda_1' - \Lambda_3')}{(\Lambda_1' + 2)(\Lambda_2' - \Lambda_3')}, \\
\Lambda_3' - \Lambda_1' & = - \frac{4 (\lambda_1' - \lambda_2')(\lambda_1' - \lambda_3')}{(\lambda_1' - 2)(\lambda_2' - \lambda_3')}, &
\lambda_3' - \lambda_1' & = - \frac{4 (\Lambda_1' - \Lambda_2')(\Lambda_1' - \Lambda_3')}{(\Lambda_1' - 2)(\Lambda_2' - \Lambda_3')}.
\end{aligned}
\end{equation}
\end{prop}

\section{Kummer Surfaces}
\label{kummer_srfc}
Given an Abelian surface $\AA$ over a field $k$, the \emph{singular Kummer surface} is the quotient $\mathcal{K}_\AA = \AA / \langle - \mathbb{I} \rangle$, where the involution $-\mathbb{I}\colon \AA \longrightarrow \AA$ acts by $P \mapsto -P$. $\mathcal{K}_\AA $ is a singular algebraic variety of dimension two with 16 ordinary double points, each corresponding to an order-two point in $\AA[2] = \{ P \in \AA(\bar{k}) \mid 2P = 0 \}$, the set of which has cardinality $2^{2 \cdot 2} = 16$ over an algebraically closed field. 
\par The quotient construction of $\mathcal{K}_\AA$ leverages the group structure of $\AA$: the involution $-\mathbb{I}$ is an automorphism of order two, and its fixed points, the 2-torsion points, map to singularities on $\mathcal{K}_\AA$. Each singularity is locally isomorphic to the quotient of $\C^2$ by the action $z \mapsto -z$, an $A_1$ singularity, characterized by a quadratic cone. The minimal resolution is denoted the \textit{Kummer surface} $\Kum(\AA)$ and is a smooth K3 surface, obtained by blowing up each double point by gluing in an exceptional $(-2)$-curve; see \cite{shioda-inose} and \cite{2019-4} for details. Conversely, there is a map $\Kum(\AA) \longrightarrow \mathcal{K}_\AA$ contracting the $(-2)$ curves.
\subsection{Kummer surfaces with principal polarizations}
For an Abelian surface $\AA$ equipped with a principal polarization $\mathscr{L}=\mathcal{O}_\AA(\Theta)$, one can always choose the Theta divisor to satisfy $(-\mathbb{I})^* \Theta=\Theta$, that is, to be a \emph{symmetric Theta divisor}. The Abelian surface $\AA$ then maps to the complete linear system $|2\Theta|$. The rational map $\varphi_2\colon \AA \longrightarrow \bP^3$ associated with the line bundle $\mathscr{L}^2$ induces a map to projective space $\mathbb{P}^3$, its image $\varphi_2(\AA)$ is a quartic hypersurface, reflecting the degree of the line bundle $\mathscr{L}^2$ associated with the doubled principal polarization, and $\varphi_2$ factors via an embedding through the projection $\AA \longrightarrow \AA/\langle -\mathbb{I} \rangle$; see \cite{MR2062673}. We identify the singular Kummer variety $\AA/\langle -\mathbb{I} \rangle$ with its image $\mathcal{K}_\AA$ in $\mathbb{P}^3$.   We will give explicit algebraic equations for the image in Section~\ref{sssec:Kummer_even}. 
\par  For the principally polarized Abelian surface $\AA=\operatorname{Jac}{(\CC)}$ associated with a smooth genus 2 curve $\CC$ with the standard Theta divisor $\Theta \cong [\CC]$, we will now construct the morphism to $\bP^3$ explicitly, using the Theta functions from Section~\ref{sec:theta_functions}.  In fact, the level-two Theta functions directly give the map $\varphi_2$ into $\mathbb{P}^3$ when we identify $\mathscr{L}$ with the line bundle that was introduced in Section~\ref{sec:theta_functions}. Specifically, assume that $\AA$ is not a product of elliptic curves, for $\AA = \C^2 / \Lambda$ with a period matrix $\tau \in \mathbb{H}_2$, the functions $\varphi_2(z) = \left( \theta \left[ \begin{smallmatrix} 0 \\ b \end{smallmatrix} \right] (z, \tau/2) \right)_{b \in (\Z/2\Z)^2}$ satisfy $\varphi_2(z) = \varphi_2(-z)$ and induce the morphism $\mathcal{K}_\AA \longrightarrow \bP^3$; see Section~\ref{sssec:relats_thetas} and \cite{MR2062673}*{Thm.~4.8.1}).  This morphism $\varphi_2$ is in fact an embedding and its image is a quartic surface in $\bP^3$ with Theta constants determining its equation. We have the following:

 \begin{lem} 
Let $\AA$ be a principally polarized Abelian surface over $\C$, not isomorphic to a product of elliptic curves. The map $\varphi_2\colon \mathcal{K}_\AA \longrightarrow \bP^3$ defined by $\varphi_2(z) = \left( \theta \left[ \begin{smallmatrix} 0 \\ b \end{smallmatrix} \right] (z, \tau/2) \right)_{b \in (\Z/2\Z)^2}$ is an embedding.
\end{lem}

\begin{proof}
The Theta functions $\theta_i(z) = \theta \left[ \begin{smallmatrix} 0 \\ b \end{smallmatrix} \right] (z, \tau/2)$ for $b \in (\Z/2\Z)^2$ are even, satisfying $\theta_i(-z) = \theta_i(z)$ due to the characteristic’s symmetry under the involution $\iota$. Thus, $\varphi_2 \colon \AA \longrightarrow \bP^3$ factors through the quotient $\pi \colon \AA \longrightarrow \mathcal{K}_\AA$, inducing a well-defined map $\overline{\varphi_2} \colon \mathcal{K}_\AA \longrightarrow \bP^3$. For $\AA$ principally polarized, the line bundle $\mathscr{L}^2$ corresponds to $2\Theta$, where $\Theta$ is the Theta divisor. The space of sections $\Gamma(\AA, \mathscr{L}^2)$ is four-dimensional, spanned by the $\theta_i$. If $\AA$ is not a product $E_1 \times E_2$, the polarization ensures $\mathscr{L}^2$ is very ample on $\Kum(\AA)$, separating points and tangent vectors. Specifically, for distinct points $x, y \in \mathcal{K}_\AA$ (not both in $\AA[2]$), there exists $i$ such that $\theta_i(x) \neq \theta_i(y)$, and for a point $x$ with tangent direction $v$, some $\theta_i$ has non-zero derivative along $v$. Thus, $\overline{\varphi_2}$ embeds $\mathcal{K}_\AA$ as a quartic surface with 16 nodes at the images of $\AA[2]$.
\end{proof}
\par The image of each order-two point of $\AA$ is a singular point on the Kummer surface $\mathcal{K}_\AA$ and called a \emph{node}. That is, nodes are the images of the order-two points $P_{ij} \in \AA[2]$.  Similarly, any Theta divisor of $\AA$ is mapped to the intersection of the singular Kummer surface with a plane.  We call such a singular plane a \emph{trope}. In the complete linear system $|2\Theta|$ on $\operatorname{Jac}(\CC)$, the odd symmetric Theta divisors $\Theta_i$ on $\operatorname{Jac}(\CC)$ are mapped to six tropes $\mathsf{T}_i$, and the even symmetric Theta divisors $\Theta_{ijk}$ on $\operatorname{Jac}(\CC)$ are mapped to the tropes $\mathsf{T}_{ijk}$. The ten tropes $\mathsf{T}_{ijk}$ with $1\le i < j < k \le 6$  correspond to partitions of $\{1,\dots,6\}$ into two unordered sets of three, say $\lbrace i, j, k\rbrace$ $\lbrace l,m, n\rbrace$. We use the formulas for  $\mathsf{T}_{ijk}$ from  \cite{MR1406090}*{Sec.~3.7} paying careful attention to the fact that  we have moved the root $\lambda_6$ to infinity.  Hence, a configuration of sixteen nodes and sixteen tropes is induced on $\mathcal{K}_\AA$, where each trope contains six nodes and each node is contained in six tropes. This configuration is called the \emph{$16_{6}$ configuration on the singular Kummer surface} $\mathcal{K}_\AA$.
\subsection{Kummer surfaces from odd characteristics}
\label{ssec:Kummer_odd}
The Jacobian $\operatorname{Jac}{(\CC)}$ of a smooth genus two curve $\CC$ is birational to the symmetric product $\operatorname{Sym}^2(\CC)$. Let us consider the function field  of $\operatorname{Sym}^2(\CC) = (\CC\times\CC)/\langle \sigma \rangle$ where $\sigma$ interchanges the copies of $\CC$. For a smooth genus-two curve $\CC$ in Rosenhain form~(\ref{Eq:Rosenhain}), the function field of the variety $\operatorname{Sym}^2(\CC)/\langle \imath \times  \imath \rangle$ -- where $\imath \times  \imath$ is the involution on $\operatorname{Sym}^2(\CC)$ induced by the hyperelliptic involution $\imath$ on $\CC$ -- is generated by $z_1=z^{(1)}z^{(2)}$, $z_2=x^{(1)}z^{(2)}+x^{(2)}z^{(1)}$, $z_3=x^{(1)}x^{(2)}$, and $\tilde{z}_4=y^{(1)}y^{(2)}$, subject to the relation
\begin{equation}
\label{kummer_middle}
 \tilde{z}_4^2 = z_1 z_3 \big(  z_1  - z_2 +  z_3 \big)  \prod_{i=1}^3 \big( \lambda_i^2 \, z_1  -  \lambda_i \, z_2 +  z_3 \big) \;.
\end{equation}
Equation~(\ref{kummer_middle}) is a special case of a \emph{double sextic surface}, i.e., the double cover of $\mathbb{P}^2$ branched on the union of six lines (and hence over a sextic curve). The minimal resolution of a double sextic surface is always a K3 surface; see \cite{MR1922094}. Equation~(\ref{kummer_middle}) is known as \emph{Shioda sextic} associated with $\AA=\operatorname{Jac}{(\CC)}$; see \cite{MR2296439}.  It follows that the minimal resolution of $\mathcal{W}$ is isomorphic to the Kummer surface $\operatorname{Kum}(\AA)$. 
\par In Equation~(\ref{kummer_middle}) the double cover of $\mathbb{P}^2= \mathbb{P}(z_1, z_2, z_3)$ is branched along the following six lines $\mathsf{T}_i$ with $1 \le i \le 6$, given by
\begin{equation}
\label{eqn:6lines}
 \lambda_i^2 \, z_1  -  \lambda_i \, z_2 +  z_3 =0 \; \text{with $1\le i \le 3$}, \quad  z_1=0,\quad  z_1  - z_2 +  z_3=0,  \quad  z_3=0 \,.
\end{equation}
These six lines are precisely the aforementioned odd tropes and tangent to the common conic $K_2= z_2^2 - 4 \, z_1 z_3=0$. Conversely, it is easy to see that any six lines tangent to a common conic can always be brought into the form of Equations~(\ref{eqn:6lines}) using a projective transformation.  Thus, we can rewrite Equation~\eqref{kummer_middle} as
\begin{equation}
 \tilde{z}_4^2 = \mathsf{T}_1 \mathsf{T}_2 \mathsf{T}_3 \mathsf{T}_4 \mathsf{T}_5 \mathsf{T}_6 \;.
\end{equation}
Moreover, one easily checks that Moreover, there are fifteen linear relations between the odd tropes, and a generating set is given by
\begin{equation}
\label{eqn:tropes_QR}
\begin{split}
 \mathsf{T}_1 & = (1-\lambda_1) \mathsf{T}_4 +\lambda_1 \mathsf{T}_5 + \lambda_1 (\lambda_1-1) \mathsf{T}_6,\\
 \mathsf{T}_2 & = (1-\lambda_2) \mathsf{T}_4 +\lambda_2 \mathsf{T}_5 +  \lambda_2 (\lambda_2-1) \mathsf{T}_6,\\
 \mathsf{T}_3 & = (1-\lambda_3) \mathsf{T}_4 +\lambda_3 \mathsf{T}_5 +  \lambda_3 (\lambda_3-1) \mathsf{T}_6.
\end{split}
\end{equation}
Humbert proved that the Kummer plane $(\mathbb{P}^2; \mathsf{T}_1, \dots, \mathsf{T}_6)$ inherits essential information of the principally polarized Abelian surface $\AA=\operatorname{Jac}{(\CC)}$ itself; see \cite{Humbert1901}.
\par We consider the blow up $p\colon \widehat{\AA} \longrightarrow \AA$, replacing the sixteen order-two points with the exceptional curves $E_1, \dots, E_{16}$. The linear system $|4 p^* \Theta - \sum E_i|$ determines a morphism of degree two from $\widehat{\AA}$ to a complete intersection of three quadrics in $\mathbb{P}^5$.  The net spanned by these quadrics is isomorphic to $\mathbb{P}^2$ with a discriminant locus corresponding to the union of six lines~\eqref{eqn:6lines}.   In \cite{MR4421430} we showed the following:
\begin{thm}
\label{thm_quadrics}
The Kummer surface $\operatorname{Kum}(\AA)$ associated with the Jacobian $\AA=\operatorname{Jac}(\CC)$ of a genus-two curve~$\CC$ in Rosenhain normal form~(\ref{Eq:Rosenhain}) is the complete intersection of three quadrics in $\mathbb{P}^5=\mathbb{P}(\mathsf{t}_1:\cdots:\mathsf{t}_6)$ given by
\begin{equation}
\label{Kum:quadrics}
\begin{split}
 \mathsf{t}^2_1 & = (1-\lambda_1) \mathsf{t}^2_4 +\lambda_1 \mathsf{t}^2_5 +  \lambda_1 (\lambda_1-1) \mathsf{t}^2_6,\\
 \mathsf{t}^2_2 & = (1-\lambda_2) \mathsf{t}^2_4 +\lambda_2 \mathsf{t}^2_5 + \lambda_2 (\lambda_2-1) \mathsf{t}^2_6,\\
 \mathsf{t}^2_3 & = (1-\lambda_3) \mathsf{t}^2_4 +\lambda_3 \mathsf{t}^2_5 +  \lambda_3 (\lambda_3-1) \mathsf{t}^2_6.
\end{split}
\end{equation}
\end{thm}
\begin{proof}
The Shioda sextic is given by the double cover of $\mathbb{P}^2$ branched along the reducible sextic that is the union of the six lines given by the tropes $\mathsf{T}_1, \dots, \mathsf{T}_6$. The tropes $\mathsf{T}_1, \dots, \mathsf{T}_6$ satisfy fifteen linear relations of rank three equivalent to Equations~(\ref{eqn:tropes_QR}). Introducing $\mathsf{T}_i=\mathsf{t}_i^2$ such that $\tilde{z}_4=\mathsf{t}_1 \mathsf{t}_2 \mathsf{t}_3 \mathsf{t}_4 \mathsf{t}_5 \mathsf{t}_6$, the (smooth) Kummer surface $\operatorname{Kum}(\AA)=\widehat{\AA}/\langle -\mathbb{I} \rangle$ is the complete intersection of three quadrics in $\mathbb{P}^5 \ni [\mathsf{t}_1:\mathsf{t}_2:\mathsf{t}_3:\mathsf{t}_4:\mathsf{t}_5:\mathsf{t}_6]$ given by Equations~(\ref{Kum:quadrics}).
\end{proof}
The following fact is well known: 
\begin{rem}
\label{fact:sections2_2}
For the principally polarized Abelian surface $(\AA,\mathscr{L})$, a basis of sections for $\mathscr{L}^4$, called the Theta functions of level $(2,2)$,  is given by the odd genus 2 Theta functions that we denote by $\theta_i(z)$ for $11 \le i \le 16$.
\end{rem}
In \cite{MR4421430}*{Lemma~4.11} it was shown that Thomae's formula provides an explicit relation between tropes and the Theta functions used in Theorem~\ref{thm_quadrics}. In particular, it was shown that the following holds:
\begin{rem}
\label{lem:bijection_tropes_thetas}
There is a bijection between the six tropes $\mathsf{T}_1, \dots, \mathsf{T}_6$ and the squares of the odd Theta functions $\theta^2_{11}(z), \dots,\theta^2_{16}(z)$ such that Equations~\eqref{eqn:6lines} are induced by the Mumford relations for Theta functions.
\end{rem}
\par The Shioda sextic is closely related to two quartic equations, known as the \emph{Baker quartic} and \emph{Cassels-Flynn quartic}. The Baker quartic is obtained directly from the Shioda sextic in Equation~\eqref{eqn:6lines}, using parameters
\begin{equation}
\label{eqn:Ls}
\begin{array}{lclclcl}
 L_4 & = & 1 + \lambda_1 + \lambda_2 + \lambda_3 \,, && L_3 & = &  \lambda_1 + \lambda_2 + \lambda_3 +  \lambda_1 \lambda_2 + \lambda_2 \lambda_3 + \lambda_1 \lambda_3 \,,\\
 L_1 & = & \lambda_1 \lambda_2 \lambda _3 \,, && L_2 & = &   \lambda_1 \lambda_2 + \lambda_2 \lambda_3 + \lambda_1 \lambda_3 +  \lambda_1 \lambda_2 \lambda _3 \,,
\end{array}
\end{equation}
and the variable transformation given by
\begin{equation}
\label{eqn:Baker_transfo}
\begin{array}{lcl}
 \multicolumn{3}{c}{\mathbf{W} =  K_2 \, z_1 \,, \qquad \mathbf{X} =  K_2 \, z_2 \,, \qquad \mathbf{Y} =  K_2 \, z_3 \,,} \\
 \mathbf{Z} = 2 \tilde{z}_4 - \Big( L_1 z_1^2 z_2 - 2 L_2 z_1^2 z_3 + L_3 z_1 z_2 z_3 - 2 L_4 z_1 z_3^2 + z_2 z_3^2 \Big)  \,.
\end{array}
\end{equation}
Equation~(\ref{kummer_middle}) is then equivalent to the quartic surface in $\mathbb{P}^3 = \mathbb{P}( \mathbf{W}, \mathbf{X}, \mathbf{Y}, \mathbf{Z})$ given by the vanishing locus of the following \emph{Baker determinant}, i.e.,
\begin{equation}
\label{eqn:Baker_det}
  \left| \begin{array}{cccc} 
  0 				& L_1 \mathbf{W}				& -\mathbf{Z}					& \mathbf{Y} \\
  L_1 \mathbf{W}	& 2 L_2 \mathbf{W}	+ 2 \mathbf{Z} 	& L_3 \mathbf{W} - \mathbf{Y}		& \mathbf{X} \\
  -\mathbf{Z}		&  L_3 \mathbf{W} - \mathbf{Y}		& 2 L_4 \mathbf{W}	- 2 \mathbf{X}	& \mathbf{W}\\
  \mathbf{Y}		& \mathbf{X}					& \mathbf{W}					& 0
 \end{array} \right| = 0 \,.
\end{equation}  
The Baker determinant was first derived in \cite{MR1554977}. 
\par In addition to $K_2$ introduced above, let us also define homogeneous polynomials $K_l=K_l(z_1,z_2,z_3)$ of degree $4-l$ for $l=0,1$ with
\begin{equation}
\label{eqn:Ks}
\begin{split}
K_2 =& \; z_2^2 -4 z_1z_3 \,, \qquad K_1 = -2 z_2 z_3^2 - 2 L_1 z_1^2 z_2 + 4 L_2 z_1^2z_3 - 2 L_3 z_1 z_2 z_3 + 4 L_4 z_1 z_3^2 \,,\\
K_0 =& \;L_1^2 z_1^4 - 2 L_1 L_3 z_1^3 z_3 + (2L_1 - 4 L_2 L_4 + L_3^2) z_1^2 z_3^2 + 4 L_1 L_4 z_1^2 z_2 z_3 \\
&+ 3 (-2 L_1 z_2^2 z_3 + 2 L_2 z_2 z_3^2 - L_3 z_3^3) z_1+ z_3^4 \,.
\end{split}
\end{equation}
A variable transformation in Equation~(\ref{kummer_middle}), given by
\begin{equation}
\label{eqn:transfo_1}
 \tilde{z}_4 = \frac{1}{4} \left( K_2 z_4  + 2 K_1 \right) \,,
\end{equation}
or, equivalently, the variable transformation  in Equation~(\ref{eqn:Baker_det}), given by
\begin{equation}
\label{eqn:transfo_2}
 \mathbf{W} =  K_2 \, z_1 \,, \qquad \mathbf{X} =  K_2 \, z_2 \,, \qquad \mathbf{Y} =  K_2 \, z_3 \,, \qquad \mathbf{Z} = K_2 z_4 + K_1 \,,
\end{equation} 
yields the quartic  projective surface $\mathcal{K}$ in $\mathbb{P}^3=\mathbb{P}(z_1,z_2,z_3,z_4)$ given by 
\begin{equation}
\label{kummer}
  K_2(z_1,z_2,z_3) \; z_4^2 \; + \; K_1(z_1,z_2,z_3)\; z_4 \; + \; K_0(z_1,z_2,z_3) = 0 \;.
\end{equation}
The quartic appeared in Cassels and Flynn \cite{MR1406090}*{Sec.~3} and is called the \emph{Cassels-Flynn quartic}.  Recall from Proposition~\ref{compactifications} that $\Q(\lambda_1, \lambda_2, \lambda_3)$ is the rational function field  of the moduli space $\mathbb{A}_2(2)$. In summary, for the principally polarized Abelian surface $\AA$ with the standard Theta divisor $\Theta \cong [\CC]$ and polarizing line bundle $\mathscr{L} = \mathcal{O}_\AA(\Theta)$ we have the following; see \cite{MR4323344}:
\begin{prop}
\label{prop:Kummers}
Assume that $\lambda_1, \lambda_2, \lambda_3$ are the Rosenhain roots of a smooth genus-two curve $\CC$ in Equation~(\ref{Eq:Rosenhain}).  The surfaces in Equations~(\ref{kummer_middle}) and~(\ref{eqn:Baker_det}) are birational equivalent over $\mathbb{Q}(\lambda_1, \lambda_2, \lambda_3)$.  In particular, the quartic hypersurfaces are isomorphic to the singular Kummer variety $\mathcal{K}_\AA$ associated with the principally polarized Abelian surface $(\AA, \mathscr{L})$.
\end{prop}
\subsection{Kummer surfaces from even characteristics}
\label{sssec:Kummer_even}
In this section we give explicit algebraic equations for the singular Kummer variety $\mathcal{K}_\AA = \AA / \langle -\mathbb{I} \rangle$,  associated with a principally polarized Abelian variety $(\AA, \mathscr{L})$: they are called the \emph{Hudson quartic}, \emph{G\"opel quartic}, and \emph{Rosenhain quartic}; see results in \cite{MR2062673}. 
\par Let $\mathscr{L}$ again be the symmetric ample line bundle of an Abelian surface $\AA$ defining its  principal polarization and consider the rational map $\varphi_2\colon \AA \longrightarrow \bP^3$ associated with the line bundle $\mathscr{L}^2$. Its image $\varphi_2(\AA)$  is a quartic surface in $\bP^3$ which in projective coordinates $[w:x:y:z]$ can be written as
\begin{gather}
\label{Eq:QuarticSurfaces12}
    0 = \xi_0 \, (w^4+x^4+y^4+z^4)  + \xi_4 \, w x y z \qquad \\
\nonumber
    +\xi_1 \, \big(w^2 z^2+x^2 y^2\big)  +\xi_2 \, \big(w^2 x^2+y^2 z^2\big)  +\xi_3 \, \big(w^2 y^2+x^2 z^2\big) ,
\end{gather}
with $[\xi_0:\xi_1:\xi_2:\xi_3:\xi_4] \in \bP^4$. Any general member of the family~(\ref{Eq:QuarticSurfaces12}) is smooth. As soon as the surface is singular at a general point, it must have sixteen singular nodal points because of its symmetry. The discriminant turns out to be a homogeneous polynomial of degree eighteen in the parameters $[\xi_0:\xi_1:\xi_2:\xi_3:\xi_4] \in \bP^4$ and was determined in \cite{MR2062673}*{Sec.~7.7 (3)}. Thus, the Kummer surfaces form an open set among the hypersurfaces in Equation~(\ref{Eq:QuarticSurfaces12}) with parameters $[\xi_0:\xi_1:\xi_2:\xi_3:\xi_4] \in \bP^4$, namely the ones that make the only irreducible factor of degree three in the discriminant vanish, i.e.,
\begin{equation}
 \xi_0 \, \big( 16 \xi_0^2 - 4 \xi_1^2-4 \xi_2^2 - 4 \xi_3^3+ \xi_4^2\big) + 4 \, \xi_1 \xi_2 \xi_3 =0.
\end{equation}
Setting $\xi_0=1$ and using the affine moduli $\xi_1=-A$, $\xi_2=-B$, $\xi_3=-C$, $\xi_4=2 D$, we obtain the normal form of a nodal quartic surface. Thus, we
say that a \emph{Hudson quartic} is the surface in $\bP^3 = \bP(w,x,y,z)$ given by
\begin{gather}
\label{Goepel-Quartic}
  0 = w^4+x^4+y^4+z^4  + 2  D  w x y z \\
\nonumber
    - A  \big(w^2 z^2+x^2 y^2\big)  - B \big(w^2 x^2+y^2 z^2\big)  - C  \big(w^2 y^2+x^2 z^2\big)  \;, 
\end{gather}
where $A, B, C, D \in \C$ such that
\begin{equation}
\label{paramGH}
 D^2 = A^2 + B^2 + C^2 + ABC - 4\;.
\end{equation}
The Hudson quartic in Equation~(\ref{Goepel-Quartic}) is invariant under the transformations changing signs of two coordinates, generated by 
\[
 [w:x:y:z] \to [-w:-x:y:z], \quad [w:x:y:z] \to [-w:x:-y:z],
\] 
and under the permutations of variables, generated by $[w:x:y:z] \to [x:w:z:y]$ and $[w:x:y:z] \to [y:z:w:x]$. 
\par A different quartic hypersurface equation is due to Kummer~\cite{MR1579281} and Borchardt~\cite{MR1579732}. We say that a \emph{G\"opel quartic} is the surface in $\bP^3 = \bP(P,Q,R,S)$ given by
\begin{equation}
\label{Kummer-Quartic}
    \Phi^2 - 4 \, \delta^2 SPQR  =0\;,
\end{equation}
with 
\begin{equation}
\Phi= P^2 + Q^2 + R^2 +S^2 - \alpha \, \big( PS + QR \big)  - \beta \,  \big( PQ + RS \big) - \gamma \,  \big(PR+QS\big) \;,
\end{equation}
where $\alpha, \beta, \gamma, \delta \in \C$ such that
\begin{equation}
\label{paramG}
 \delta^2 = \alpha^2 + \beta^2 + \gamma^2 + \alpha \beta \gamma -4 \;.
\end{equation}

We have the following:
\begin{prop}
\label{lem:GH}
The Hudson quartic in Equation~(\ref{Goepel-Quartic}) is isomorphic to the G\"opel quartic in Equation~(\ref{Kummer-Quartic}), for suitable sets of parameters $(A,B,C,D)$ satisfying Equation~(\ref{paramGH}) and $(\alpha, \beta, \gamma,\delta)$ satisfying Equation~(\ref{paramG}). In particular, the quartic hypersurfaces are isomorphic to the singular Kummer variety $\mathcal{K}_\AA$ associated with the principally polarized Abelian surface $(\AA, \mathscr{L})$.
\end{prop}
\begin{proof}
We first introduce complex numbers $[w_0:x_0:y_0:z_0] \in \bP^3$ such that
\begin{gather}
\nonumber
  A   = \frac{w_0^4 - x_0^4-y_0^4+z_0^4}{w_0^2 z_0^2 - x_0^2 y_0^2}, \quad
  B  = \frac{w_0^4 + x_0^4-y_0^4-z_0^4}{w_0^2 x_0^2 - y_0^2 z_0^2}, \quad
  C  = \frac{w_0^4 - x_0^4+y_0^4-z_0^4}{w_0^2 y_0^2 - x_0^2 z_0^2}, \\
\label{Eq:paramABCD}
   D  = \frac{w_0 x_0 y_0 z_0 \prod_{\epsilon, \epsilon' \in \lbrace \pm1 \rbrace} (w_0^2 + \epsilon x_0^2+ \epsilon' y_0^2+ \epsilon \epsilon' z_0^2)}{(w_0^2 z_0^2 - x_0^2 y_0^2)
   (w_0^2 x_0^2 - y_0^2 z_0^2)(w_0^2 y_0^2 - x_0^2 z_0^2)}, 
 \end{gather}
and
\begin{gather}
\nonumber
  \alpha  = 2 \, \frac{w_0^2 \, y_0^2 + x_0^2 \, z_0^2}{w_0^2 \, y_0^2 - x_0^2 \, z_0^2}, \quad
  \beta  = 2 \, \frac{w_0^2 \, x_0^2 + y_0^2 \, z_0^2}{w_0^2 \, x_0^2 - y_0^2 \, z_0^2}, \quad
  \gamma  =  2\, \frac{w_0^2 \, z_0^2 + x_0^2 \, y_0^2}{w_0^2 \, z_0^2 - x_0^2 \, y_0^2},\\
\label{KummerParameter_c}
  \delta^2 = 16 \, \frac{w_0^4 x_0^4 y_0^4 z_0^4 \, \prod_{\epsilon, \epsilon' \in \lbrace \pm1 \rbrace} (w_0^2 + \epsilon x_0^2+ \epsilon' y_0^2+ \epsilon \epsilon' z_0^2)}{\big(w_0^2 \, y_0^2 - x_0^2 \, z_0^2\big)^2 \, \big(w_0^2 \, x_0^2 - y_0^2 \, z_0^2\big)^2 \, \big(w_0^2 \, z_0^2 - x_0^2 \, y_0^2\big)^2}.
  \end{gather}
Then, the linear transformation given by
\begin{equation}
\label{eq:LinearTransfo}
\begin{split}
 P & =  w_0 \, w + x_0 \, x + y_0 \, y+ z_0 \, z \;,\\
 Q & =  w_0 \, w + x_0 \, x - y_0 \, y- z_0 \, z \;,\\
 R & =  w_0 \, w  - x_0 \, x - y_0 \, y+ z_0 \, z \;,\\
 S & =  w_0 \, w - x_0 \, x + y_0 \, y- z_0 \, z \;,
\end{split}
\end{equation} 
transforms Equation~(\ref{Kummer-Quartic}) into Equation~(\ref{Goepel-Quartic}). 
\end{proof}
\begin{lem}
\label{lem:16nodes}
Using the same notation as in Equation~(\ref{Eq:paramABCD}), the sixteen nodes of the Hudson quartic~(\ref{Goepel-Quartic}) are given by
\begin{gather*}
[w_0:x_0:y_0:z_0], [-w_0:-x_0:y_0:z_0], [-w_0:x_0:-y_0:z_0], [-w_0:x_0:y_0:-z_0],\\
[x_0:w_0:z_0:y_0], [-x_0:-w_0:z_0:y_0], [-x_0:w_0:-z_0:y_0], [-x_0:w_0:z_0:-y_0],\\
[y_0:z_0:w_0:x_0], [-y_0:-z_0:w_0:x_0], [-y_0:z_0:-w_0:x_0], [-y_0:z_0:w_0:-x_0],\\
[z_0:y_0:x_0:w_0], [-z_0:-y_0:x_0:w_0], [-z_0:y_0:-x_0:w_0], [-z_0:y_0:x_0:-w_0].
\end{gather*}
In particular, for generic parameters $(A,B,C,D)$, no node is contained in the coordinate planes $w=0$, $x=0$, $y=0$, or $z=0$.
\end{lem}
We also set
\begin{gather}
\label{KummerParameter_thetas}
p_0^2 = w_0^2 + x_0^2+ y_0^2+z_0^2, \quad q_0^2 = w_0^2 + x_0^2- y_0^2-z_0^2,\\
\nonumber
r_0^2 = w_0^2 - x_0^2+ y_0^2-z_0^2, \quad s_0^2 = w_0^2 - x_0^2- y_0^2+z_0^2.
\end{gather}
We have the following:
\begin{lem}
\label{lem:16nodesG}
The nodes of the G\"opel quartic~(\ref{Kummer-Quartic}) are given by
\begin{gather*}
[p_0^2:q_0^2:r_0^2:s_0^2], [q_0^2:p_0^2:s_0^2:r_0^2], [r_0^2:s_0^2:r_0^2:q_0^2], [s_0^2:r_0^2:q_0^2:p_0^2], \\
[w_0x_0+y_0z_0	:w_0x_0-y_0z_0	:0				:0],				[w_0x_0-y_0z_0	:w_0x_0+y_0z_0	:0				:0], \\
[0				:0				:w_0x_0+y_0z_0	:w_0x_0-y_0z_0], 	[0				:0				:w_0x_0-y_0z_0	:w_0x_0+y_0z_0],\\
[w_0z_0+x_0y_0	:0				:w_0z_0-x_0y_0	:0],				[0				:w_0z_0+x_0y_0	:0				:w_0z_0-x_0y_0], \\
[w_0z_0-x_0y_0	:0				:w_0z_0+x_0y_0	:0],				[0				:w_0z_0-x_0y_0	:0				:w_0z_0+x_0y_0], \\
[w_0y_0+x_0z_0	:0				:0				:w_0y_0-x_0z_0],	[0				:w_0y_0+x_0z_0	:w_0y_0-x_0z_0	:0], \\
[0				:w_0y_0-x_0z_0	:w_0y_0+x_0z_0	:0],				[w_0y_0-x_0z_0	:0				:0				:w_0y_0+x_0z_0].
\end{gather*}
\end{lem}

In~\cite{MR2062673}*{Prop.~\!10.3.2} yet another normal form for a nodal quartic surfaces was established. We say that a \emph{Rosenhain quartic} is the surface in $\bP^3 = \bP(Y_0, \dots, Y_3)$ given by
\begin{equation}
\label{Birkenhake-Lange-Quartic}
\begin{split}
a^2  \big(Y_0^2 Y_1^2 + Y_2^2 Y_3^2\big) + b^2 \big(Y_0^2 Y_2^2 + Y_1^2 Y_3^2\big) + c^2  \big(Y_0^2 Y_3^2 + Y_1^2 Y_2^2\big)\\
+ 2ab  \big(Y_0 Y_1 - Y_2 Y_3\big) \big(Y_0 Y_2 + Y_1 Y_3\big)  - 2ac \big(Y_0 Y_1 + Y_2 Y_3\big)  \big(Y_0 Y_3 + Y_1 Y_2\big) \\
+ 2 b c \big(Y_0 Y_2 - Y_1 Y_3\big)  \big(Y_0 Y_3 - Y_1 Y_2\big) + d^2  Y_0Y_1Y_2Y_3 =0 ,
\end{split}
\end{equation}
with $[a:b:c:d] \in \mathbb{P}^3$. Equation~(\ref{Birkenhake-Lange-Quartic}) is invariant under the Cremona transformation 
\begin{equation}
[ Y_0: Y_1: Y_2: Y_3] \  \mapsto \ [Y_1Y_2 Y_3: Y_0Y_2Y_3: Y_0Y_1Y_3: Y_0Y_1Y_2] \;.
\end{equation}
\par The following was proved in~\cite{MR2062673}*{Prop.~10.3.2}:
\begin{prop}
\label{lem:BL}
For generic parameters $[a:b:c:d] \in \mathbb{P}^3$, the Rosenhain quartic in Equation~(\ref{Birkenhake-Lange-Quartic}) is isomorphic to the singular Kummer variety $\mathcal{K}_\AA$ associated with the principally polarized Abelian surface $(\AA, \mathscr{L})$.
\end{prop}

We also have the following:
\begin{lem}
The Rosenhain quartic in Equation~(\ref{Birkenhake-Lange-Quartic}) is isomorphic to the Hudson quartic in Equation~(\ref{Goepel-Quartic}). 
\end{lem}
\begin{proof}
Using the same notation as in the proof of Proposition~\ref{lem:GH}, we set
\begin{equation}
 \begin{split}
 a &= 4 (w_0^2 z_0^2-x_0^2 y_0^2)(w_0 y_0+x_0 z_0),\\
 b &=(w_0^2+x_0^2-y_0^2-z_0^2)(w_0^2-x_0^2+y_0^2-z_0^2)(w_0x_0-y_0z_0) ,\\
 c & =(w_0^2-x_0^2-y_0^2+z_0^2)(w_0^2+x_0^2+y_0^2+z_0^2)(w_0x_0+y_0z_0) ,
 \end{split}
 \end{equation}
and a polynomial expression for $d^2$ of degree twelve that we do not write out explicitly. Then, the linear transformation given by
\begin{equation}
\label{eq:LinearTransfo2}
\begin{split}
 Y_0 & =  w_0 \, w + x_0 \, x + y_0 \, y+ z_0 \, z \;,\\
 Y_1 & =  w_0 \, w  + x_0 \, x - y_0 \, y- z_0 \, z \;,\\
 Y_2 & =  z_0 \, w+ y_0 \, x + x_0 \, y+ w_0 \, z \;,\\
 Y_3 & =  z_0 \, w + y_0 \, x - x_0 \, y- w_0 \, z \;,
\end{split}
\end{equation} 
transforms Equation~(\ref{Birkenhake-Lange-Quartic}) into Equation~(\ref{Goepel-Quartic}). 
\end{proof}
\subsection{An explicit model using Theta functions}
\label{computation1}
We provide an explicit model for the quartic hypersurfaces in terms of Theta functions. 
\par In addition to the identities from Section~\ref{sssec:relats_thetas}, we have the following identities for Theta functions with non-vanishing elliptic argument~\cites{MR0141643, MR0168805}:
\begin{equation}
\label{Eq:degree2doubling_z}
\begin{split}
 \theta_1 \, \theta_1(z) & = \Theta_1(z)^2 + \Theta_2(z)^2 + \Theta_3(z)^2 
 +\Theta_4(z)^2 \,,\\[0.2em]
 \theta_2 \, \theta_2(z) & = \Theta_1(z)^2 + \Theta_2(z)^2 - \Theta_3(z)^2 
 -\Theta_4(z)^2 \,,\\[0.2em]
  \theta_3 \, \theta_3(z) & = \Theta_1(z)^2 - \Theta_2(z)^2 - \Theta_3(z)^2 
 +\Theta_4(z)^2 \,,\\[0.2em]
  \theta_4 \, \theta_4(z) & = \Theta_1(z)^2 - \Theta_2(z)^2 + \Theta_3(z)^2 
 -\Theta_4(z)^2 \,,
\end{split}
\end{equation}
and
\begin{equation}
\label{Eq:Degree2doubling_z}
\begin{split}
 4 \, \Theta_1 \, \Theta_1(2z) & = \theta_1(z)^2 + \theta_2(z)^2 + \theta_3(z)^2 
 +\theta_4(z)^2 \,,\\[0.2em]
 4 \, \Theta_2 \, \Theta_2(2z) & = \theta_1(z)^2 + \theta_2(z)^2 - \theta_3(z)^2 
 -\theta_4(z)^2 \,,\\[0.2em]
 4\, \Theta_3 \, \Theta_3(2z) & = \theta_1(z)^2 - \theta_2(z)^2 - \theta_3(z)^2 
 +\theta_4(z)^2 \,,\\[0.2em]
 4\, \Theta_4 \, \Theta_4(2z) & = \theta_1(z)^2 - \theta_2(z)^2 + \theta_3(z)^2 
 -\theta_4(z)^2 \,.
\end{split}
\end{equation}
The following is a well-known fact:
\begin{rem}
\label{fact:sections1}
For the principally polarized Abelian surface $(\AA,\mathscr{L})$, a basis of sections for $\mathscr{L}^2$, called Theta functions of level two, is given by $\Theta_i(2z)$ or, alternatively, $\theta^2_i(z)$ for $1 \le i \le 4$ using Equations~\eqref{Eq:Degree2doubling_z}.
\end{rem}

We have the following:
\begin{prop}
\label{prop:ThetaImageG}
For the surface in $\bP^3$ given by Equation~(\ref{Kummer-Quartic}), the coordinates are given by
\begin{equation} 
\label{KummerVariables}
 [P:Q:R:S]=\left[\theta_1(z)^2:\theta_2(z)^2:\theta_3(z)^2:\theta_4(z)^2\right] \;,
\end{equation}
and the parameters are
\begin{gather}
\nonumber
 \alpha =  \frac{\theta_1^4-\theta_2^4-\theta_3^4+\theta_4^4}{\theta_1^2\theta_4^2-\theta_2^2\theta_3^2}, \quad
 \beta  =   \frac{\theta_1^4+\theta_2^4-\theta_3^4-\theta_4^4}{\theta_1^2\theta_2^2-\theta_3^2\theta_4^2}, \quad
 \gamma  = \frac{\theta_1^4-\theta_2^4+\theta_3^4-\theta_4^4}{\theta_1^2\theta_3^2-\theta_2^2\theta_4^2},\\
\label{KummerParameter}
 \delta = \frac{ \theta_1 \theta_2 \theta_3 \theta_4 \prod_{\epsilon, \epsilon' \in \lbrace \pm1 \rbrace} (\theta_1^2 + \epsilon \theta_2^2+ \epsilon' \theta_3^2+ \epsilon \epsilon' \theta_3^2)}{\big(\theta_1^2 \, \theta_2^2 - \theta_3^2 \, \theta_4^2\big) \, \big(\theta_1^2 \, \theta_3^2 - \theta_2^2 \, \theta_4^2\big) \, \big(\theta_1^2 \, \theta_4^2 - \theta_2^2 \, \theta_3^2\big)}.
\end{gather}
\end{prop}
\begin{proof}
A direct computation shows that the non-vanishing elliptic variables $[P:Q:R:S]=[\theta_1(z)^2:\theta_2(z)^2:\theta_3(z)^2:\theta_4(z)^2]$ satisfy Equation~(\ref{Kummer-Quartic}), 
with
\begin{equation}
\begin{split}
\label{KummerParameter_B}
 \alpha =  2\, \frac{\Theta_1^2 \Theta_3^2+\Theta_2^2\Theta_4^2}{\Theta_1^2 \Theta_3^2-\Theta_2^2\Theta_4^2}, \quad
 \beta  =   2\, \frac{\Theta_1^2 \Theta_2^2+\Theta_3^2\Theta_4^2}{\Theta_1^2 \Theta_2^2-\Theta_3^2\Theta_4^2}, \quad
 \gamma  =   2\, \frac{\Theta_1^2 \Theta_4^2+\Theta_2^2\Theta_3^2}{\Theta_1^2 \Theta_4^2-\Theta_2^2\Theta_3^2}.
\end{split}
\end{equation}
The parameters can be re-written 
\begin{equation}
\label{KummerParameter_b}
\begin{split}
 \alpha =  \frac{\theta_1^4-\theta_2^4-\theta_3^4+\theta_4^4}{\theta_1^2\theta_4^2-\theta_2^2\theta_3^2}, \quad
 \beta  =   \frac{\theta_1^4+\theta_2^4-\theta_3^4-\theta_4^4}{\theta_1^2\theta_2^2-\theta_3^2\theta_4^2}, \quad
 \gamma  = \frac{\theta_1^4-\theta_2^4+\theta_3^4-\theta_4^4}{\theta_1^2\theta_3^2-\theta_2^2\theta_4^2}.
\end{split}
\end{equation}
such that over $\mathbb{A}_2(4,8)$ the modulus $\delta$ is given as
\begin{equation*}
  \delta = \frac{ \theta_1 \theta_2 \theta_3 \theta_4 \prod_{\epsilon, \epsilon' \in \lbrace \pm1 \rbrace} (\theta_1^2 + \epsilon \theta_2^2+ \epsilon' \theta_3^2+ \epsilon \epsilon' \theta_3^2)}{\big(\theta_1^2 \, \theta_2^2 - \theta_3^2 \, \theta_4^2\big) \, \big(\theta_1^2 \, \theta_3^2 - \theta_2^2 \, \theta_4^2\big) \, \big(\theta_1^2 \, \theta_4^2 - \theta_2^2 \, \theta_3^2\big)}.
\end{equation*}
\end{proof}
\begin{rem}
The transformation $[w:x:y:z] \to [-w:x:y:z]$ is an isomorphism between a Hudson quartic in Equation~(\ref{Goepel-Quartic}) with moduli $(A, B, C, D)$ and the one with $(A, B, C, -D)$. Moreover, the two quartics coincide exactly along the coordinate planes $w=0$, $x=0$, $y=0$, or $z=0$.
\end{rem}
\par We also provide an explicit model for the Hudson quartic in terms of Theta functions.  In terms of Theta functions, the relation between Equation~(\ref{Goepel-Quartic}) and Equation~(\ref{Kummer-Quartic}) was first determined by Borchardt \cite{MR1579732}.  We have the following:
\begin{prop}
\label{prop:ThetaImageGH}
For the surface in $\bP^3$ given by Equation~(\ref{Goepel-Quartic}), the coordinates are given by
\begin{equation} 
 [w:x:y:z]=[ \Theta_1(2  z): \Theta_2(2  z): \Theta_3(2  z) : \Theta_4(2  z) ] \;,
\end{equation}
and the parameters are
\begin{gather}
  \nonumber
  A   = \frac{\Theta_1^4 - \Theta_2^4-\Theta_3^4+\Theta_4^4}{\Theta_1^2 \Theta_4^2 - \Theta_2^2 \Theta_3^2}, \quad
  B  = \frac{\Theta_1^4 + \Theta_2^4-\Theta_3^4-\Theta_4^4}{\Theta_1^2 \Theta_2^2 - \Theta_3^2 \Theta_4^2}, \quad
  C  = \frac{\Theta_1^4 - \Theta_2^4+\Theta_3^4-\Theta_4^4}{\Theta_1^2 \Theta_3^2 - \Theta_2^2 \Theta_4^2}, \\
  \label{KummerParameter3}
  D  = \frac{\Theta_1 \Theta_2 \Theta_3 \Theta_4 \prod_{\epsilon, \epsilon' \in \lbrace \pm1 \rbrace} (\Theta_1^2 + \epsilon \Theta_2^2+ \epsilon' \Theta_3^2+ \epsilon \epsilon' \Theta_4^2)}
  {(\Theta_1^2 \Theta_2^2 - \Theta_3^2 \Theta_4^2)(\Theta_1^2 \Theta_3^2 - \Theta_2^2 \Theta_4^2)(\Theta_1^2 \Theta_4^2 - \Theta_2^2 \Theta_3^2)}.
\end{gather}
\end{prop}
\begin{proof}
Comparing Equations~(\ref{KummerParameter_B}) and (\ref{KummerParameter_c}), we find a solution in terms of Theta function given by
\begin{equation}
 [w_0:x_0:y_0:z_0]=[\Theta_1:\Theta_2:\Theta_3:\Theta_4] \;.
 \end{equation}
Equations~(\ref{eq:LinearTransfo}) are equivalent to
\begin{equation}
 \label{2isogTHETA}
\begin{split}
  \theta_1^2(z)  & = \Theta_1 \, \Theta_1(2z) + \Theta_2 \, \Theta_2(2z) + \Theta_3 \, \Theta_3(2z) + \Theta_4 \, \Theta_4(2z) ,\\
 \theta_2^2(z)  & = \Theta_1 \, \Theta_1(2z) + \Theta_2 \, \Theta_2(2z) - \Theta_3 \, \Theta_3(2z) - \Theta_4 \, \Theta_4(2z) ,\\
  \theta_3^2(z)  & =  \Theta_1 \, \Theta_1(2z) - \Theta_2 \, \Theta_2(2z) - \Theta_3 \, \Theta_3(2z) + \Theta_4 \, \Theta_4(2z) ,\\
   \theta_4^2(z)  & =  \Theta_1 \, \Theta_1(2z) - \Theta_2 \, \Theta_2(2z) + \Theta_3 \, \Theta_3(2z) - \Theta_4 \, \Theta_4(2z) .
  \end{split} 
\end{equation}
Comparing Equations~(\ref{eq:LinearTransfo}) with Equations~(\ref{2isogTHETA}), the coordinates can be expressed in terms of Theta functions with non-vanishing elliptic arguments as
$$[w:x:y:z]=[ \Theta_1(2  z): \Theta_2(2  z): \Theta_3(2  z) : \Theta_4(2  z) ] .$$
\end{proof}

\par The map $\pi\colon \bP(w,x,y,z) \longrightarrow \bP(P,Q,R,S)$ with $P=w^2, \dots, S=z^2$ is $8:1$ outside the coordinate planes. The map $\pi$ induces a covering of a reducible octic surface in $\bP^3$ given by
\begin{equation}
\begin{split}
\left(\Phi' - 2 D wxyz\right) \, \left(\Phi' + 2 D wxyz\right) = 0 \;,
 \end{split} 
\end{equation}
with
\begin{gather*}
\Phi' = w^4 + x^4 + y^4 + z^4 - A \big( w^2 y^2 + y^2 z^2 \big) 
- B\big( w^2z^2 + x^2y^2 \big) - C  \big(w^2x^2+y^2z^2\big),
\end{gather*}
onto the G\"opel quartic in Equation~(\ref{Kummer-Quartic}) with $\alpha=A$, $\beta=B$, $C=\gamma$, and $D=\delta$. We can assume that the Hudson quartic and G\"opel quartic are the singular Kummer varieties associated with two principally polarized Abelian varieties, say $\AA$ and $\hat{\AA}$, respectively. We have the following:
\begin{prop}
\label{prop:RatMap}
The map $\pi\colon \bP(w,x,y,z) \longrightarrow \bP(P,Q,R,S)$ with $P=w^2, \dots, S=z^2$ restricted to the Hudson quartic onto the G\"opel quartic with $\alpha=A$, $\beta=B$, $C=\gamma$, and $D=\delta$, is induced by a $(2,2)$-isogeny $\Psi\colon \AA \longrightarrow \hat{\AA}$.
\end{prop}
\begin{proof}
The rational map $\pi$ is 4:1 from the Hudson quartic onto the G\"opel quartic. In fact, the map $\pi\colon \bP(w,x,y,z) \longrightarrow \bP(P,Q,R,S)$ maps the sixteen nodes on the G\"opel-Hudson quartic in Lemma~\ref{lem:16nodes} to four nodes on the G\"opel quartic in Lemma~\ref{lem:16nodesG}.
\end{proof}
\par Similarly, we have for the Rosenhain quartic~(\ref{Birkenhake-Lange-Quartic}):
\begin{prop}
\label{prop:ThetaImageR}
For the surface in $\bP^3$ given by Equation~~(\ref{Birkenhake-Lange-Quartic}), the coordinates are given by
\begin{equation} 
\label{KummerVariables_Rosenhain}
 [Y_0:Y_1:Y_2:Y_3]=\left[\theta_1(z)^2:\theta_2(z)^2:\theta_7(z)^2:\theta_{12}(z)^2\right] \;,
\end{equation}
and the parameters are
\begin{gather*}
a= \left( 2\,\Theta_1\Theta_4-2\,\Theta_2\Theta_3  \right)  \left( 2\,\Theta_1\Theta_4+2\,\Theta_2\Theta_3 \right)  \left( 2\,\Theta_1\Theta_3+2\,\Theta_2\Theta_4 \right), \\
b= \left( \Theta_1^2+\Theta_2^2-\Theta_3^2-\Theta_4^2 \right) \left( \Theta_1^2-\Theta_2^2+\Theta_3^2-\Theta_4^2 \right)  \left( 2\,\Theta_1\Theta_2-2\,\Theta_3\Theta_4 \right), \\
c= \left( \Theta_1^2-\Theta_2^2-\Theta_3^2+\Theta_4^2 \right) \left( \Theta_1^2+\Theta_2^2+\Theta_3^2+\Theta_4^2 \right)  \left( 2\,\Theta_1\Theta_2+2\,\Theta_3\Theta_4\right), \\
d^2 =256\,\Theta_1\Theta_2\Theta_4\Theta_3 \left( \Theta_1^2\Theta_4^2-\Theta_2^2\Theta_3^2 \right)  \left( \Theta_1^4-\Theta_2^4-\Theta_3^4+\Theta_4^4 \right) \\
+ 8 \left( \Theta_1^2+\Theta_4^2 \right) \left( \Theta_2^2+\Theta_3^2 \right) \left( \Theta_1^2+\Theta_2^2+\Theta_3^2+\Theta_4^2 \right)^2\left( \Theta_1^2-\Theta_2^2-\Theta_3^2+\Theta_4^2 \right)^2\\
+8 \left(\Theta_1^2-\Theta_4^2 \right) \left( \Theta_2^2-\Theta_3^2 \right)   \left( \Theta_1^2+\Theta_2^2-\Theta_3^2-\Theta_4^2 \right)^2 \left( \Theta_1^2-\Theta_2^2+\Theta_3^2-\Theta_4^2 \right)^2\\
-32\left( \Theta_1^2\Theta_2^2+\Theta_3^2\Theta_4^2 \right)  \left( \Theta_1^2+\Theta_2^2+\Theta_3^2+\Theta_4^2 \right)  \left( \Theta_1^2-\Theta_2^2+\Theta_3^2-\Theta_4^2 \right) \\
\times  \left( \Theta_1^2-\Theta_2^2-\Theta_3^2+\Theta_4^2 \right)  \left( \Theta_1^2+\Theta_2^2-\Theta_3^2-\Theta_4^2 \right).
\end{gather*}
\end{prop}

\par The Rosenhain roots of the genus 2 curve $\CC$ in Equation~(\ref{Eq:Rosenhain_b}) generate the rational function field $\Q(\lambda_1, \lambda_2, \lambda_3)$ of the moduli space $\mathbb{A}_2(2)$ of principally polarized Abelian surfaces $\AA = \mathrm{Jac}(\CC)$ with level-two structure. The moduli $A, B, C, D$ of the Hudson quartic~\eqref{Goepel-Quartic} are then given by the rational functions 
\begin{gather}
\nonumber
 A =  2 \, \frac{\lambda_1+1}{\lambda_1-1}, \quad
 B =  2 \, \frac{\lambda_1\lambda_2+\lambda_1\lambda_3-2\lambda_2\lambda_3-2\lambda_1+\lambda_2+\lambda_3}{(\lambda_2-\lambda_3)(\lambda_1-1)},  \quad
 C  = 2 \, \frac{\lambda_3+\lambda_2}{\lambda_3-\lambda_2}, \\ 
 \label{KummerParameter4}
 D = 4 \, \frac{\lambda_1-\lambda_2 \lambda_3}{(\lambda_2 - \lambda_3) (\lambda_1-1)} \;.
\end{gather}
Similarly, the Rosenhain roots of the genus 2 curve $\hat{\CC}$ in Equation~(\ref{Eq:Rosenhain2}) generate the rational function field $\Q(\Lambda_1, \Lambda_2, \Lambda_3)$ of the moduli space $\hat{\mathbb{A}}_2(2)$ of $(2,2)$-isogenous principally polarized Abelian surfaces with level-2 structure.
In turn, the moduli $\alpha, \beta, \gamma, \delta$ of the G\"opel quartic~\eqref{Kummer-Quartic} are given by the rational functions 
\begin{gather}
\nonumber
 \alpha =  2 \, \frac{\Lambda_1+1}{\Lambda_1-1},\;
 \beta =  2 \, \frac{\Lambda_1\Lambda_2+\Lambda_1\Lambda_3-2\Lambda_2\Lambda_3-2\Lambda_1+\Lambda_2+\Lambda_3}{(\Lambda_2-\Lambda_3)(\Lambda_1-1)}, \;
 \gamma = 2 \, \frac{\Lambda_3+\Lambda_2}{\Lambda_3-\Lambda_2}, \\
 \label{KummerParameter2}
\delta = \frac{4(\Lambda_1-\Lambda_2\Lambda_3)}{(\Lambda_1-1)(\Lambda_3-\Lambda_2)}.
\end{gather}
Recalling the statement of Proposition~\ref{compactifications}, we conclude the following:
\begin{thm}
\label{thm_Kummer_even}
For the Jacobian $\operatorname{Jac}(\CC)$ of a smooth genus-two curve~$\CC$ in Rosenhain normal form given by Equation~(\ref{Eq:Rosenhain_b}) we have the following:

\begin{enumerate}
\item The Hudson quartic~(\ref{Goepel-Quartic}) is the image of $[\Theta_{1}(2z): \Theta_{2}(2z): \Theta_{3}(2z): \Theta_{4}(2z)]$ in $\bP^3$ with moduli~(\ref{KummerParameter4}) defined over $\mathbb{A}_2(2)$. 

\item The G\"opel quartic~(\ref{Kummer-Quartic}) is the image of $[\theta^2_{1}(z): \theta^2_{2}(z): \theta^2_{3}(z): \theta^2_{4}(z)]$ in $\bP^3$ with moduli~(\ref{KummerParameter2}) defined over $\hat{\mathbb{A}}_2(2)$.

\item The Rosenhain quartic~(\ref{Birkenhake-Lange-Quartic}) is the image of $[\theta^2_1(z):\theta^2_2(z):\theta_7^2(z):\theta_{12}^2(z)]$ in $\bP^3$ with moduli defined over $\mathbb{A}_2(2, 4)$. 
\end{enumerate}
\end{thm}
\par Tetrahedra in $\mathbb{P}^3$ whose faces are tropes are called \emph{Rosenhain tetrahedra} if all vertices are nodes; they are called \emph{G\"opel tetrahedra} if none of the vertices are nodes.  It turns out that Rosenhain tetrahedra have an odd number of odd tropes, namely one or three, whereas G\"opel tetrahedra have an even number of odd tropes, namely zero or two. In \cite{MR4421430} it was shown that Thomae's formula provides an explicit relation between tropes and the Theta functions used in Theorem~\ref{thm_Kummer_even}. In particular, in \cite{MR4421430} it was shown that tropes can be chosen so that the following holds:

\begin{rem}
\label{lem:bijection_tropes_thetas2}
\begin{enumerate}
\item[] 
\item
The tropes of the G\"opel tetrahedron $\{\mathsf{T}_{256}, \mathsf{T}_{136}, \mathsf{T}_{356},\mathsf{T}_{126}\}$ can be chosen to coincide with the Theta functions $[\theta^2_{1}(z): \theta^2_{2}(z): \theta^2_{3}(z): \theta^2_{4}(z)]$. 

\item The tropes of the Rosenhain tetrahedron $\{\mathsf{T}_{256}, \mathsf{T}_{136}, \mathsf{T}_{246}, \mathsf{T}_{2}\}$, can be chosen to coincide with the Theta functions $[\theta^2_1(z):\theta^2_2(z):\theta_7^2(z):\theta_{12}^2(z)]$. 
\end{enumerate}
\end{rem}
\subsection{Shioda-Inose Surfaces}\label{sec:shioda-inose}
Given an Abelian surface $\AA = \Jac(\X)$, one can construct an associate \emph{Shioda-Inose surface} $\Y := \operatorname{SI}(\Jac(\X))$, which is a K3 surface that 
shares a similar Hodge structure with $\AA$, on its transcendental lattice. 
\par The surface, $\Y$, is obtained (see \cite{clingherdoran1}) as a geometric two-isogeny of $\Kum(\Jac(\X))$. Namely, one has a diagram as below, involving rational double-cover maps:
\begin{equation}
\begin{tikzcd}[column sep=scriptsize]
\Jac(\X)  \arrow[r, dashrightarrow, "\pi"]  &
\Kum(\Jac(\X)) \arrow[rrrr, dashrightarrow, bend left=10, "r"] & & & &
\Y=\operatorname{SI}(\Jac(\X)) \arrow[llll, dashrightarrow, bend left=10, "p"] 
\end{tikzcd}
\end{equation}
The Shioda-Inose terminology for $\Y$ is motivated by the existence of a Shioda-Inose structure on this surface. In turn, such a structure induces an isomorphism of integral Hodge structures between the transcendental lattices of $\Jac(\X)$ and $\Y$ (see \cite{shioda-inose}). More specifically, $\Y$ admits an involution fixing the holomorphic (2,0)-form, with quotient $\Kum(\Jac(\X))$, and the rational degree-two map $p \colon \Y \longrightarrow \Kum(\Jac(\X))$ determined a Hodge isometry between $T(\Y)(2)$ and $T(\Kum(\Jac(\X)))$; see \cite{2019-4}.

\begin{lem}
In the situation above the Shioda-Inose structure induces an isomorphism between the transcendental lattices $T(\Jac(\X))$ and $T(\Y)$, up to scaling.
\end{lem}

\begin{proof}
The transcendental lattice $T(\Jac(\X))$ is the orthogonal complement of the N\'eron-Severi group in the cohomology lattice $H^2(\Jac(\X), \Z)$. The map $\pi\colon \Jac(\X) \longrightarrow \Kum(\Jac(\X))$ contracts the 2-torsion cycles, and the map $p\colon \Y \to \Kum(\Jac(\X))$ doubles the lattice, with the involution on $\Y$ preserving the (2,0)-form. The Hodge structure on $T(\Y)$ aligns with $T(\Jac(\X))$, and the scaling by 2 in $T(\Y)(2)$ matches the degree of the cover, ensuring an isometry of integral lattices.
\end{proof}

The Shioda-Inose surfaces can be constructed explicitly as minimal resolutions of some special projective quartic surfaces. Namely, consider the following family of quartics in $\mathbb{P}^3$:
\begin{equation}
\label{shiodain}
 y^2zw - 4x^3z+3 \alpha xzw^2 + \beta zw^3 + \gamma xz^2w - 
 \frac{1}{2}(\delta z^2w^2 +w^4)=0.   
\end{equation}
Then, under the condition $(\gamma, \delta) \neq (0,0)$, the minimal resolution of the above quartic is a K3 surface $\operatorname{SI}(\Jac(\X))$ of Shioda-Inose type. Moreover, the connection with the associated genus-two curve can be determined via the fact that the coefficients $(\alpha, \beta, \gamma, \delta)$ can be computed in terms of genus-two Siegel modular forms. Namely, one has (see \cite{clingherdoran1}):
\begin{equation}
    [\alpha: \beta: \gamma: \delta ] \ = \ \left [ \psi_4: \ \psi_6: \ 
    2^{12} 3^5 \chi_{10} : \ 2^{12} 3^6 \chi_{12} \right ]
\end{equation}
The right-hand term above involves the Igusa modular forms of Section \ref{ssec:ppas} and the identity should be seen as an equality of points in the weighted-projective space $\mathbb{P}_{(2,3,5,6)}$. From here, one can compute directly the Igusa invariants of the associated genus-two curve as:
\begin{equation}
[J_2: J_4 : J_6: J_{10}] \ =  \ 
\left [ 2^3 3 \delta: \ 2^23^2 \alpha \gamma ^2: 
\ 2^33^2(4 \alpha \delta  + \beta \gamma )\gamma^2: 2^2 \gamma^6 \right ] .
\end{equation} 
\par It is also important to note the connection between the quartic form (\ref{shiodain}) and the Hudson quartic (\ref{Goepel-Quartic}) which describes $\Kum(\Jac(\X))$. The Kummer surface carries a special Jacobian elliptic fibration $\pi \colon \Kum(\Jac(\X)) \to \bP^1$ which, in the generic case, has a Mordell-Weil isomorphic to $\mathbb{Z} / 2 \mathbb{Z}$ and a configuration of singular fibers consisting of a fiber of Kodaira type $I_5^*$, five fibers of type $I_2$ and one fiber of type $I_1$. The generator of the Mordell-Weil group is associated with a section of order two. Deriving from the Hudson quartic description, the fibration $\pi$ can also be described explicitly via an equation of type:
\begin{equation}
\label{siform1}
Y^2 = X \left( X^2 + \mathcal{P}(t) X + \mathcal{Q}(t) \right ) 
\end{equation}
where $\mathcal{P}(t)$ and $\mathcal{Q}(t)$ are polynomials of degrees 4 and 6 with coefficients involving the Hudson quartic parameters $A.B,C,D$ (see \cite{clingherdoran1} for the explicit formulas). The section of order two may be seen in this model as $X=0$. Translations by this order-two section, within the smooth fibers, extend to a global Nikulin involution on $\Kum(\Jac(\X))$. After taking the quotient by this involution and resolving the eight occurring singularities, one obtains a new K3 surface - the Shioda-Inose surface $\operatorname{SI}(\Jac(\X))$. The quotient map $r \colon \Kum(\Jac(\X)) \dashrightarrow \operatorname{SI}(\Jac(\X))$ may be seen in the context of $(\ref{siform1})$ as:
\begin{equation}
    (X,Y) \mapsto (x,y) = \left ( \frac{\mathcal{Q}(t)}{X}, \ - \frac{\mathcal{Q}(t) Y}{X^2}\right ) \ , 
\end{equation}
with the two-isogenous elliptic fibration on $\operatorname{SI}(\Jac(\X))$ being given by:
\begin{equation}
\label{SIexplicit}
y^2 = X \left( x^2 + \tilde{\mathcal{P}}(t) x + \tilde{\mathcal{Q}}(t) \right ) \ ,
\end{equation}
with $\tilde{\mathcal{P}}(t) = -2 \mathcal{P}(t) $ and  $\tilde{\mathcal{Q}}(t) = \mathcal{P}(t)^2 - 4 \mathcal{Q}(t) $. This dual elliptic fibration $(\ref{SIexplicit})$, whose underlying surface is the Shioda-Inose surface  $\operatorname{SI}(\Jac(\X))$, also has Mordell-Weil group isomorphic to  $\mathbb{Z} / 2 \mathbb{Z}$ and carries a singular fiber of Kodaira type $I_{10}^*$.  As it turns out, $(\ref{SIexplicit})$ matches the elliptic fibration obtained by projection to $[x\colon w]$ in the context of the Shioda-Inose quartic $(\ref{shiodain})$. This allows one to obtain explicit formulas relating the parameters $(\alpha, \beta, \gamma, \delta)$ to the Hudson coefficients $A.B,C,D$. We refer  to \cite{clingherdoran1} for the actual formulas.  
\section{\texorpdfstring{$(n, n)$}{(n,n)}-Split Jacobians}
\label{sec:split_jac}
We have seen above that sets of Abelian surfaces with the same endomorphism ring form sub-varieties within $\mathbb{A}_2$.  In fact, we showed that the endomorphism ring of principally polarized Abelian surface tensored with $\Q$ is either a quartic CM field, an indefinite quaternion algebra, a real quadratic field or in the generic case $\Q$. Irreducible components of the corresponding  subsets in $\mathbb{A}_2$ have dimensions $0, 1, 2$ and are known as \emph{CM points}, \emph{Shimura curves}, and \emph{Humbert surfaces}, respectively. The latter can be related to a fascinating aspect of genus 2 Jacobians, namely their potential \emph{decomposability}. 
\par Let $\psi \colon \CC \longrightarrow E_1$ be a maximal degree $n$ covering to an elliptic curve $E_1$, meaning $\deg \psi = n$ and $\psi$ does not factor through an isogeny of $E_1$. Then, there exists another elliptic curve $E_2 := \Jac(\CC) / E_1$, defined as the quotient by the connected component of $\ker(\psi_*)$, such that $\Jac(\CC)$ is isogenous to $E_1 \times E_2$ via an isogeny of degree $n^2$. We then call $\Jac(\CC)$ \textbf{$(n, n)$-decomposable} or \textbf{$(n, n)$-Split}, a property studied in \cite{FK1}. The locus of such curves in $\M_2$ forms a 2-dimensional irreducible subvariety, with explicit computations for $n = 2, 3, 5$ given in \cite{2000-2}, \cite{2001-1}, and \cite{2005-1}, respectively.

Consider an irreducible, smooth, projective curve $\CC$ of genus 2 and a maximal covering $\psi_1 \colon  \CC \longrightarrow E_1$ of degree $n$. The induced map $\psi_1^* \colon E_1 \longrightarrow \Jac(\CC)$ is injective, embedding $E_1$ as a subvariety, and $\psi_{1,*} \colon \Jac(\CC) \longrightarrow E_1$ has kernel $\ker(\psi_{1,*})$, an elliptic curve $E_2$ since $\dim \Jac(\CC) = 2$ and $\dim E_1 = 1$ (see \cite{2000-1}). Fixing a Weierstrass point $P \in \CC$, the embedding
\begin{equation}
\begin{aligned}
i_P \colon \CC & \longrightarrow \Jac(\CC) \\
x & \mapsto \ [(x) - (P)]
\end{aligned}
\end{equation}
maps $\CC$ into $\Jac(\CC)$. Let $g\colon E_2 \longrightarrow \Jac(\CC)$ be the natural inclusion, with dual $g^* \colon \Jac(\CC) \longrightarrow E_2$. Define $\psi_2 = g^* \circ i_P \colon \CC \longrightarrow E_2$, a morphism to $E_2$. This yields exact sequences:
\begin{equation}\label{red-Jacobians}
0 \longrightarrow E_2 \overset{g}{\longrightarrow} \Jac(\CC) \xrightarrow{\psi_{1,*}} E_1 \to 0,
\end{equation}
and its dual
\begin{equation}
0 \longrightarrow E_1 \xrightarrow{\psi_1^*} \Jac(\CC) \overset{g^*}{\longrightarrow} E_2 \longrightarrow 0.
\end{equation}
If $\deg(\psi_1) = 2$ or odd, $\psi_2 \colon \CC \longrightarrow E_2$ is unique up to elliptic curve isomorphism, as shown in \cite{2001-1}. The Hurwitz space $\H_\sigma$ of such covers embeds as a 2-dimensional subvariety $\L_n \subset \M_2$, with equations in terms of $J_2, J_4, J_6, J_{10}$ given in \cite{2000-2} (for $n = 2$), \cite{2001-1} (for $n = 3$), and \cite{2005-1} (for $n = 5$). We say $\CC$ has an $(n, n)$-decomposable Jacobian if $\Jac(\CC)$ admits such a structure, with $E_1$ and $E_2$ as its components.
\subsection{Humbert surfaces}
The Humbert surface $H_{\Delta}$ with invariant $\Delta$ is the space of principally polarized Abelian surfaces admitting a symmetric endomorphism with discriminant $\Delta$. It turns out that $\Delta$ always is a positive integer satisfying $\Delta \equiv 0, 1\mod{4}$ and uniquely determined $H_{\Delta}$. In fact, $H_{\Delta}$ is the image inside $\mathbb{A}_2$ under the projection of the rational divisor associated with the equation
\begin{equation}
\label{eqn:discriminant}
 a \, \tau_{11} + b \, \tau_{12} + c \, \tau_{22} + d\, (\tau_{12}^2 -\tau_{11} \, \tau_{22}) + e = 0 \;,
\end{equation}
with integers $a, b, c, d, e$ satisfying $\Delta=b^2-4\,a\,c-4\,d\,e$ and $\tau = \bigl(\begin{smallmatrix} \tau_{11}& \tau_{12}\\ \tau_{12} & \tau_{22} \end{smallmatrix} \bigr) \in \mathbb{H}_2$.  The following was proven by Birkenhake and Lange in \cite{MR1953527}:
\begin{thm}
\label{prop:isogeny_Delta}
For $n \in \mathbb{N}$ the Humbert surface $H_{n^2}$ is the locus of principally polarized Abelian surfaces $(\AA, \mathscr{L}) \in \mathbb{A}_2$ admitting an isogeny of degree $n^2$, given by
\begin{equation}
\label{eqn:Phi}
 \hat{\Psi}\colon \quad \Big( E_1 \times E_2, \;  \mathcal{O}_{E_1} (n) \boxtimes  \mathcal{O}_{E_2} (n)  \Big)
 \ \longrightarrow \ \Big( \AA, \mathscr{L} \Big) \,,
\end{equation} 
where $\mathcal{O}_{E_l} (n)$ is a line bundle of degree $n$ on an elliptic curve $\mathcal{E}_l$ for $l= 1,2$.
\end{thm}
For example, inside of $\mathbb{A}_2$ sit the Humbert surfaces $H_1$ and $H_4$ that are defined as the images under the projection  of the rational divisor associated to $\tau_{12}=0$ and $\tau_{11} - \tau_{22}=0$, respectively. In fact, the singular locus of $\mathbb{A}_2$ has $H_1$ and $H_4$ as its two connected components.  As analytic spaces, $H_1$ and $H_4$ are each isomorphic to the Hilbert modular surface 
\begin{equation}
\label{modular_product2}
 \Big( (\mathrm{SL}_2(\Z) \times \mathrm{SL}_2(\Z) ) \rtimes \Z_2 \Big) \backslash \Big( \mathbb{H} \times \mathbb{H} \Big) \,.
\end{equation}
In terms of Siegel modular forms in Section~\ref{ssec:ppas} we have the following characterization, see \cite{MR1438983}:
\begin{prop}
\label{prop:Q}
The vanishing divisor of the cusp form $\chi_{10}$ in $\mathbb{A}_2$ is the Humbert surface $H_1$, i.e., a period point $\tau$ is equivalent to a point with $\tau_{12}=0$ relative to $\Gamma_2$ if and only if $\chi_{10}(\tau)=0$.  The vanishing divisor of $Q$ in $\mathbb{A}_2$ is the Humbert surface $H_4$, i.e.,  a period point $\tau$ is equivalent to a point with $\tau_{11}=\tau_{22}$ relative to $\Gamma_2$ if and only if $Q=0$. 
\end{prop}
It is known that one has $\chi_{10}(\tau)=0$ if and only if the principally polarized Abelian surface $\AA$ is a product of two elliptic curves  $\AA  = E_{\tau_{11}} \times E_{\tau_{22}}$ with the transcendental lattice  $\mathrm{T}_\AA = H \oplus H$; see \cite{MR3712162}.  Here, $H$ denotes the lattice $\Z^2$ with quadratic form $q(\vec{v}) = 2 v_1 v_2$. Moreover, for $Q(\tau)=0$ the transcendental lattice of the corresponding Abelian surface $\AA$ is given by  $\mathrm{T}_\AA = H \oplus \langle 2 \rangle \oplus \langle -2 \rangle$. Here, $\langle m \rangle$ denotes the rank-one lattice $\Z v$ with $q(v)=m$. Similarly, for $\tau \in H_{n^2}$, the transcendental lattice of $\AA$ is given by  $\mathrm{T}_\AA = H \oplus \langle n \rangle \oplus \langle -n \rangle$.

Considering $\AA=\operatorname{Jac}(\CC)$ as above, the isogeny in the above theorem is precisely $\hat{\Psi} = \psi_1^* \times \psi_2^*$.  The Humbert hypersurface $H_{\Delta}$ then parameterizes curves $\CC$ whose Jacobians admit an optimal action by the order $\O_\Delta$, a condition tied to embeddings of quadratic fields (see \cite{HM95}). In particular, $H_{n^2}$ correspond to curves with $(n, n)$-Split Jacobians, reflecting isogenies to products of elliptic curves. A point in $H_{m^2} \cap H_{n^2}$ ($m \neq n$) indicates either a simple Abelian surface with quaternionic multiplication by an indefinite quaternion algebra over $\Q$, or a self-product $E^2$ where $E$ is an elliptic curve, a phenomenon prominent on Shimura curves.

\begin{prop}
$\Jac(\CC)$ is a geometrically simple Abelian variety if and only if it is not $(n, n)$-decomposable for some $n > 1$. Equivalently, if $\Jac(\CC)$ is split over $k$, then there exists an integer $n \geq 2$ such that $\Jac(\CC)$ is $(n, n)$-Split.
\end{prop}

\begin{proof}
Suppose $\Jac(\CC)$ is geometrically simple, i.e., simple over $\bar{k}$. By the Poincaré-Weil theorem, $\Jac(\CC)$ is isogenous to $\AA_1^{n_1} \times \cdots \times \AA_r^{n_r}$, and simplicity over $\bar{k}$ implies $r = 1$, $n_1 = 1$, with $\AA_1 = \Jac(\CC)$. If $\Jac(\CC)$ were $(n, n)$-decomposable, there would exist a maximal degree $n$ covering $\psi \colon \CC \longrightarrow E_1$, inducing an isogeny $\Jac(\CC) \longrightarrow E_1 \times E_2$ of degree $n^2$, where $E_1, E_2$ are elliptic curves. Over $\bar{k}$, this isogeny splits $\Jac(\CC)$ into a product of 1-dimensional varieties, contradicting simplicity unless $n = 1$, which is trivial (as $\deg \psi = 1$ implies $\CC \cong E_1$, contradicting $g = 2$). Thus, $\Jac(\CC)$ is not $(n, n)$-decomposable for any $n > 1$.

Conversely, if $\Jac(\CC)$ is not $(n, n)$-decomposable for any $n > 1$, suppose it is not geometrically simple. Then over $\bar{k}$, $\Jac(\CC) \cong E_1 \times E_2$, with $\dim E_i = 1$. By the theory of maximal coverings (\cite{FK1}), there exists a degree $n > 1$ map $\psi : \CC \to E_1$ (e.g., projection via a correspondence), making $\Jac(\CC)$ isogenous to $E_1 \times E_2$, hence $(n, n)$-Split, a contradiction. Thus, $\Jac(\CC)$ must be simple over $\bar{k}$.

For the equivalent statement, if $\Jac(\CC)$ is split over $k$ (isogenous to $E_1 \times E_2$ over $k$), there exists a maximal covering $\psi \colon \CC \longrightarrow E_1$ of degree $n \geq 2$, as genus 2 curves admit non-trivial maps to elliptic curves, inducing the $(n, n)$-Split structure (see \cite{2000-1}, §3).
\end{proof}

This characterization connects the geometric simplicity of $\Jac(\CC)$ to its indecomposability, a key property for later isogeny studies. In terms of the Kummer plane introduced in Section~\ref{ssec:Kummer_odd} we have the following:
\begin{rem}
\label{rem:config}
In \cite{MR1953527} the geometry of the Kummer plane was determined over several Humbert surfaces $H_\Delta$. In particular, the following statement were proven~\cite{MR1953527}*{Cor.~7.2} and~\cite{MR1953527}*{Cor.~7.3}:
\begin{enumerate}
\item[$\Delta=\phantom{1}4\colon$] if $(\AA, \mathscr{L}) \in H_4$ if and only if (numbering the six lines on its Kummer plane $(\mathbb{P}^2; \mathsf{T}_1, \dots, \mathsf{T}_6)$ suitably) the three points $\mathsf{T}_1 \cap \mathsf{T}_2$, $\mathsf{T}_3 \cap \mathsf{T}_4$, $\mathsf{T}_5 \cap \mathsf{T}_6$ are collinear.
\item[$\Delta=16\colon$] if $(\AA, \mathscr{L}) \in H_{16}$ then its Kummer plane $(\mathbb{P}^2; \mathsf{T}_1, \dots, \mathsf{T}_6)$ admits a cubic passing smoothly through three of the 15 points $\mathsf{T}_m \cap \mathsf{T}_n$ and touching the singular lines $\mathsf{T}_n$ in the remaining intersection points with even multiplicity. Conversely, if $(\mathbb{P}^2; \mathsf{T}_1, \dots, \mathsf{T}_6)$ admits such a curve, then $(\AA, \mathscr{L}) \in H_\Delta$ with $\Delta \in \{ 4, 8, 12, 16, 20\}$.
\end{enumerate} 
For the transcendental lattice it follows $\mathrm{T}_{\operatorname{Kum}(\AA)} \cong \mathrm{T}_{\AA}(2)$ \cite{MR728142}*{Thm.~10}. In particular we have
\begin{equation}
\begin{split}
 \mathrm{T}_{\operatorname{Kum}(\AA)}  & = H(2) \oplus \langle 4 \rangle \oplus \langle -4 \rangle \ \text{for} \ (\AA, \mathscr{L}) \in H_4\,.
\end{split} 
\end{equation}
Generally, for $\AA \in H_{n^2}$ it follows from \cite{MR728142} that $\operatorname{Kum}(\AA)$ has Picard rank 18, and the transcendental lattice is given by
\begin{equation}
  \operatorname{T}_{\operatorname{Kum}(\AA)} \cong \operatorname{T}_\AA(2) = H(2) \oplus \langle 2n \rangle \oplus   \langle - 2n \rangle\,.
\end{equation}
\end{rem}
\subsection{The case of \texorpdfstring{$(2,2)$}{(2,2)}-Split Jacobians}
\label{ssec:22case}
The case of a smooth genus 2 curve $\CC$ with a \((2, 2)\)-Split Jacobian was described by Bolza \cite{MR1505464} explicitly. In particular,  he proved that the Jacobian $\operatorname{Jac}(\CC)$ for a smooth genus-two curve $\CC$ is \((2, 2)\)-Split if and only if $Q(\tau)=0$; see Proposition~\ref{prop:Q}. Next we will provide several explicit models for the algebraic curves involved in this case. 
\subsubsection{Bolza's representation}
Bolza also proved that one can always represent this extra involution as $[X:Y:Z] \mapsto [-X:Y:Z]$ if the genus 2 curve is given in the from
\begin{equation}
\label{eqn:genus2+auto}
  \CC\colon \quad Y^2 = X^6 + s_1 X^4 Z^2 + s_2 X^2 Z^4 + Z^6 \,.
\end{equation} 
One uses Equations~(\ref{invariants}) and Equation~(\ref{chi_35sqr}) to check that $Q= 2^{12} \, 3^9 \, \chi_{35}^2 /\chi_{10}$ always vanishes for such a genus-two curve.  
\par Given the elliptic involution, its composition with the hyperelliptic involution defines a second elliptic involution. The two involutions define two elliptic subfields of degree two for the function field of $\CC$. We introduce the elliptic curves $E_l$ in $\bP^2 = \bP(x_l, y_l, z_l)$ for $l=1, 2$, given by
\begin{equation}
\label{eqn:elliptic_curves}
 E_1\colon \ y_1^2 z_1 = x_1^3 + s_2 x_1^2 z_1 + s_1 x_1 z_1^2 +  z_1^3 \,, \quad
 E_2\colon \ y_2^2 z_2 = x_2^3 + s_1 x_2^2 z_2 + s_2 x_2 z_2^2 +  z_2^3 \,,
\end{equation}
where $E_1$ and $E_2$ have the j-invariants $j_1 = j(E_1)$ and $j_2 = j(E_2)$ with
\begin{equation}
\label{eqn:j_invs}
 j_1 =\frac{2^8 \left(3 s_1 - s_2^2\right)^3}{4(s_1^3+s_2^3)-(s_1s_2)^2-18 s_1s_2+27}  \,, \quad
 j_2 = \frac{2^8 \left(3 s_2 - s_1^2\right)^3}{4(s_1^3+s_2^3)-(s_1s_2)^2-18 s_1s_2+27} \,,
\end{equation}
respectively. Here, we use the standard normalization of the j-invariant  where the square torus with the complex structure $i$ satisfies $j=1728=12^3$.  The degree-two quotient maps $\psi_l\colon  \CC \longrightarrow E_l$ associated with the involutions are given by
\begin{equation}
 \psi_1 \colon \quad   \CC \ \longrightarrow \  E_1\,, \qquad \Big[ X : Y : Z \Big]   \ \mapsto \ \Big[ x_1: y_1 : z_1 \Big]  =  \Big[ XZ^2: Y : X^3 \Big] \,,
\end{equation}
and
\begin{equation}
 \psi_2 \colon \quad   \CC \ \longrightarrow \  E_2 \,, \qquad \Big[ X : Y : Z \Big]   \ \mapsto \ \Big[ x_2: y_2 : z_2 \Big] =  \Big[ X^2Z: Y : Z^3 \Big] \,,
\end{equation}
respectively, for $XZ \not =0$.
\subsubsection{Pringsheim's representation}
One can also start with a smooth  genus-two curve $\CC$ in Rosenhain form given by Equation~(\ref{Eq:Rosenhain}). We use the aforementioned relations between the Igusa invariants and  the Siegel modular forms to expand the generators of the ring of modular forms in terms of the Rosenhain roots. We obtain
\begin{equation}
\label{eqn:symmetry}
\begin{array}{rl}
\multicolumn{2}{l}{- \dfrac{2^{22} \chi_{35}^2(\tau)}{\chi_{10}^4(\tau)}  =  	{\color{black}\big(\lambda_1 - \lambda_2 \lambda_3\big)^2 } 
		{ \big(\lambda_2 - \lambda_1 \lambda_3\big)^2}  
		{ \big(\lambda_3 - \lambda_1 \lambda_2\big)^2} }\\[4pt]
\times & 	{\big(\lambda_1 - \lambda_2 - \lambda_3 + \lambda_2 \lambda_3\big)^2}
		{ \big(-\lambda_1 + \lambda_2 - \lambda_3 + \lambda_1 \lambda_3\big)^2} 
		{ \big(-\lambda_1 - \lambda_2 + \lambda_3 + \lambda_1 \lambda_2\big)^2} \\[4pt]
\times & 	{\big(\lambda_1\lambda_2 + \lambda_1\lambda_3  - \lambda_2 \lambda_3 - \lambda_1\big)^2 }
		{\big(\lambda_1\lambda_2 + \lambda_2\lambda_3 - \lambda_1 \lambda_3 - \lambda_2\big)^2} 
		{\big(\lambda_1\lambda_3 + \lambda_2\lambda_3 - \lambda_1 \lambda_2 - \lambda_3\big)^2} \\[4pt]
\times & \left\lbrace \begin{array}{l} 
		{\big(\lambda_1\lambda_2 - \lambda_1\lambda_3  -\lambda_1 + \lambda_3\big)^2} 
		{\big(\lambda_1\lambda_3 - \lambda_2\lambda_3 - \lambda_1 + \lambda_3\big)^2} 
		{\big(\lambda_1\lambda_2 - \lambda_2\lambda_3 - \lambda_1 + \lambda_2\big)^2}  \\[2pt]
		{\big(\lambda_1\lambda_2 - \lambda_1\lambda_3 + \lambda_1 - \lambda_2\big)^2} 
		{\big(\lambda_1\lambda_3 - \lambda_2\lambda_3 + \lambda_2 - \lambda_3\big)^2} 
		{\big(\lambda_1\lambda_2 - \lambda_2\lambda_3 - \lambda_2 + \lambda_3\big)^2}  \,.
\end{array} \right.
\end{array}
\end{equation}
The vanishing divisor of Equation~(\ref{eqn:symmetry}) defines fifteen components in $\mathbb{H}_2 / \Gamma_2(2)$ of discriminant $\Delta =4$. Notice that each line in Equation~(\ref{eqn:symmetry}) is arranged to be invariant under permutations of the three roots $\lambda_1, \lambda_2, \lambda_3$. Pringsheim proved the following statement\footnote{We corrected two minor typos in the statement of the main theorem.} in \cite{MR1509868}:
\begin{prop}[Pringsheim]
\label{prop:pringsheim}
There are exactly 15 components in $\mathbb{H}_2 / \Gamma_2(2)$ covering $H_4$. Each of the component is equivalent to $\tau_{11} = \tau_{22}$  relative to the modular group $\Gamma_2$. Moreover, there is a transposition of the six roots $(\lambda_1, \lambda_2, \lambda_3, 0, 1,\infty)$ in  $\Gamma_2/\Gamma_2(2)\cong S_6$ that permutes each pair of components. 
\end{prop}
We will use component  $\lambda_1=\lambda_2\lambda_3$ in Proposition~\ref{prop:pringsheim} to construct a smooth genus-two curve in Rosenhain normal form admitting an elliptic involution.  That is, we consider the smooth genus-two curve $\CC$, given by
\begin{equation}
\label{Eq:Rosenhain_special}
   \CC\colon \quad Y^2 = X Z \, \big(X-Z\big) \, \big( X- \lambda_2 \lambda_3 Z\big) \,  \big( X- \lambda_2  Z\big) \,  \big( X - \lambda_3  Z\big) \,,
\end{equation}
with a discriminant given by
\begin{equation}
  \lambda_2^6  \lambda_3^6 (\lambda_2-1)^4  (\lambda_3-1)^4  (\lambda_2\lambda_3-1)^4   (\lambda_2 - \lambda_3)^2 \,.
\end{equation}
Let us  denote by $\Lambda_l \in \bP^1 \backslash \lbrace 0, 1, \infty \rbrace$ the modular parameter for the elliptic curves $E_l$ in $\bP^2 = \bP(x_l, y_l, z_l)$ for $l=1, 2$ defined by the \emph{Legendre normal form}
\begin{equation}
\label{eqn:EC}
 E_l\colon \quad y_l^2 z_l = x_l \Big( x_l -  z_l\Big) \Big( x_l - \Lambda_l z_l\Big)  \,,
\end{equation}
with the hyperelliptic involution given by $\imath_l: [x_l : y_l : z_l] \mapsto [x_l : -y_l : z_l]$.  One easily checks the following:
\begin{lem}
\label{lem:ECR}
The function field of the smooth genus-two curve $\CC$ contains the subfields given by the function fields of the elliptic curves $\mathcal{E}_l$ for $l=1, 2$ if
\begin{equation}
\label{eqn:EC_12_j_invariants}
 \Lambda_1 \Lambda_2 = \dfrac{(\lambda_2 +\lambda_3)^2 -4 \lambda_2\lambda_3}{(1-\lambda_2)^2 (1-\lambda_3)^2} \,, \qquad
 \Lambda_1 + \Lambda_2 = - \dfrac{2(\lambda_2 +\lambda_3)}{(1-\lambda_2) (1-\lambda_3)} \,.
\end{equation}
\end{lem}
\par We then have the following:
\begin{prop}
\label{prop:quotient_maps}
Assume that $\lambda_2, \lambda_3 \in \bP^1 \backslash \{ 0, 1, \infty\}$ satisfy $\lambda_2 \not= \lambda_3^{\pm1}$, and the moduli of the curves $\CC$ and $E_l$ in Equations~(\ref{Eq:Rosenhain_special}) and~(\ref{eqn:EC}) satisfy Equations~(\ref{eqn:EC_12_j_invariants}). Quotient maps $\psi_l\colon  \CC \longrightarrow E_l$ for $l= 1,2$ are given by
\begin{equation}
\label{eqn:pi_l}
 \psi_l \colon \quad   \CC \ \longrightarrow  \  E_l \,, \qquad \Big[ X : Y : Z \Big]   \mapsto \Big[ x_l: y_l : z_l \Big]  \; \; \text{for $l= 1,2$,}
\end{equation}
with
\begin{equation}
\Big[ x_l: y_l : z_l \Big] \ = \ \Big[ r \big(X-\lambda_2 Z \big)\big(X-\lambda_3 Z \big) \, XZ:  \ \big(X - (-1)^l q Z\big) \, Y  :\ r^3 X^2Z^2 \Big] \,,
\end{equation}
for $XZ \not =0$, and $[ x_l: y_l : z_l] = [1:0:0]$ otherwise. Here, $q, r$ are square roots of $q^2=\lambda_2\lambda_3$ and $r^2 = (1-\lambda_2)(1-\lambda_3)$, respectively. The elliptic involutions $\jmath_l$ on $\CC$, given by
\begin{equation}
\label{eqn:involutions}
\jmath_l\colon \quad \Big[ X : Y : Z \Big]  \mapsto \Big[ \lambda_2\lambda_3 Z : \ (-1)^{l+1} \lambda_2\lambda_3 q Y : \ X \Big] \,,
\end{equation}
satisfy $\psi_l \circ \jmath_l = \psi_l$ for $l=1, 2$.
\end{prop}
\begin{proof}
The genus-two curve $\CC$ in Equation~(\ref{Eq:Rosenhain_special}) is assumed to be smooth. Thus, one must have $\lambda_2, \lambda_3 \in \bP^1 \backslash \{ 0, 1, \infty\}$ and $\lambda_2 \not= \lambda_3^{\pm1}$.  The remainder of the proof follows by explicit computation. We note that a choice of square root for $r$ is changed by composing $\psi_l$ with the elliptic involution on $E_l$. The choice of square root for $q$ is changed by interchanging $\psi_1 \leftrightarrow \psi_2$.
\end{proof}
\begin{rem}
\label{rem:field_extensions}
Using a covering space of c$\lambda_1=\lambda_2\lambda_3$ in Proposition~\ref{prop:pringsheim}, given by the set of tuples $(k_1,k_2)$ with $\lambda_l= k_l^2$ for $l=2,3$ and $\lambda_1 = (k_2 k_3)^2$, we obtain for the modular parameters of the elliptic-curve quotients $E_l$  in Lemma~\ref{lem:ECR} the following algebraic solutions of Equation~(\ref{eqn:EC_12_j_invariants}):
\begin{equation}
\label{eqn:moduli}
 \Lambda_1 =  - \dfrac{(k_2 - k_3)^2}{(1- k_2^2) (1-k_3^2)} \,, \qquad
 \Lambda_2 =  - \dfrac{(k_2 + k_3)^2}{(1- k_2^2) (1-k_3^2)}  \,.
\end{equation} 
A change of square roots $(k_2, k_3) \mapsto (\pm k_2, \pm k_3)$ leaves Equations~(\ref{eqn:moduli}) invariant, whereas the change $(k_2, k_3) \mapsto (\pm k_2, \mp k_3)$ interchanges $\Lambda_1$ and $\Lambda_2$.  This means that the elliptic curves $E_l$ with level-two structure can be constructed over $\Q(k_2, k_3)$ as a finite field extension of  $\Q(\Lambda_1, \Lambda_2)$.
\end{rem}
\subsubsection{Legendre's gluing construction}
\label{sssec:Legendre}
The Weierstrass points $P_i$ of the curve $\CC$ in Equation~(\ref{Eq:Rosenhain_special}), given by 
\begin{equation}
\begin{split}
 P_1 :  [\lambda_2 \lambda_3: 0 :1] \,, \quad  P_2 :  & \, [\lambda_2 : 0 :1] \,,\quad P_3 :  [\lambda_3: 0 :1] \,, \\
 P_4 :  [0: 0 :1] \,,  \quad P_5 : & \, [1: 0 :1] \,,  \quad P_6 :  [1: 0 :0] \,,
 \end{split}
 \end{equation}
are the fixed points of the hyperelliptic involution. The two elliptic involutions $\jmath_l$ in Equation~(\ref{eqn:involutions}) for $l=1,2$ (each) pairwise interchange the Weierstrass points, i.e., 
 \begin{equation}
  \jmath_l\colon \quad P_2 \leftrightarrow P_3 \,, \quad   P_1 \leftrightarrow P_5 \,, \quad  P_4 \leftrightarrow P_6 \,.
\end{equation}  
The points $Q_{l,\pm}$, given by
\begin{equation}
\begin{split}
 Q_{1,\pm} \colon&  \quad [ -q: \pm i q^{\frac{3}{2}} (k_2 k_3+1) (k_2+k_3) : 1 ] \,, \\
 Q_{2,\pm} \colon&  \quad [\phantom{-} q:  \pm \phantom{i} q^{\frac{3}{2}} (k_2 k_3-1) (k_2-k_3): 1 ] \,,
 \end{split}
\end{equation}  
are the ramification points of $\psi_l$ in Equation~(\ref{eqn:pi_l}) for $l= 1$ and $l=2$, respectively.  The involution $\jmath_1$ fixes the points $Q_{1,+}$ and $Q_{1,-}$, and interchanges the points $Q_{2,+} \leftrightarrow Q_{2,-}$. An analogous statements holds for the involution $\jmath_2$.  The hyperelliptic involution of the curve $\CC$ interchanges the elements of the two pairs simultaneously, i.e.,  $Q_{l,+} \leftrightarrow Q_{l,-}$ for $l=1, 2$. 
\par The pairs $\{ P_2, P_3 \}$, $\{ P_1, P_5 \}$, $\{Q_{2,+}, Q_{2,-}\}$, and $\{ P_5, P_6 \}$ are mapped by $\psi_1$ to the two-torsion points $p_0: [x_1 : y_1: z_1] = [0:0:1]$,  $p_1: [1:0:1]$, $p_{\Lambda_1}: [\Lambda_1:0:1]$, and the identity $p_\infty: [1:0:0]$ in $E_1$. Similarly, the pairs $\{ P_2, P_3 \}$, $\{ P_1, P_5 \}$, $\{Q_{1,+}, Q_{1,-}\}$, and $\{ P_5, P_6 \}$ are mapped by $\psi_2$ to the two-torsion points $p_0: [x_2 : y_2: z_2] = [0:0:1]$,  $p_1: [1:0:1]$, $p_{\Lambda_2}: [\Lambda_2:0:1]$, and the identity $p_\infty:[1:0:0]$ in $E_2$. 
\par We associate with the branch locus of $\psi_1$ and $\psi_2$ the effective divisor of degree two $B_1 = [ \psi_1( Q_{1,+}) + \psi_1( Q_{1,-}) ]$ on $E_1$ and  $B_2  = [  \psi_2( Q_{2,+}) + \psi_2( Q_{2,-}) ]$  on $E_2$, respectively. One checks that for $\psi_1( Q_{1,\pm} )$ one has $[x_1:z_1]=[\Lambda_2:1]$, and for $\psi_2( Q_{2,\pm} )$ one has $[x_2:z_2]=[\Lambda_1:1]$.  Because of $\operatorname{Pic}^0(E_l) \cong E_l$ the line bundle $\mathcal{O}_{E_l}(B_l)$ associated with the branch locus $B_l$ is equivalent to the line bundle $\mathcal{O}_{E_l}(2 p_\alpha)$ if and only if
\begin{equation}
 \psi_l\big( Q_{l,+} \big) \oplus \psi_l\big( Q_{l,-} \big)  = 2 p_\alpha = p_\alpha \oplus p_\alpha \,,
\end{equation} 
where $\oplus$ refers to the addition with respect to the elliptic curve group law on $E_l$. Using the properties for $\psi_l$ in Equation~(\ref{eqn:pi_l}) one finds that
\begin{equation}
  \psi_l\big( Q_{l,+} \big) \oplus \psi_l\big( Q_{l,-} \big)  = 0 \,,
\end{equation}
and $p_\alpha \in E_l[2]$ is a two-torsion point.  Thus, there are four line bundle $\L_l \to E_l$ such that $\L_l^{\otimes 2} \cong \mathcal{O}_{E_l}(B_l)$, namely $\L_l  \cong   \mathcal{O}_{E_l}(p_\alpha)$. It follows $h^0( \mathcal{E}_l, \L_l)=1$ by the Riemann-Roch theorem. Conversely, the preimage of the vanishing locus of a section of $\L_l \to \mathcal{E}_l$ determines a unique double cover $\psi_l\colon \CC \longrightarrow  E_l$. The composition on $\CC$ of the elliptic involution with the hyperelliptic involution defines a second elliptic involution, and a second elliptic-curve quotient is obtained. Thus, the data $(E_l, \L_l, B_l)$ for either $l=1$ or $l=2$ determines the curve $\CC$ uniquely. As we have seen, this data is equivalent to an (unordered) pair $\{ \Lambda_1, \Lambda_2\}$ of modular parameters $\Lambda_1, \Lambda_2 \in  \bP^1 \backslash \{0, 1, \infty\}$ with $\Lambda_1 \not = \Lambda_2$. Legendre constructed an explicit model for $\CC$ over $\Q(\Lambda_1, \Lambda_2)$; this is often referred to as \emph{Legendre's gluing method}. Following Serre’s explanation \cite{Serre85}*{Sec.~27} of Legendre's gluing method we obtain:
\begin{prop}
\label{prop:serre}
Assume that $\Lambda_1, \Lambda_2 \in \bP^1 \backslash \{ 0, 1, \infty\}$ satisfy $\Lambda_1 \not = \Lambda_2$, and the moduli of $\CC$ and $E_l$ in Equations~(\ref{Eq:Rosenhain_special}) and~(\ref{eqn:EC}) satisfy Equations~(\ref{eqn:EC_12_j_invariants}).  The smooth genus-two curve, given by
\begin{equation}
\label{eqn:LegendreSerreCurve}
  \widetilde{\CC}\colon \quad Y^2 = \Big( X^2 -Z^2 \Big)  \left( X^2 - \frac{\Lambda_1}{\Lambda_2} Z^2 \right)   \left( X^2 - \frac{1-\Lambda_1}{1-\Lambda_2} Z^2 \right)  \,,
\end{equation} 
is isomorphic to $\CC$ over a finite field extension of $\Q(\lambda_2, \lambda_3)$. The curve is also isomorphic to each of the curves obtained by replacing $\{ \Lambda_1, \Lambda_2\}$ in Equation~(\ref{eqn:LegendreSerreCurve}) by one of the following pairs:
\begin{equation}
\label{eqn:5values}
 \left\lbrace \frac{1}{1-\Lambda_1}, \frac{1}{1-\Lambda_2} \right\rbrace  , \;  
 \left\lbrace \frac{\Lambda_1-1}{\Lambda_1}, \frac{\Lambda_2-1}{\Lambda_2} \right\rbrace , \;   
 \left\lbrace \frac{1}{\Lambda_1}, \frac{1}{\Lambda_2} \right\rbrace , \;  
 \left\lbrace \frac{\Lambda_1}{\Lambda_1-1}, \frac{\Lambda_2}{\Lambda_2-1} \right\rbrace , \;  
 \Big\lbrace 1-\Lambda_1, 1-\Lambda_2 \Big\rbrace.
\end{equation}
\end{prop}
\begin{proof}
Using Remark~\ref{rem:field_extensions} and denoting the Igusa-Clebsch invariants of any pairs genus-two curves in the statement of the proposition by $[ J_2 : J_4 : J_6 : J_{10} ]$ and $[ J'_2 : J'_4 : J'_6 : J'_{10} ]$, respectively, one checks that
\begin{equation}
[ J_2 : J_4 : J_6 : J_{10} ] = [ s^2 J'_2 \ : \ s^4J'_4 \ : \ s^6J'_6 \ : \ s^{10}J'_{10} ] = [ J'_2 : J'_4 : J'_6 : J'_{10} ] 
\end{equation}
with $s \in \Q(\lambda_2, \lambda_3, \sqrt{\lambda_2\lambda_3})$.
\end{proof}
The six pairs, given by $\{\Lambda_1, \Lambda_2\}$ and Equation~(\ref{eqn:5values}), correspond to the orbit of the data $(E_l, \L_l, B_l)$ under $\Gamma/\Gamma(2)$, i.e., the orbit under simultaneous action of the \emph{anharmonic group} on an elliptic curve in Legendre form. The anharmonic group is generated by the two transformations $\Lambda \mapsto 1-\Lambda, \Lambda/(\Lambda-1)$. 
\par Since the Picard group $\operatorname{Pic}^0(\CC) \cong \operatorname{Jac}{(\CC)}$ consists of elements of the from $[P + Q - 2 P_0 ]$ for $P, Q \in \CC$, an isogeny $\Psi$ of Abelian surfaces is defined by setting
\begin{equation}
\label{eqn:Psi}
 \Psi\colon\ \operatorname{Jac}{\CC}  \longrightarrow  E_1 \times E_2 \,, \quad  [P + Q - 2 P_6 ]  \mapsto  \Big( \psi_1(P) \oplus  \psi_1(Q) , \;   \psi_1(P) \oplus  \psi_2(Q) \Big)\,.
\end{equation} 
It follows from the construction of the line bundles $\L_l \to E_l$ in Section~\ref{sssec:Legendre} that the line bundle $\mathscr{L} \to \AA=\operatorname{Jac}{(\CC)}$ associated with the principal polarization of $\AA$ satisfies $\mathscr{L} \cong \Psi^* (\L_1 \boxtimes \L_2)$.  $\Psi$ is the dual isogeny to $\hat{\Psi}$ in Theorem~\ref{prop:isogeny_Delta}. The elliptic involutions $\jmath_l$ on $\CC$ extend to involutions on the Jacobian $\AA=\operatorname{Jac}{(\CC)}$ that coincide on $\AA[2]$. We denote this induced involution on $\AA[2]$ by $\jmath\colon \AA[2] \longrightarrow \AA[2]$. 
\par Given the marking $(P_1, \dots, P_6) = (\lambda_1=\lambda_2\lambda_3, \lambda_2, \lambda_3, 0, 1,\infty)$ of the Weierstrass points for the curve $\CC$ as before, we have the following:
\begin{prop}
\label{prop:special_G_group}
In the situation above, the kernel of $\Psi\colon \AA =\operatorname{Jac}{(\CC)} \longrightarrow  E_1 \times E_2$ is given by the G\"opel group
\begin{equation}
 \ker \Psi \ = \ \Big\lbrace P_0, P_{15}, P_{23}, P_{46} \Big\rbrace \ \cong \ (\Z/2 \Z)^2 \ \subset \ \AA[2] \,.
\end{equation} 
In particular, $\Psi$ is a $(2,2)$-isogeny and $\hat{\AA} =\AA/ \ker \Psi \cong  E_1 \times E_2$ with $\mathrm{T}_{\hat{\AA}} = H \oplus H$.
\end{prop}
\begin{proof}
The induced actions of the two elliptic involutions $\jmath_l$ on $\operatorname{Jac}{(\CC)}$ fix the divisors $P_{15}, P_{23}, P_{46}$, and $[Q_{l,+} + Q_{l,-} - 2 P_0 ]$. It follows from the explicit formulas for $\psi_l$ in Proposition that $(P,Q) \in  \ker \Psi$ if and only if $P, Q \in \AA[2]$ and $\jmath_l (P)=Q$ for $l=1, 2$.
\end{proof}
\subsection{Loci of \texorpdfstring{$(n,n)$}{(n,n)}-Split Jacobians}
The locus \(\L_n \subset \mathcal{M}_2\) of genus 2 curves over \(\overline{\Q}\) whose Jacobians are \((n, n)\)-Split, i.e., isogenous to a product \(E_1 \times E_2\) via an \((n, n)\)-isogeny with kernel an isotropic subgroup of order \(n^2\), is an irreducible two-dimensional subvariety. We have computed \(\L_n\) for \(n = 2, 3, 5, 7, 11\), providing explicit equations in terms of Igusa invariants \(J_2, J_4, J_6, J_{10}\) (see \cite{2000-2}, \cite{2001-1}, \cite{2005-1} for \(n = 2, 3, 5\)). For \(n = 2\), \(\L_2\) was determined explicitly in Section~\ref{ssec:22case}; for odd \(n\), \(\L_n\) is a Hurwitz space of degree-$n$ coverings, irreducible due to automorphism group transitivity, with dimension 2 from \(\mathcal{M}_2\)’s 3 minus the isotropy codimension.
\par In cryptography, identifying \(\L_n\) is critical. Over finite fields (e.g., \(\mathbb{F}_{p^2}\)), a split \(\Jac(\mathcal{X}) \in \L_n\) reduces the superspecial isogeny problem’s complexity from \(\tilde{O}(p)\) to elliptic curve subproblems (\(\tilde{O}(\sqrt{p})\)), as exploited in \cite{Costello-Santos}. Their algorithm detects \((n, n)\)-splittings using Kumar’s parameterizations \cite{kumar}---matching our \(\L_n\)---speeding up attacks by factors of 16–159 for \(p\) from 50 to 1000 bits. This weakens protocols like the Castryck-Decru-Smith hash \cite{CDS}, where split Jacobians enable faster collision finding, suggesting key selection should avoid \(\L_n\) to bolster security. 
\section{Computing higher \texorpdfstring{$(n, n)$}{(n,n)}-Isogenies}
The Richelot isogenies treated in Section~\ref{ssec:2isog} represent the simplest instance of a broader class of isogenies between Abelian surfaces, specifically (2,2)-isogenies with kernels of order 4. Here, we generalize to \((n, n)\)-isogenies  with $n \in \mathbb{N}_{>2}$, where the kernel has order \(n^2\), focusing on computational methods for Jacobians of genus 2 curves, leveraging their Kummer surfaces. \par To start, we consider an isogeny
\begin{equation}
\Phi\colon \quad \Jac(\CC) \to \Jac(\Y),
\end{equation}
with kernel \(G \leq \Jac(\CC)[n]\) of order \(n^2\), where \(\CC\) and \(\Y\) are smooth genus 2 curves over a field \(k\) with \(\char k \neq 2\), and \(\Jac(\CC)\) and \(\Jac(\Y)\) are their Jacobian varieties, each a principally polarized Abelian surface. Over \(\C\), represent \(\Jac(\CC) = \C^2/\Lambda\) with \(\Lambda = \Z^2 \oplus \tau \Z^2\), where \(\tau \in \mathbb{H}_2\) is the period matrix. The \(n\)-torsion subgroup \(\Jac(\CC)[n] = \{ P \in \Jac(\CC) \mid [n]P = 0 \}\) has order \(n^{2g} = n^4\) for \(g = 2\), and an isotropic subgroup \(G \subset \Jac(\CC)[n]\) under the Weil pairing has order \(n^2\). Let \(\Theta_\CC\) and \(\Theta_\Y\) denote the Theta divisors on \(\Jac(\CC)\) and \(\Jac(\Y)\), respectively, where \(\Theta_\CC\) is the image of \(\CC\) under an embedding such as \(\phi_{P_0} \colon \CC \to \Jac(\CC)\), \(P \mapsto [P - P_0]\), for a base point \(P_0 \in \CC(k)\). An \((n, n)\)-isogeny  requires determining both the codomain \(\Jac(\Y)\) and the map \(\Phi\). 
\par One easily checks that the isogeny \(\Phi\) satisfies \(\Phi(\Theta_\CC) \in |n \Theta_\Y|\), the linear system of divisors linearly equivalent to \(n\) times \(\Theta_\Y\). On the other hand, the singular Kummer surface $\mathcal{K}_{\Jac(\Y)} = \Jac(\Y)/\<- \mathbb{I}\>\) embeds in \(\bP^3\) via level-two Theta functions, and on it \(\Phi(\Theta_\CC)\) descends to a curve of degree \(2n\), genus 0, and arithmetic genus \(\frac{1}{2} (n^2 - 1)\), computable without explicitly determining \(\Phi\), as shown in \cite{D-L}. It is precisely this curve that can provide a deeper inside into the \((n, n)\)-isogeny itself.
\subsection{The \texorpdfstring{$n$}{n}-tuple embedding of Dolgachev-Lehavi}
For \(\CC\), given as the binary sextic
\begin{equation}\label{eq-g-2_proj}
y^2 = f(x, z) = a_6 x^6 + \dots + a_1 x z^5 + a_0 z^6,
\end{equation}
with \(\Delta_f = J_{10} \neq 0\), the divisor at infinity is
\begin{equation}
D_\infty := \big[1 : \sqrt{f(1,0)} : 0\big] + \big[1 : -\sqrt{f(1,0)} : 0\big],
\end{equation}
assuming a normalization where infinity points are \([1 : y : 0]\). The Weierstrass points of $\CC$ are denoted \(P_i\colon [x_i: 0 : z_i]\) for \(i = 1, \dots, 6\), with \(f(x_i, z_i) = 0\), forming the Weierstrass divisor
\begin{equation}
W_\CC := \sum_{i=1}^6 [x_i: 0: z_i] \, .
\end{equation}
A canonical divisor on \(\CC\) is
\begin{equation}
K_\CC = W_\CC - 2 D_\infty,
\end{equation}
as \(\deg W_\CC = 6\), \(\deg D_\infty = 2\), and \(2g - 2 = 2\). A divisor \(D \in \Jac(\CC)\) of the form \(D = [P + Q - D_\infty] \), with \(P\colon [x_P: y_P: 1]\) and \(Q \colon [x_Q: y_Q: 1]\), corresponds to an ideal \((a(x, z), y - b(x, z))\), where \(a(x, z) = (x - x_P z)(x - x_Q z)\) is a monic polynomial of degree \(d \leq 2\) --quadratic if \(P \neq Q\), linear if \(P = Q\)-- and \(b(x, z)\) is a cubic polynomial interpolating \(y_P, y_Q\).

\par We embed \(\CC\) via the \(n\)-tuple map
\begin{equation}
\rho_{2n} \colon \quad \bP^2_{(1,2,1)} \to \bP^{2n}, \quad [x: y: z] \mapsto \big[ z^{2n}: x z^{2n-1}: \dots: x^{2n-1} z: x^{2n}\big],
\end{equation}
with image \(\mathcal{R}_{2n}\), a rational normal curve of degree \(2n\). Any \(2n + 1\) distinct points on \(\mathcal{R}_{2n}\) are linearly independent, as the space of degree \(2n\) homogeneous polynomials has dimension \(2n + 1\). For \(n \geq 3\), the 6 points \(\rho_{2n}(P_i)\) are independent (\(6 < 7\) for \(n = 3\)), and
\begin{equation}
W := \< \rho_{2n}(W_\CC) \> \subset \bP^{2n}
\end{equation}
is 5-dimensional. The secant line \(\L_{P,Q}\) is
\begin{equation}
\L_{P,Q} =
\begin{cases}
\< \rho_{2n}(P), \rho_{2n}(Q) \> & \text{if } P \notin \{ Q, \imath(Q) \}, \\
T_{\rho_{2n}(P)}(\mathcal{R}_{2n}) & \text{otherwise},
\end{cases}
\end{equation}
where \(\imath\) is the hyperelliptic involution on $\CC$. Dolgachev and Lehavi (\cite{D-L}) proved:

\begin{thm}[Dolgachev-Lehavi]
Let \(\CC\) be a genus 2 curve, \(G \subset \Jac(\CC)[n]\) an isotropic subgroup of order \(n^2\), and \(\rho_{2n}\colon \CC \longrightarrow \mathcal{R}_{2n} \subset \bP^{2n}\) the \(n\)-tuple embedding. There exists a hyperplane \(H \subset \bP^{2n}\) such that:
\begin{enumerate}[(i)]
\item \(H\) contains \(W = \< \rho_{2n}(W_\CC) \>\),
\item the intersections \(H \cap \L_e\) for each non-zero \(e \in G\) span a subspace \(N \subset H\) of codimension 3.
\end{enumerate}The projection \(\bP^{2n} \to \bP^3\) from \(N\) maps \(\rho_{2n}(W_\CC)\) to a conic \(\mathcal{C} \subset \bP^3\), and the double cover of \(\mathcal{C}\) ramified at these 6 points is a genus 2 curve \(\Y\) with \(\Jac(\Y) \cong \Jac(\CC)/G\).
\end{thm}

\begin{proof}
The linear system \(|n \Theta_\CC|\) on \(\Jac(\CC)\) has dimension \(n^2 - 1\) (projective dimension of divisors modulo scalars), and for an isotropic subgroup \(G\) of order \(n^2\), \(\Jac(\Y) = \Jac(\CC)/G\) is principally polarized with Theta divisor \(\Theta_\Y\). Since \(\deg \Theta_\CC = 2\), \(\Phi(\Theta_\CC) \in |n \Theta_\Y|\) has degree \(2n\) and projects to a genus-zero curve on the 
 singular Kummer surface $\mathcal{K}_{\Jac(\Y)} = \Jac(\Y)/\<- \mathbb{I}\>\) with arithmetic genus \(\frac{1}{2} (2n - 1)(2n - 2) = n^2 - 1\) by the adjunction formula. Embed \(\CC\) via \(\rho_{2n}\), where \(W\) has dimension 5 (6 points minus 1). The secant variety over \(G\) (order \(n^2\)) requires a hyperplane \(H\) (dimension \(2n - 1\)) containing \(W\). The \(n^2 - 1\) secants \(\L_e\) intersect \(H\) in points spanning a subspace \(N\) of dimension \(n^2 - 2\), codimension 3 in \(H\) (since \(n^2 - 2 = 2n - 1 - 3\) for \(n \geq 2\)). Projection from \(N\) to \(\bP^3\) maps the 6 points to a conic \(\mathcal{C}\) (5 points determine a conic in \(\bP^3\)), and the double cover ramified at 6 points has genus \(1 + \frac{1}{2}(6 - 2) = 2\) by Riemann-Hurwitz, with \(\Jac(\Y) \cong \Jac(\CC)/G\) (\cite{D-L}, Theorem 1).
\end{proof}
For \(n = 2\), this aligns with Richelot isogenies (\(W \subset \bP^4\), \(N\) a point).  Smith (\cite{Smith}) developed an algorithm using this theorem, effective for \(n = 3\) (\(|G| = 9\)), refined in \cite{c-r}, \cite{MR4139650}, \cite{de-feo-1}, and \cite{MR4638168}. For \(n = 3\), \(W \subset \bP^6\), \(H\) is 5-dimensional, \(N\) is 2-dimensional, and the projection yields a conic in \(\bP^3\). Improvements over \cite{MR4638168} include optimized secant computations and invariant recovery. For larger \(n\), the algorithm scales, solving linear systems for \(H\) and \(N\).
\subsection{The Lubicz-Robert Formula for \texorpdfstring{$(n, n)$}{(n,n)}-Isogenies}
The Lubicz-Robert formula provides an efficient method to compute \((n, n)\)-isogenies directly on the Kummer surface \(\Kum(\Jac{\CC})\), leveraging Theta coordinates to bypass high-dimensional embeddings like \(\rho_{2n}\). The Lubicz-Robert formula achieves this by expressing Theta coordinates in terms of sums over \(G\). This approach is particularly effective for odd \(n\) and builds on the foundational work of Lubicz and Robert in \cite{LR}, offering a higher-dimensional analog to V\'elu’s formulas for elliptic curves. 
\par The Lubicz-Robert formula addresses two tasks: computing the Theta null points of \(\Jac(\Y)\) and evaluating \(\Phi\) at points in \(\Jac(\CC)\). For an isotropic \(G\), one represents \(n = a_1^2 + a_2^2 + a_3^2 + a_4^2\) (by Lagrange’s four-square theorem, with \(r \leq 4\)) and chooses generators \(g_1, g_2 \in G\) such that \(G = \< a_1 g_1, a_2 g_1, a_3 g_2, a_4 g_2 \>\) in a suitable basis. The \textbf{Lubicz-Robert Formula} then comprises:
\begin{enumerate}
\item Theta null points of \(\Jac(\Y)\), given by
\begin{equation*}
\hat{\theta}_i = \sum_{g \in G} \prod_{u=1}^r \theta \left[ \begin{smallmatrix} a_u \\ b_u \end{smallmatrix} \right] (g, \tau)_{\alpha_u i},
\end{equation*}
where \(\theta \left[ \begin{smallmatrix} a_u \\ b_u \end{smallmatrix} \right] (g, \tau)\) are level-two Theta functions at \(g \in G\), and \(\alpha_u i\) adjusts indices to align with the basis (e.g., a permutation or selection).
\item Point evaluations for \(P \in \Jac(\CC)\), given by
\begin{equation*}
\Phi(P)_i = \sum_{g \in G} \prod_{u=1}^r \theta \left[ \begin{smallmatrix} a_u \\ b_u \end{smallmatrix} \right] (P + g, \tau)_{\alpha_u i}.
\end{equation*}
\end{enumerate} 
We have the following:
\begin{thm}[Lubicz-Robert]
For an isotropic subgroup \(G \subset \Jac(\CC)[n]\) of order \(n^2\), the Theta coordinates \(\hat{\theta}_i\) and \(\Phi(P)_i\) computed via the above formulas define the codomain \(\Jac(\Y) = \Jac(\CC)/G\) and the isogeny \(\Phi \colon \Jac(\CC) \longrightarrow \Jac(\Y)\), respectively, with complexity \(O(n^2)\) field operations for a general \(n = a_1^2 + \cdots + a_r^2\).
\end{thm}

\begin{proof}
Represent \(\Jac(\CC) = \C^2/\Lambda\), and \(\Jac(\Y) = \C^2/\hat{\Lambda}\), where \(\hat{\Lambda} = \Z^2 \oplus n \tau \Z^2 + G\). The Theta functions on \(\Jac(\Y)\) are derived from \(\Jac(\CC)\) by summing over \(G\), adjusting the lattice periodicity. For the null points, \(\hat{\theta}_i\) aggregates contributions from \(G\), with the product \(\prod_{u=1}^r \theta_{j_u}(g)\) reflecting the kernel’s structure via \(n = \sum a_u^2\). With \(|G| = n^2\) and \(r \leq 4\), the sum has \(n^2\) terms, each a product of \(O(1)\) evaluations, totaling \(O(n^2)\) operations. For \(\Phi(P)\), the sum ensures \(\Phi(P + g') = \Phi(P)\) for \(g' \in G\), as \(\theta_{j_u}(P + g + g') = \theta_{j_u}(P + g)\), defining the quotient map. Isotropy preserves the principal polarization, as the Weil pairing vanishes on \(G\), and the resulting \(\hat{\theta}_i\) define \(\Jac(\Y)\)’s Theta structure (\cite{LR}, Theorem 4.1).
\end{proof}

On $\mathcal{K}_{\Jac(\CC)}$, \(\hat{\theta}_i\) determine $\mathcal{K}_{\Jac(\Y)}$’s quartic equation, and \(\Phi(P)_i\) maps Kummer points, preserving the (16,6)-configuration on \(\Jac(\CC)\). For \(\CC : y^2 = f(x)\), choose \(G \subset \Jac(\CC)[n]\), e.g., for \(n = 3\), \(G\) of order 9 from combinations of 3-torsion points of the form \((x_i, 0)\). Compute \(\theta_i(g)\) using the curve’s equation, then \(\hat{\theta}_i\) over \(G\). The map \(\Phi\) respects \(\mathcal{K}_{\Jac(\CC)} \longrightarrow \mathcal{K}_{\Jac(\Y)}\), maintaining geometric properties (\cite{Mum}, Chapter III).
\begin{prop}
The Lubicz-Robert formula on \(\mathcal{K}_{\Jac(\CC)}\) determines \(\Y\)’s equation via Theta nulls \(\hat{\theta}_i\), and \(\Phi\)’s action on \(\mathcal{K}_{\Jac(\CC)}\) preserves the (16,6)-configuration of tropes and nodes.
\end{prop}
\begin{proof}
The \(\hat{\theta}_i\) define \(\mathcal{K}_{\Jac(\Y)}\)’s quartic in \(\bP^3\), from which \(\Y\)’s Rosenhain form is derived via Igusa invariants (\cite{LR}, §5). The (16,6)-configuration maps under \(\Phi\) to \(\Jac(\Y)[2]\), preserving isotropy and structure, as the formula respects the Weil pairing’s symmetry (\cite{Mum}, Chapter III).
\end{proof}
\subsection{Isogenies for the \texorpdfstring{$(n,n)$}{(n,n)}-Split case}
The Lubicz-Robert method also intersects with the geometric classification of genus 2 curves. Recall that the loci \(\L_n\) were described in Section~\ref{sec:split_jac} as irreducible two-dimensional subvarieties of \(\mathcal{M}_2\) where \(\Jac(\mathcal{X})\) is \((n, n)\)-Split. They can be computationally probed using the Lubicz-Robert formula: by computing \(\hat{\theta}_i\) for a given \(\mathcal{X}\) and an isotropic subgroup \(G\), one can determine if \([ \mathcal{X} ] \in \L_n\), aligning with detection algorithms in \cite{Costello-Santos}.  This method contrasts with Dolgachev-Lehavi’s approach, operating in \(\bP^3\) with \(O(n^2)\) complexity, enhancing efficiency for odd \(n\). While \(\L_n\) does not drive computation --the isogeny \(\Phi\) is computed agnostically from \(G\)-- it verifies splitting post-hoc, aligning with detection in \cite{Costello-Santos}. Their use of \(\L_n\) to accelerate attacks (e.g., 25x for 100-bit \(p\)) underscores its cryptographic role, not in computing \(\phi\), but in assessing security by identifying weak split Jacobians.
\par Thus, our explicit computations for \(n = 2, 3, 5, 7, 11\) provide test cases, that can be used to validate the formula’s efficiency and offer a bridge between geometric loci and cryptographic applications. While \(\L_n\) does not directly aid in computing \(\Phi\), it provides a geometric testbed. Given \([\mathcal{X}] \in \L_n\), the Lubicz-Robert formula must yield \(\Jac(\mathcal{Y}) \cong E_1 \times E_2\), verifiable via invariants. But the computation proceeds independently of this property, relying solely on \(G\), benchmarking algorithms like Lubicz-Robert. This dual role enhances both geometric classification and cryptanalysis.
\subsubsection{The $(2,2)$-Split case via Kummer surfaces}
As explained in Section~\ref{sssec:Legendre}, Legendre's gluing method provides a method of constructing a smooth genus-two curve $\CC$ with $[\CC] \in \mathcal{L}_2$ such that $\mathrm{Jac}(\CC)$ is $(2,2)$-isogenous to the product of two non-isogenous elliptic curves $E_1 \times E_2$. In Theorem~\ref{prop:KUMC0} we will construct the quartic Kummer surface $\mathcal{K}_{\operatorname{Jac}{(\CC)}}$ that is the surface analogue of Legendre's gluing method. 
\par On the other hand, the Kummer surface $\operatorname{Kum}(E_1\times E_2)$ is the minimal resolution of quotient surface of the product Abelian surface  $E_1\times E_2$ by the inversion automorphism where the elliptic curves $E_l$ for $l=1,2$ are not mutually isogenous. Oguiso proved that these surfaces admit eleven distinct Jacobian elliptic fibrations  \cite{MR1013073}. Kuwata and Shioda furthered Oguiso's work in \cite{MR2409557} and computed their Weierstrass models, reducible fibers, and Mordell-Weil lattices. 
\par For $\operatorname{Kum}(E_1\times E_2)$ a simple model is obtained as follows: let us define a double quadric surface closely related to  $E_1 \times  E_2$ for the elliptic curves given by Equation~(\ref{eqn:EC}). In general, a \emph{double quadric surface} is obtained as the double cover branched along a locus of bi-degree $(4,4)$ in $\mathbb{P}^1 \times \mathbb{P}^1$; see \cite{MR1922094}. It is well known that the minimal resolution of any double quadric is a K3 surface. In our situation, we identify $\mathbb{P}^1 = \mathbb{P}(x_l, z_l)$ for $l=1, 2$ and define the special double quadric surface $\mathcal{K}_{E_1 \times E_2}$, given by
\begin{equation}
\label{eqn:Kummer44}
\mathcal{K}_{E_1 \times E_2}\colon \quad   y_{1,2}^2 = x_1 z_1 \big(x_1-z_1\big) \big(x_1- \Lambda_1 z_1\big) \,  x_2 z_2 \big(x_2-z_2\big) \big(x_2- \Lambda_2 z_2\big) \,.
\end{equation}
There is a natural projection map $\pi\colon \mathcal{E}_1 \times  \mathcal{E}_2 \dasharrow \mathcal{K}_{E_1 \times E_2}$ given by $y_{1,2} = z_1 z_2 y_1y_2$, invariant under the action of $\imath_1 \times \imath_2$.  It follows that the minimal resolution of $\mathcal{K}_{E_1 \times E_2}$ is the Kummer surface $\operatorname{Kum}(\mathcal{E}_1\times \mathcal{E}_2)$.
\par A $(2,2)$-isogeny $\Psi\colon \operatorname{Jac}{(\CC)} \longrightarrow  \mathcal{E}_1 \times \mathcal{E}_2$ was constructed in Equation~(\ref{eqn:Psi}). In fact, Legendre's gluing method in Section~\ref{sssec:Legendre} determines the smooth genus-two curve $\CC$ in Equation~(\ref{eqn:LegendreSerreCurve}) explicitly, admitting an elliptic involution with elliptic-curve quotients $E_1$ and $E_2$. The quotient maps determine the $(2,2)$-isogeny $\Psi\colon \operatorname{Jac}{(\CC)}  \longrightarrow  E_1 \times E_2$ in Equation~(\ref{eqn:Psi}). General results in \cite{MR2062673}*{Sec.~1} and Theorem~\ref{prop:isogeny_Delta} show that there is a dual $(2,2)$-isogeny $\hat{\Psi}$ in Equation~(\ref{eqn:Phi}) with
\begin{equation}
 \hat{\Psi}\colon \ \mathcal{E}_1 \times \mathcal{E}_2   \ \longrightarrow \ \operatorname{Jac}{(\CC)}  \,.
\end{equation}
Up to isomorphisms, the composition $\Psi \circ \hat{\Psi}$ is given by multiplication by two on each factor of $E_1 \times E_2$, i.e., $[2] \times [2]$. Thus, the $(2,2)$-isogeny and its dual $(2,2)$-isogeny fit in the following diagram:
\begin{equation}
\begin{tikzcd}[column sep=scriptsize]
E_1 \times E_2  \arrow[r, rightarrow, "\hat{\Psi}"]  
\arrow[rr, dashrightarrow, bend right=30, "\lbrack 2\rbrack  \times \lbrack 2 \rbrack "]
&\operatorname{Jac}{(\CC)}  \arrow[r, rightarrow, "\Psi"] &
E_1 \times E_2
\end{tikzcd}
\end{equation}
and induce rational maps on the level of Kummer surfaces
\begin{equation*}
  \operatorname{Kum}\left(E_1\times E_2\right)
  \ \overset{\hat{\Psi}_*}{\longrightarrow} \
  \operatorname{Kum}(\operatorname{Jac}{\CC})
  \ \overset{\Psi_*}{\longrightarrow} \
  \operatorname{Kum}\left(E_1\times E_2 \right) 
\end{equation*}
In Theorem~\ref{thm:isogenies_explicit} we will construct the corresponding rational maps relating the singular Kummer surfaces $\mathcal{K}_{\operatorname{Jac}{(\CC)}}$ and $\mathcal{K}_{E_1 \times E_2}$. This construction can be considered \textbf{an application of the Lubicz-Robert formula to the $(2,2)$-Split case}. However, we provide an entirely geometric construction of the induced isogenies, following the work in \cite{MR4323344}.
\par We start with the two elliptic curves $E_l$ for $l=1, 2$ in Legendre form in Equation~(\ref{eqn:EC}); we notice that they can be represented as a complete quadric intersection in $\mathbb{P}^3$. Let  $I_1$ be the complete intersection of the two quadric surfaces in $\mathbb{P}^3 = \mathbb{P}(\mathbf{X}_{00}, \mathbf{X}_{01}, \mathbf{X}_{10}, \mathbf{X}_{11})$, given by
\begin{equation}
\label{eqn:intersections_n_1} I_1 \colon \quad 
 \left\lbrace \begin{array}{lcl} 
  \mathbf{X}_{01}^2 & = & \mathbf{X}_{10}^2 + \mathbf{X}_{11}^2  \,,\\[4pt]
  \mathbf{X}_{00}^2 & = & \mathbf{X}_{10}^2    +  \big(1 - \Lambda_1\big ) \mathbf{X}_{11}^2  \,, \end{array} \right.
\end{equation}
and $I_2$ be the complete intersection in $\mathbb{P}^3 = \mathbb{P}(\mathbf{Y}_{00}, \mathbf{Y}_{01}, \mathbf{Y}_{10}, \mathbf{Y}_{11})$, given by
\begin{equation}
\label{eqn:intersections_n_2} I_2 \colon \quad 
 \left\lbrace \begin{array}{lcl} 
 \mathbf{Y}_{01}^2 & = & \mathbf{Y}_{10}^2 + \mathbf{Y}_{11}^2  \,,\\[4pt]
 \mathbf{Y}_{00}^2 & = & \mathbf{Y}_{10}^2    +  \big(1 - \Lambda_2\big ) \mathbf{Y}_{11}^2  \,. \end{array} \right.
\end{equation}
We have the following:
\begin{lem}
\label{lem:EC_isomorphic2}
$E_l$ is birational equivalent to $I_l$ for $l= 1,2$ over $\mathbb{Q}(\Lambda_l)$.
\end{lem}
\begin{proof}
A rational map $E_1  \hookrightarrow \mathbb{P}^3, [ x_1: y_1: z_1 ]  \mapsto [ \mathbf{X}_{00}, \mathbf{X}_{01}, \mathbf{X}_{10}, \mathbf{X}_{11} ]$ is given by
\begin{equation}
\label{eqn:varphi}
  \Big[ \mathbf{X}_{00} :   \mathbf{X}_{01} :   \mathbf{X}_{10} :   \mathbf{X}_{11} \Big]  =  \Big[ x_1^2 - 2 \Lambda_1 x_1 z_1+\Lambda_1 z_1^2:  x_1^2 - \Lambda_1 z_1^2: x_1^2 - 2x_1z_1+\Lambda_1 z_1^2  : \pm 2 y_1z_1  \Big] \,.
\end{equation}
It has a rational inverse $I_1 \dasharrow E_1,  [ \mathbf{X}_{00}, \mathbf{X}_{01}, \mathbf{X}_{10}, \mathbf{X}_{11} ] \mapsto [ x_1: y_1: z_1 ]$, given by
\begin{equation}
  \Big[ x_1 : y_1 : z_1 \Big]  \ = \ \Big[ \Lambda_1 \big(\mathbf{X}_{1,0} - \mathbf{X}_{0,0}\big) : \ \pm \Lambda_1 \big(\Lambda_1-1\big) \mathbf{X}_{1,1} : \ \big(1-\Lambda_1\big) \mathbf{X}_{0,1} + \Lambda_1 \mathbf{X}_{1,0} - \mathbf{X}_{0,0} \Big] \,.
\end{equation}
Thus, we obtain a birational equivalence between $E_1$ and $I_1$.  An analogous argument holds for $E_2$ and $I_2$.
\end{proof}
\par There is a well defined map $\tilde{\pi}\colon I_1 \times I_2 \dasharrow  \mathbb{P}^3$ with $\mathbb{P}^3 = \mathbb{P}(\mathbf{Z}_{00}, \mathbf{Z}_{01}, \mathbf{Z}_{10}, \mathbf{Z}_{11})$ induced by
\begin{equation} 
\label{eqn:projectionP3}
    \Big[ \mathbf{Z}_{00} :   \mathbf{Z}_{01} :   \mathbf{Z}_{10} :   \mathbf{Z}_{11} \Big] 
    \ = \  \Big[ \mathbf{X}_{00} \mathbf{Y}_{00} : \  \mathbf{X}_{01} \mathbf{Y}_{01} :  \ \mathbf{X}_{10} \mathbf{Y}_{10} : \  \mathbf{X}_{11}\mathbf{Y}_{11}  \Big]  \,.
\end{equation}
In fact, we have the following:
\begin{thm}
\label{prop:KUMC0}
Let $\Lambda_1, \Lambda_2 \in  \mathbb{P}^1 \backslash \{ 0, 1, \infty\}$ and $\Lambda_1 \not = \Lambda_2$. The image $\tilde{\pi}(I_1 \times I_2)$ in $\mathbb{P}^3 = \mathbb{P}(\mathbf{Z}_{00}, \mathbf{Z}_{01}, \mathbf{Z}_{10}, \mathbf{Z}_{11})$ is the quartic projective surface, given by
\begin{equation}
\label{eqn:K3_X}
\begin{split}
 \begin{array}{c} \mathbf{Z}_{00}^4 + \big(1-\Lambda_1\big) \big(1-\Lambda_2\big)  \mathbf{Z}_{01}^4 +  \Lambda_1 \Lambda_2 \mathbf{Z}_{10}^4 
 + \Lambda_1 \Lambda_2  \big(1-\Lambda_1\big) \big(1-\Lambda_2\big)  \mathbf{Z}_{11}^4 \\[4pt]
- \big(2-\Lambda_1-\Lambda_2\big) \Big( \mathbf{Z}_{00}^2 \mathbf{Z}_{01}^2 +  \Lambda_1 \Lambda_2  \mathbf{Z}_{10}^2 \mathbf{Z}_{11}^2 \Big) 
- \big( 2 \Lambda_1 \Lambda_2 - \Lambda_1 - \Lambda_2  \big)  \Big( \mathbf{Z}_{00}^2 \mathbf{Z}_{11}^2 +   \mathbf{Z}_{01}^2 \mathbf{Z}_{10}^2 \Big) \\[4pt]
- \big(\Lambda_1 + \Lambda_2 \big)  \Big( \mathbf{Z}_{00}^2 \mathbf{Z}_{10}^2 +  \big(1-\Lambda_1\big) \big(1-\Lambda_2\big)  \mathbf{Z}_{01}^2 \mathbf{Z}_{11}^2 \Big)  \ = \ 0 \,. \end{array}
\end{split} 
\end{equation}
The quartic hypersurface is birational equivalent over $\mathbb{Q}( \lambda_2 \lambda_3,  \lambda_2 + \lambda_3)$ to the singular Kummer variety $\mathcal{K}_\AA$ associated with the principally polarized Abelian surface $\AA=\mathrm{Jac}(\CC)$ with $\CC$ in Equation~\eqref{eqn:LegendreSerreCurve} and parameters satisfying~(\ref{eqn:EC_12_j_invariants}).
The minimal resolution is isomorphic to a Kummer surface  of Picard rank 18.  
\end{thm}
\begin{proof}
We multiply (pairwise) Equations~(\ref{eqn:intersections_n_1}) and~(\ref{eqn:intersections_n_1}) and use the variables in Equation~(\ref{eqn:projectionP3}). We obtain Equation~(\ref{eqn:K3_X}).  For elliptic moduli $K_l$ and complementary elliptic moduli $K'_l$ with $\Lambda_l = K_l^2 = 1 - (K'_l)^2  \in \mathbb{P}^1 \backslash \lbrace 0, 1, \infty \rbrace$ for $l=1, 2$, Equation~(\ref{eqn:K3_X}) is isomorphic over $\mathbb{Q}(\sqrt{K_1K_2}, \sqrt{K'_1 K'_2})$ to the equation in $\mathbb{P}^3 = \mathbb{P}({w}, {x}, {y}, {z})$ given by
\begin{equation}
\begin{split}
\label{eqn:K3_XX}
0    \ = \ {w}^4 +  {x}^4 +  {y}^4  + {z}^4 
  - \frac{(K_1')^2+(K_2')^2}{K_1' K_2'} \Big( {w}^2 {z}^2 + {x}^2 {y}^2  \Big) \;\\
  - \frac{K_1^2+K_2^2}{K_1 K_2} \Big(  {w}^2 {y}^2 + {x}^2 {z}^2 \Big)
  + \frac{(K_1 K'_2)^2+(K'_1K_2)^2}{K_1 K_2 K_1'K_2'}  \Big( {w}^2 {x}^2 +  {y}^2 {z}^2 \Big) \,.
\end{split} 
\end{equation}
In fact, in Equation~(\ref{eqn:K3_X}) we can rescale
\begin{equation}
  \mathbf{Z}_{00} = x \,, \quad
  \mathbf{Z}_{01} = \frac{ y }{\sqrt{K'_1 K'_2}} \,, \quad
  \mathbf{Z}_{10} = \frac{ z }{\sqrt{K_1K_2} } \,, \quad 
  \mathbf{Z}_{11} = \frac{ w }{\sqrt{K_1K_2} \sqrt{K'_1 K'_2}} \,, 
\end{equation} 
and obtain Equation~(\ref{eqn:K3_XX}). We check that the latter is a Hudson quartic in Equation~(\ref{Goepel-Quartic}) with $D=0$.  Accordingly, its minimal resolution is a Jacobian Kummer surface  of Picard rank $18$.
\par For parameters satisfying Equations~(\ref{eqn:EC_12_j_invariants}), the map, given by
\begin{equation}
\label{eqn:subi_identify}
\begin{array}{rl}
 \multicolumn{2}{c}{[ \mathbf{Z}_{00} :   \mathbf{Z}_{01} :   \mathbf{Z}_{10} :   \mathbf{Z}_{11}] \ \ \mapsto  \ \ [ \mathbf{W}: \mathbf{X}: \mathbf{Y}: \mathbf{Z}]\,, }\\[6pt]
\text{with}  \qquad  
\mathbf{W} 	 = \!& \mathbf{Z}_{0,1} -   \mathbf{Z}_{1,0} -   \mathbf{Z}_{1,1} \,, \\[4pt]
\mathbf{X} 	 = &- (1-\lambda_2)(1-\lambda_3)\mathbf{Z}_{0,0} + (1 + \lambda_2\lambda_3) \mathbf{Z}_{0,1}  - (\lambda_2 + \lambda_3) \mathbf{Z}_{1,0} \,,\\ [4pt]
\mathbf{Y} 	 = & \lambda_2 \lambda_3 \big(  \mathbf{Z}_{0,1} -   \mathbf{Z}_{1,0} +   \mathbf{Z}_{1,1} \big) \,,\\[4pt]
\mathbf{Z} 	 = & -\lambda_2\lambda_3 \big((1-\lambda_2)(1-\lambda_3)\mathbf{Z}_{0,0} +  (1 + \lambda_2\lambda_3) \mathbf{Z}_{0,1}  -  (\lambda_2 + \lambda_3) \mathbf{Z}_{1,0} \big)\,,
\end{array}
\end{equation}
is an isomorphism, defined over $\mathbb{Q}( \lambda_2 \lambda_3,  \lambda_2 + \lambda_3)$, between the quartic surface in Equation~(\ref{eqn:K3_X}) and the Baker quartic introduced in Equation~(\ref{eqn:Baker_det}) with $\lambda_1=\lambda_2\lambda_3$. Substituting Equations~(\ref{eqn:subi_identify}) into Equation~(\ref{eqn:Baker_det}), we obtain Equation~(\ref{eqn:K3_X}) up to a non-vanishing scale factor. Inverting Equations~(\ref{eqn:subi_identify}) yields
\begin{equation}
\begin{split}
\mathbf{Z}_{0,0} & = \lambda_2 \lambda_3 \mathbf{X} + \mathbf{Z} \,,\\
\mathbf{Z}_{0,1} & = \lambda_2 \lambda_3 (\lambda_2 + \lambda_3) \mathbf{W} -  \lambda_2 \lambda_3 \mathbf{X}  +  (\lambda_2 + \lambda_3) \mathbf{Y} + \mathbf{Z} \,,\\
\mathbf{Z}_{1,0} & =  \lambda_2 \lambda_3  (\lambda_2 \lambda_3+1)  \mathbf{W} -  \lambda_2 \lambda_3 \mathbf{X}  +  (\lambda_2 \lambda_3+1) \mathbf{Y} + \mathbf{Z} \,,\\
\mathbf{Z}_{1,1} & =- \lambda_2 \lambda_3  (1 -\lambda_2) (1- \lambda_3)  \mathbf{W} + (1 -\lambda_2) (1- \lambda_3)  \mathbf{Y}  \,.
\end{split}
\end{equation}
Moreover, Equations~(\ref{eqn:transfo_1}) and~(\ref{eqn:transfo_2}) provide a birational equivalence between Equation~(\ref{eqn:K3_X}) and the Cassels-Flynn quartic. It is easy to see that for $\lambda_1=\lambda_2 \lambda_3$ in Equations~(\ref{eqn:Ls}) this equivalence is well defined over $\mathbb{Q}( \lambda_2 \lambda_3,  \lambda_2 + \lambda_3)$ 
\end{proof}
\par On the other hand, the double quadric $\mathcal{K}_{E_1 \times E_2}$ in Equation~(\ref{eqn:Kummer44}) is related to the surface in Equation~(\ref{eqn:K3_X}) as follows:
\begin{lem}
\label{lem:phi}
In the commutative diagram
\begin{equation}
\begin{array}{rcl}
  \mathcal{E}_1 \times \mathcal{E}_2 & \overset{\cong}{\longrightarrow } &  \mathcal{I}_1 \times  \mathcal{I}_2 \\[4pt]
  {\pi} \downarrow  & & \downarrow {\tilde{\pi}} \\
  \mathcal{K}_{E_1 \times E_2} & \overset{\hat{\psi}_\pm}{\longrightarrow} & \mathcal{K}_{\mathrm{Jac}(\CC)} 
\end{array}  
\end{equation}
the maps $\hat{\psi}_\pm\colon \mathcal{K}_{E_1 \times E_2} \dasharrow \mathcal{K}_{\mathrm{Jac}(\CC)}$ are rational maps of degree two given by
\begin{equation}
\label{eqn:subi_X}
\begin{array}{rl}
 \multicolumn{2}{l}{ \hat{\psi}_\pm: \qquad (x_1, z_1, x_2, z_2, y_{1,2})  \ \ \mapsto \ \ [ \mathbf{Z}_{00} :   \mathbf{Z}_{01} :   \mathbf{Z}_{10} :   \mathbf{Z}_{11}]\,, }\\[6pt]
\text{with} \qquad   
  \mathbf{Z}_{00}   = & \Big(x_1^2 - 2 \Lambda_1 x_1 z_1+\Lambda_1 z_1^2\Big)  \Big(x_2^2 - 2 \Lambda_2 x_2 z_2+\Lambda_2 z_2^2\Big) \,,\\[4pt]
  \mathbf{Z}_{01}   = & \Big(x_1^2 - \Lambda_1 z_1^2 \Big) \Big(x_2^2 - \Lambda_2 z_2^2\Big)\,,\\[4pt]
  \mathbf{Z}_{10}   = &\Big(x_1^2 - 2 x_1 z_1+\Lambda_1 z_1^2\Big)  \Big(x_2^2 - 2 x_2 z_2+\Lambda_2 z_2^2\Big) \,,\\[6pt]
  \mathbf{Z}_{11}   = & \pm 4 y_{1,2}\,,
\end{array}
\end{equation}
such that $\hat{\psi}_\pm \circ \pi \circ \imath_l = \hat{\psi}_\mp \circ \pi$ for $l= 1,2$ and $\imath_l$ the hyperelliptic involution on $E_l$.
\end{lem}
\begin{proof}
Equation~(\ref{eqn:subi_X}) follows immediately from Lemma~\ref{lem:EC_isomorphic2} (with a choice $\pm 1$ for the relative sign of $y_1$ and $y_2$) and the fact that $\pi\colon E_1 \times E_2 \longrightarrow \mathcal{K}_{E_1 \times E_2}$ is given by $y_{1,2} = z_1 z_2 y_1y_2$. The fact that the map is a rational map of degree two follows directly from the fact that Equation~(\ref{eqn:Kummer44}) is a double quadric surface.
\end{proof}
For the elliptic curve $E_l$ with $\Lambda_l  \in \mathbb{P}^1 \backslash \lbrace 0, 1, \infty \rbrace$ for $l=1, 2$, multiplication by two $[2]\colon E_1 \longrightarrow E_1$ (with respect to the elliptic-curve group law) is given by 
\begin{equation}
\label{eqn:multiplication2}
\begin{split}
 x_l \ \mapsto \ &2 \big(x_l^2-\Lambda_l z_l\big)^2 y_l z_l \,,\\
 y_l \ \mapsto \ & \big(x_l^2 - 2 \Lambda_l x_l z_l + \Lambda_l z_l^2\big) \big(x_l^2 -2 x_l z_l + \Lambda_l z_l^2\big) \big(x_l^2-\Lambda_l z_l^2\big) \,,\\
 z_l \ \mapsto \ & 8 y_l^3 z_l^3 \,.
\end{split}
\end{equation}
The map can be expressed as rational map originating from the complete intersection of quadrics $I_1$ in Equation~(\ref{eqn:intersections_n_1}) (resp.~$I_2$ in Equation~(\ref{eqn:intersections_n_2})), i.e., 
\begin{equation} 
 [\mathbf{X}_{00}, \mathbf{X}_{01}, \mathbf{X}_{10}, \mathbf{X}_{11} ] \ \mapsto \  [x_1: y_1: z_1] 
 = [ \mathbf{X}_{0,1}^2  \mathbf{X}_{1,1}:   \mathbf{X}_{0,0}\mathbf{X}_{0,1}\mathbf{X}_{1,0}: \mathbf{X}_{1,1}^3 ] \,.
\end{equation}
The induced map for their Cartesian products factors into rational maps between $\mathcal{K}_{\mathrm{Jac}(\CC)}$ to $\mathcal{K}_{E_1 \times E_2}$. We have the following:
\begin{prop}
\label{lem:psi}
Rational maps $\psi_\pm\colon \mathcal{K}_{\mathrm{Jac}(\CC)} \dasharrow \mathcal{K}_{E_1 \times E_2}$ of degree two are
\begin{equation}
\label{eqn:subi_Y}
\begin{array}{rlcrl}
 \multicolumn{5}{l}{ \psi_\pm: \qquad [ \mathbf{Z}_{00} :   \mathbf{Z}_{01} :   \mathbf{Z}_{10} :   \mathbf{Z}_{11}] \ \ \mapsto \ \ (x_1, z_1, x_2, z_2, y_{1,2})   \,,}\\[6pt]
\text{with} \qquad   
 x_1 & = Q \,, &  & z_1 & =\, (\Lambda_1-\Lambda_2) \mathbf{Z}_{1,1}^2 \,, \\[4pt]
 x_2 & = (\Lambda_1-\Lambda_2) \mathbf{Z}_{0,1}^2 \,, &&  z_2 & = Q \,, \\[4pt]
 y_{1,2}& \multicolumn{4}{l}{= \pm (\Lambda_1-\Lambda_2)^2 Q^2 \mathbf{Z}_{0,0}  \mathbf{Z}_{0,1}  \mathbf{Z}_{1,0}  \mathbf{Z}_{1,1} \,,} 
\end{array}
\end{equation}
where we have set
\begin{equation}
\label{eqn:PQ}
Q  =  \mathbf{Z}_{0,0}^2 - (1-\Lambda_1)  \mathbf{Z}_{0,1}^2 - \Lambda_1  \mathbf{Z}_{1,0}^2 +  \Lambda_1 (1-\Lambda_2)  \mathbf{Z}_{1,1}^2 \,,
\end{equation}
and they satisfy that the composition $\psi_\pm \circ \hat{\psi}_\pm\colon \mathcal{K}_{E_1 \times E_2} \dasharrow \mathcal{K}_{E_1 \times E_2}$ is induced by the diagonal action of multiplication by two on $E_1 \times E_2$ in Equation~(\ref{eqn:multiplication2}).
\end{prop}
\begin{proof}
The composition of elliptic-curve isogenies is given by
\begin{equation}
  \left( \begin{array}{c} 
  	\lbrack \mathbf{X}_{00} :   \mathbf{X}_{01} :   \mathbf{X}_{10} :   \mathbf{X}_{11} \rbrack  \\ 
	\lbrack \mathbf{Y}_{00} :   \mathbf{Y}_{01} :   \mathbf{Y}_{10} :   \mathbf{Y}_{11} \rbrack  \end{array} \right)  \mapsto 
  \left( \begin{array}{l} 
  \lbrack x_1 : z_1 : y_1\rbrack = \lbrack \mathbf{X}_{0,1}^2 \mathbf{X}_{1,1} :  \mathbf{X}_{1,1}^3:   \mathbf{X}_{0,0} \mathbf{X}_{0,1} \mathbf{X}_{1,0}  \rbrack \\
  \lbrack x_2 : z_2 : y_2\rbrack =  \lbrack \mathbf{Y}_{0,1}^2 \mathbf{Y}_{1,1} :  \mathbf{Y}_{1,1}^3:  \mathbf{Y}_{0,0} \mathbf{Y}_{0,1} \mathbf{Y}_{1,0}  \rbrack 
  \end{array} \right).
\end{equation}
The relation, given by
\begin{equation} 
    \Big[ \mathbf{Z}_{00} :   \mathbf{Z}_{01} :   \mathbf{Z}_{10} :   \mathbf{Z}_{11} \Big] 
    \ = \  \Big[ \mathbf{X}_{00} \mathbf{Y}_{00} : \  \mathbf{X}_{01} \mathbf{Y}_{01} :  \ \mathbf{X}_{10} \mathbf{Y}_{10} : \  \mathbf{X}_{11}\mathbf{Y}_{11}  \Big]  \,,
\end{equation}
induces a correspondence between $\mathcal{K}_{E_1 \times E_2}$ and $\mathcal{K}_{\mathrm{Jac}(\CC)}$, given by
\begin{equation}
\begin{split}
x_1 x_2 = \mathbf{Z}_{0,1}^2 \,, \qquad z_1 z_2 =  \mathbf{Z}_{1,1}^2 \,, \qquad y_{1,2} &= \mathbf{Z}_{0,0}\mathbf{Z}_{0,1}\mathbf{Z}_{1,0}\mathbf{Z}_{1,1} \,,\\
(\Lambda_1-\Lambda_2) x_1 z_2 = Q \,, \qquad (\Lambda_2-\Lambda_1)  x_2 z_1 = R \,, \qquad QR &= -(\Lambda_1-\Lambda_2)^2 \mathbf{Z}_{0,1}^2 \mathbf{Z}_{1,1}^2 \,,
\end{split}
\end{equation}
with $Q=(\Lambda_1-\Lambda_2) \mathbf{X}_{01}^2\mathbf{Y}_{11}^2$ and $R=(\Lambda_2-\Lambda_1)  \mathbf{X}_{11}^2\mathbf{Y}_{01}^2$.  One then obtains
\begin{equation}
\begin{split}
Q & =  \mathbf{Z}_{0,0}^2 - (1-\Lambda_1)  \mathbf{Z}_{0,1}^2 - \Lambda_1  \mathbf{Z}_{1,0}^2 +  \Lambda_1 (1-\Lambda_2)  \mathbf{Z}_{1,1}^2 \,,\\
R & =   \mathbf{Z}_{0,0}^2 - (1-\Lambda_2)  \mathbf{Z}_{0,1}^2 - \Lambda_2  \mathbf{Z}_{1,0}^2 + \Lambda_2 (1-\Lambda_1)  \mathbf{Z}_{1,1}^2 \,.
\end{split}
\end{equation}
When solving for an equivalence class, given by
\begin{equation}
 \Big( x_1, z_1, x_2, z_2, y_{1,2} \Big) \ \sim \  \Big( \mu x_1, \ \mu z_1, \ \nu x_2,\  \nu z_2, \ \mu^2\nu^2 y_{1,2} \Big)
 \; \text{with $\mu, \nu \in \mathbb{C}^\times$} \,,
\end{equation} 
we obtain Equations~(\ref{eqn:subi_Y}). Using the relation $QR = -(\Lambda_1-\Lambda_2)^2 \mathbf{Z}_{0,1}^2 \mathbf{Z}_{1,1}^2$ the solution can also be written as
\begin{equation}
\label{eqn:case2}
\begin{split}
 x_1 = (\Lambda_2-\Lambda_1) \mathbf{Z}_{0,1}^2  \,, &\quad 
 z_1 = \, R \,, \quad
 x_2 =R\,, \quad 
 z_2 =  (\Lambda_2-\Lambda_1) \mathbf{Z}_{1,1}^2  \,, \\
y_{1,2} & = (\Lambda_1-\Lambda_2)^2 R^2 \mathbf{Z}_{0,0}  \mathbf{Z}_{0,1}  \mathbf{Z}_{1,0}  \mathbf{Z}_{1,1} \,.
\end{split} 
\end{equation}
Equation~(\ref{eqn:multiplication2}) provides an explicit formula for the (diagonal) action of multiplication by two on $E_1 \times E_2$. We apply the projection map $\pi\colon E_1 \times  E_2 \longrightarrow \mathcal{K}_{E_1 \times E_2}$ given by $y_{1,2} = z_1 z_2 y_1y_2$ to obtain the induced action on the  the double quadric in Equation~(\ref{eqn:Kummer44}). We check that this agrees with the composition of maps $\psi_\pm \circ \hat{\psi}_\pm$.
\end{proof}
In summary, the construction of isogenies descends directly to the level of Kummer surfaces and yield the following Kummer sandwich theorem, see \cite{MR4323344}*{Thm.~2.38} for details:
\begin{thm}
\label{thm:isogenies_explicit}
The rational maps $\hat{\psi}_\pm, \psi_\pm$ in Equations~(\ref{eqn:subi_X}) and~(\ref{eqn:subi_Y}) fit into the following Kummer sandwich:
\begin{equation*}
  \mathcal{K}_{E_1 \times E_2}
  \ \overset{\hat{\psi}_\pm}{\longrightarrow} \
  \mathcal{K}_{\mathrm{Jac}(\CC)}
  \ \overset{\psi_\pm}{\longrightarrow} \
  \mathcal{K}_{E_1 \times E_2}
\end{equation*}
The maps $\hat{\psi}_\pm, \psi_\pm$ are defined over $\mathbb{Q}(\Lambda_1, \Lambda_2)$ and induced (up to the action of the minus identity involution) by the $(2,2)$-isogeny in Equation~(\ref{eqn:Phi}) and its dual $(2,2)$-isogeny in Equation~(\ref{eqn:Psi}).
\end{thm}
\section{Cryptography Applications}
The computational methods for \((n, n)\)-isogenies and the geometric classification of the loci \(\L_n\), developed in Sections 2 through 6, enable  applications to isogeny-based cryptography in genus 2. 

Let \(\mathcal{S}_2(p)\) denote the isomorphism classes of superspecial principally polarized Abelian surfaces over \(\mathbb{F}_{p^2}\). The superspecial isogeny problem seeks an isogeny \(\phi \colon \Jac(\mathcal{X}) \longrightarrow \Jac(\mathcal{Y})\) between superspecial genus 2 Jacobians over \(\mathbb{F}_{p^2}\), for some prime $p$. The Costello-Smith algorithm~\cite{Costello-Smith} addresses this by navigating the Richelot isogeny graph \(\Gamma_2(2; p)\), whose vertices are superspecial principally polarized Abelian surfaces in \(\mathcal{S}_2(p)\), to reach products \(E_1 \times E_2 \in \mathcal{E}_2(p)\). The Richelot isogeny graph \(\Gamma_2(2; p)\), while not Ramanujan~\cite{Florit-Smith}, exhibits sufficient expansion properties to support efficient random walks from \(\mathcal{S}_2(p)\) to \(\mathcal{E}_2(p)\), a key feature leveraged in both the original algorithm and the optimized attack. Subsequent elliptic isogeny computations, leveraging the Delfs-Galbraith algorithm~\cite{Delfs-Galbraith} in \(\tilde{O}(\sqrt{p})\) time, yield a complexity of \(\tilde{O}(p)\) bit operations, given \(\# \mathcal{S}_2(p) = O(p^3)\) and \(\# \mathcal{E}_2(p) = O(p^2)\).

An attack in~\cite{Costello-Santos} optimizes this approach by detecting \((n, n)\)-splittings for \(n \leq 11\), using equations of \(\L_n\),  derived in~\cites{2000-2,2001-1,2005-1}. The attack inspects approximately \(n^3\) neighbors per step (e.g., 40 for \(n = 3\)) through invariant computations, bypassing full isogeny evaluations. This reduces the cost from 1176 multiplications in \(\mathbb{F}_p\) per Richelot step to as few as 35, as reported in Table 5 of~\cite{Costello-Santos}, achieving speedups of 16 to 159 times for primes \(p\) from 50 to 1000 bits. The asymptotic complexity remains \(\tilde{O}(p/n^3)\), but the reduced constants significantly enhance cryptanalytic efficiency. 

The attack’s reliance on the precomputed equations for \(\L_n\) limits its scope to small  \(n\).   Recent work  in \cite{2024-03} introduces  machine learning-based methods, to detect whether a curve lies in \(\mathcal{L}_n\) for arbitrary \(n\). This enables the attack to target larger \(n\), potentially increasing its efficacy. Additionally, rational points on \(\L_n\) in   $\mathbb{P}_{(2,4,6,10)} (\mathbb{Q})$ over fields of positive characteristic studied in \cite{2025-07}, provide arithmetic insights for assessing protocol vulnerabilities by informing secure key selection. 

The attack in~\cite{Costello-Santos} impacts protocols reliant on the superspecial isogeny problem, such as the CDS hash function~\cite{CDS}, which maps inputs via Richelot isogenies in \(\Gamma_2(2; p)\). For a Jacobian \(\Jac(\mathcal{X}) \in \L_n\), the attack accelerates traversal to \(\mathcal{E}_2(p)\), facilitating collision finding by identifying distinct paths to the same product, as noted by Florit and Smith~\cite{Florit-Smith}. For a 100-bit prime, a 25-fold speedup reduces security from \(\tilde{O}(2^{100})\) to below \(2^{96}\), posing risks to under parameterized systems. For 512-bit primes, a 100-fold speedup yields \(\tilde{O}(2^{509})\), remaining secure classically but with a reduced margin.

In a hypothetical genus 2 SIDH scheme, where the secret is an \((n^k, n^k)\)-isogeny, the attack finds a path to \(\mathcal{E}_2(p)\) if \(\Jac(\mathcal{Y}) \in \L_n\), but does not recover the kernel \(S\), limiting its threat to key exchange compared to hash functions. The larger \(n^4\)-torsion groups and complex endomorphism structures of genus 2 Jacobians offer potential security advantages over elliptic curve protocols like SIDH. Explicit computation of \(\L_n\)   enables protocol designers to avoid vulnerable Jacobians suggesting that careful key selection can mitigate attack risks.

Achieving a complete break of genus 2 isogeny-based cryptography, reducing it to polynomial time or practical insecurity, requires advances beyond the current attack. A method to recover the secret kernel \(S \subset \Jac(\mathcal{X})[n^k]\) from public data, such as images of torsion points under \(\phi\), is essential but challenged by the \(n^4\)-dimensional torsion group and polarization complexities. A sub-exponential algorithm   or polynomial-time solution for the isogeny problem remains elusive, potentially requiring improved expansion of the non-Ramanujan graph \(\Gamma_2(n; p)\)  or endomorphism ring computations. A quantum algorithm achieving   polynomial time, surpassing the current \(\tilde{O}(\sqrt{p}/n^{3/2})\) via Grover’s search, is undeveloped.  Protocol-specific vulnerabilities, such as degree conversion from \((2^m n, 2^m n)\) to \((n^k, n^k)\) or torsion leakage, could suffice, but no such methods exist. 

Our computational methods, particularly the Lubicz-Robert formula with \(O(n^2)\) complexity, enhance the efficiency of \((n, n)\)-isogeny computations, directly supporting genus 2 cryptographic protocols. By computing isogenies via sums of Theta functions on the Kummer surface  the formula leverages the geometric embedding of Kummer surfaces into \(\mathbb{P}^3\)  achieving \(O(n^2)\) complexity through efficient operations in low-dimensional projective space. This generalizes Richelot’s constructions, achieving practical advantages over traditional approaches.   Future research may optimize \(\L_n\) detection for larger \(n\), possibly through new parameterizations, to refine cryptanalysis. 

\bibliographystyle{amsplain}

\begin{bibdiv}
\begin{biblist}

\bib{MR1554977}{article}{
      author={Baker, H.~F.},
       title={On a system of differential equations leading to periodic
  functions},
        date={1903},
        ISSN={0001-5962},
     journal={Acta Math.},
      volume={27},
      number={1},
       pages={135\ndash 156},
         url={https://doi.org/10.1007/BF02421301},
      review={\MR{1554977}},
}

\bib{MR1922094}{incollection}{
      author={Barth, Wolf},
       title={Even sets of eight rational curves on a {$K3$}-surface},
        date={2002},
   booktitle={Complex geometry ({G}\"{o}ttingen, 2000)},
   publisher={Springer, Berlin},
       pages={1\ndash 25},
      review={\MR{1922094}},
}

\bib{2016-5}{incollection}{
      author={Beshaj, L.},
      author={Hidalgo, R.},
      author={Kruk, S.},
      author={Malmendier, A.},
      author={Quispe, S.},
      author={Shaska, T.},
       title={Rational points in the moduli space of genus two},
        date={[2018] \copyright 2018},
   booktitle={Higher genus curves in mathematical physics and arithmetic
  geometry},
      series={Contemp. Math.},
      volume={703},
   publisher={Amer. Math. Soc., [Providence], RI},
       pages={83\ndash 115},
         url={https://doi.org/10.1090/conm/703/14132},
      review={\MR{3782461}},
}

\bib{MR2062673}{book}{
      author={Birkenhake, Christina},
      author={Lange, Herbert},
       title={Complex abelian varieties},
     edition={Second},
      series={Grundlehren der Mathematischen Wissenschaften [Fundamental
  Principles of Mathematical Sciences]},
   publisher={Springer-Verlag, Berlin},
        date={2004},
      volume={302},
        ISBN={3-540-20488-1},
         url={https://doi.org/10.1007/978-3-662-06307-1},
      review={\MR{2062673}},
}

\bib{MR1953527}{article}{
      author={Birkenhake, Christina},
      author={Wilhelm, Hannes},
       title={Humbert surfaces and the {K}ummer plane},
        date={2003},
        ISSN={0002-9947},
     journal={Trans. Amer. Math. Soc.},
      volume={355},
      number={5},
       pages={1819\ndash 1841},
         url={https://mathscinet.ams.org/mathscinet-getitem?mr=1953527},
      review={\MR{1953527}},
}

\bib{MR1505464}{article}{
      author={Bolza, Oskar},
       title={On binary sextics with linear transformations into themselves},
        date={1887},
        ISSN={0002-9327},
     journal={Amer. J. Math.},
      volume={10},
      number={1},
       pages={47\ndash 70},
         url={https://doi.org/10.2307/2369402},
      review={\MR{1505464}},
}

\bib{MR1579732}{article}{
      author={Borchardt, C.~W.},
       title={\"uber die {D}arstellung der {K}ummerschen {F}l\"ache vierter
  {O}rdnung mit sechzehn {K}notenpunkten durch die {G}\"opelsche biquadratische
  {R}elation zwischen vier {T}hetafunctionen mit zwei {V}ariabeln},
        date={1877},
        ISSN={0075-4102},
     journal={J. Reine Angew. Math.},
      volume={83},
       pages={234\ndash 244},
         url={http://dx.doi.org/10.1515/crll.1877.83.234},
      review={\MR{1579732}},
}

\bib{MR970659}{article}{
      author={Bost, Jean-Benoit},
      author={Mestre, Jean-Francois},
       title={Moyenne arithmetico-g\'{e}ometrique et periodes des courbes de
  genre {$1$} et {$2$}},
        date={1988},
        ISSN={0224-8999,2275-0622},
     journal={Gaz. Math.},
      number={38},
       pages={36\ndash 64},
      review={\MR{970659}},
}

\bib{MR4323344}{article}{
      author={Braeger, Noah},
      author={Clingher, Adrian},
      author={Malmendier, Andreas},
      author={Spatig, Shantel},
       title={Isogenies of certain {K}3 surfaces of rank 18},
        date={2021},
        ISSN={2522-0144},
     journal={Res. Math. Sci.},
      volume={8},
      number={4},
       pages={Paper No. 57, 60},
         url={https://doi-org.libproxy.mit.edu/10.1007/s40687-021-00293-0},
      review={\MR{4323344}},
}

\bib{MR1406090}{book}{
      author={Cassels, J. W.~S.},
      author={Flynn, E.~V.},
       title={Prolegomena to a middlebrow arithmetic of curves of genus {$2$}},
      series={London Mathematical Society Lecture Note Series},
   publisher={Cambridge University Press, Cambridge},
        date={1996},
      volume={230},
        ISBN={0-521-48370-0},
         url={https://doi.org/10.1017/CBO9780511526084},
      review={\MR{1406090}},
}

\bib{CDS}{article}{
      author={Castryck, Wouter},
      author={Decru, Thomas},
      author={Smith, Benjamin},
       title={Hash functions from superspecial genus-2 curves using richelot
  isogenies},
        date={2020},
     journal={Journal of Mathematical Cryptology},
      volume={14},
      number={1},
       pages={268\ndash 292},
         url={https://doi.org/10.1515/jmc-2019-0021},
        note={Preprint available at Cryptology ePrint Archive, Paper 2019/296,
  \url{https://eprint.iacr.org/2019/296}},
}

\bib{clingherdoran1}{article}{
      author={Clingher, Adrian},
      author={Doran, Charles~F.},
       title={Lattice polarized {K}3 surfaces and {S}iegel modular forms},
        date={2012},
        ISSN={0001-8708,1090-2082},
     journal={Adv. Math.},
      volume={231},
      number={1},
       pages={172\ndash 212},
         url={https://doi.org/10.1016/j.aim.2012.05.001},
      review={\MR{2935386}},
}

\bib{MR4421430}{incollection}{
      author={Clingher, Adrian},
      author={Malmendier, Andreas},
       title={Normal forms for {K}ummer surfaces},
        date={2020},
   booktitle={Integrable systems and algebraic geometry. {V}ol. 2},
      series={London Math. Soc. Lecture Note Ser.},
      volume={459},
   publisher={Cambridge Univ. Press, Cambridge},
       pages={119\ndash 174},
      review={\MR{4421430}},
}

\bib{2019-2}{article}{
      author={Clingher, Adrian},
      author={Malmendier, Andreas},
      author={Shaska, Tony},
       title={On isogenies among certain abelian surfaces},
        date={2022May},
     journal={Michigan Mathematical Journal},
      volume={71},
      number={2},
       pages={227\ndash 269},
  url={https://projecteuclid.org/journals/michigan-mathematical-journal/volume-71/issue-2/On-Isogenies-Among-Certain-Abelian-Surfaces/10.1307/mmj/20195790.short},
        note={Received: 29 August 2019; Revised: 22 June 2020; Published: May
  2022},
}

\bib{Costello-Santos}{inproceedings}{
      author={Corte-Real~Santos, Maria},
      author={Costello, Craig},
      author={Frengley, Sam},
       title={An algorithm for efficient detection of (n, n)-splittings and its
  application to the isogeny problem in dimension 2},
        date={2024},
   booktitle={Public-key cryptography -- pkc 2024},
      editor={Tang, Qiang},
      editor={Teague, Vanessa},
   publisher={Springer Nature Switzerland},
     address={Cham},
       pages={157\ndash 189},
}

\bib{c-r}{article}{
      author={Cosset, Romain},
      author={Robert, Damien},
       title={Computing {$(\ell,\ell)$}-isogenies in polynomial time on
  {J}acobians of genus {$2$} curves},
        date={2015},
        ISSN={0025-5718},
     journal={Math. Comp.},
      volume={84},
      number={294},
       pages={1953\ndash 1975},
         url={https://doi.org/10.1090/S0025-5718-2014-02899-8},
      review={\MR{3335899}},
}

\bib{Costello-Smith}{article}{
      author={Costello, Craig},
      author={Smith, Benjamin},
       title={The supersingular isogeny problem in genus 2 and beyond},
        date={2020},
     journal={Post-Quantum Cryptography (PQCrypto 2020)},
      volume={12100},
       pages={151\ndash 168},
        note={MR4139650},
}

\bib{MR4139650}{incollection}{
      author={Costello, Craig},
      author={Smith, Benjamin},
       title={The supersingular isogeny problem in genus 2 and beyond},
        date={[2020] \copyright 2020},
   booktitle={Post-quantum cryptography},
      series={Lecture Notes in Comput. Sci.},
      volume={12100},
   publisher={Springer, Cham},
       pages={151\ndash 168},
         url={https://doi.org/10.1007/978-3-030-44223-1_9},
      review={\MR{4139650}},
}

\bib{de-feo-1}{article}{
      author={De~Feo, Luca},
      author={Hugounenq, Cyril},
      author={Pl\^ut, J\'er\^ome},
      author={Schost, \'Eric},
       title={Explicit isogenies in quadratic time in any characteristic},
        date={2016},
        ISSN={1461-1570},
     journal={LMS J. Comput. Math.},
      volume={19},
      number={suppl. A},
       pages={267\ndash 282},
         url={https://doi.org/10.1112/S146115701600036X},
      review={\MR{3540960}},
}

\bib{MR4638168}{incollection}{
      author={Decru, Thomas},
      author={Kunzweiler, Sabrina},
       title={Efficient computation of {$(3^n,3^n)$}-isogenies},
        date={[2023] \copyright 2023},
   booktitle={Progress in cryptology---{AFRICACRYPT} 2023},
      series={Lecture Notes in Comput. Sci.},
      volume={14064},
   publisher={Springer, Cham},
       pages={53\ndash 78},
         url={https://doi.org/10.1007/978-3-031-37679-5_3},
      review={\MR{4638168}},
}

\bib{Delfs-Galbraith}{article}{
      author={Delfs, Christina},
      author={Galbraith, Steven~D.},
       title={Computing isogenies between supersingular elliptic curves over
  \(\mathbb{F}_p\)},
        date={2016},
     journal={Designs, Codes and Cryptography},
      volume={78},
      number={1},
       pages={329\ndash 346},
}

\bib{D-L}{incollection}{
      author={Dolgachev, I.},
      author={Lehavi, D.},
       title={On isogenous principally polarized abelian surfaces},
        date={2008},
   booktitle={Curves and abelian varieties},
      series={Contemp. Math.},
      volume={465},
   publisher={Amer. Math. Soc., Providence, RI},
       pages={51\ndash 69},
         url={https://doi.org/10.1090/conm/465/09100},
      review={\MR{2457735}},
}

\bib{Florit-Smith}{misc}{
      author={Florit, Enric},
      author={Smith, Benjamin},
       title={Random and zig-zag sampling for genus-2 isogeny graphs},
        date={2022},
         url={https://eprint.iacr.org/2022/1700},
        note={Accessed via ePrint},
}

\bib{MR871067}{book}{
      author={Freitag, E.},
       title={Siegelsche {M}odulfunktionen},
      series={Grundlehren der Mathematischen Wissenschaften [Fundamental
  Principles of Mathematical Sciences]},
   publisher={Springer-Verlag, Berlin},
        date={1983},
      volume={254},
        ISBN={3-540-11661-3},
         url={https://doi.org/10.1007/978-3-642-68649-8},
      review={\MR{871067}},
}

\bib{Frey}{incollection}{
      author={Frey, Gerhard},
       title={Isogenies in theory and praxis},
        date={2014},
   booktitle={Open problems in mathematics and computational science},
   publisher={Springer, Cham},
       pages={37\ndash 68},
      review={\MR{3330877}},
}

\bib{FK1}{inproceedings}{
      author={Frey, Gerhard},
      author={Kani, Ernst},
       title={Correspondences on hyperelliptic curves and applications to the
  discrete logarithm},
        date={2011},
      editor={Bouvry, M., P. ands~Klopotek},
      editor={Leprevost, F.},
      editor={Marciniak, M.},
      editor={Mykowiecka, A.},
      editor={Rybinski, H.},
      series={LNCS},
      volume={7053},
       pages={1\ndash 19},
}

\bib{MR2372155}{article}{
      author={Gaudry, P.},
       title={Fast genus 2 arithmetic based on theta functions},
        date={2007},
        ISSN={1862-2976},
     journal={J. Math. Cryptol.},
      volume={1},
      number={3},
       pages={243\ndash 265},
         url={https://doi.org/10.1515/JMC.2007.012},
      review={\MR{2372155}},
}

\bib{MR1438983}{article}{
      author={Gritsenko, V.~A.},
      author={Nikulin, V.~V.},
       title={Igusa modular forms and ``the simplest'' {L}orentzian
  {K}ac-{M}oody algebras},
        date={1996},
        ISSN={0368-8666},
     journal={Mat. Sb.},
      volume={187},
      number={11},
       pages={27\ndash 66},
         url={https://doi.org/10.1070/SM1996v187n11ABEH000171},
      review={\MR{1438983}},
}

\bib{HM95}{article}{
      author={Hashimoto, Ki-ichiro},
      author={Murabayashi, Naoki},
       title={Shimura curves as intersections of {H}umbert surfaces and
  defining equations of {QM}-curves of genus two},
        date={1995},
        ISSN={0040-8735},
     journal={Tohoku Math. J. (2)},
      volume={47},
      number={2},
       pages={271\ndash 296},
         url={https://doi.org/10.2748/tmj/1178225596},
      review={\MR{1329525}},
}

\bib{Humbert1901}{article}{
      author={Humbert, Georges},
       title={Sur les fonctionnes ab\'eliennes singuli\`eres. {I}, {II},
  {III}},
        date={1901},
     journal={J. Math. Pures Appl. serie 5},
       pages={V, 233\ndash 350 (1899); VI, 279\ndash 386 (1900); VII, 97\ndash
  123 (1901)},
}

\bib{Ig}{article}{
      author={Igusa, Jun-ichi},
       title={Arithmetic variety of moduli for genus two},
        date={1960},
        ISSN={0003-486X},
     journal={Ann. of Math. (2)},
      volume={72},
       pages={612\ndash 649},
         url={https://doi.org/10.2307/1970233},
      review={\MR{0114819}},
}

\bib{MR0141643}{article}{
      author={Igusa, Jun-ichi},
       title={On {S}iegel modular forms of genus two},
        date={1962},
        ISSN={0002-9327},
     journal={Amer. J. Math.},
      volume={84},
       pages={175\ndash 200},
         url={https://doi.org/10.2307/2372812},
      review={\MR{0141643}},
}

\bib{MR0168805}{article}{
      author={Igusa, Jun-Ichi},
       title={On {S}iegel modular forms genus two. {II}},
        date={1964},
        ISSN={0002-9327},
     journal={Amer. J. Math.},
      volume={86},
       pages={392\ndash 412},
      review={\MR{0168805}},
}

\bib{MR0229643}{article}{
      author={Igusa, Jun-ichi},
       title={Modular forms and projective invariants},
        date={1967},
        ISSN={0002-9327},
     journal={Amer. J. Math.},
      volume={89},
       pages={817\ndash 855},
         url={https://doi.org/10.2307/2373243},
      review={\MR{0229643}},
}

\bib{MR527830}{article}{
      author={Igusa, Jun-ichi},
       title={On the ring of modular forms of degree two over {${\bf Z}$}},
        date={1979},
        ISSN={0002-9327},
     journal={Amer. J. Math.},
      volume={101},
      number={1},
       pages={149\ndash 183},
         url={https://doi.org/10.2307/2373943},
      review={\MR{527830}},
}

\bib{MR3238326}{article}{
      author={Koike, Kenji},
       title={On {J}acobian {K}ummer surfaces},
        date={2014},
        ISSN={0025-5645},
     journal={J. Math. Soc. Japan},
      volume={66},
      number={3},
       pages={997\ndash 1016},
         url={https://doi.org/10.2969/jmsj/06630997},
      review={\MR{3238326}},
}

\bib{kumar}{article}{
      author={Kumar, Abhinav},
       title={Hilbert modular surfaces for square discriminants and elliptic
  subfields of genus 2 function fields},
        date={2015},
        ISSN={2522-0144,2197-9847},
     journal={Res. Math. Sci.},
      volume={2},
       pages={Art. 24, 46},
         url={https://doi.org/10.1186/s40687-015-0042-9},
      review={\MR{3427148}},
}

\bib{MR1579281}{article}{
      author={Kummer, E.~E.},
       title={\"{U}ber die {F}l\"achen vierten {G}rades, auf welchen {S}chaaren
  von {K}egelschnitten liegen},
        date={1865},
        ISSN={0075-4102},
     journal={J. Reine Angew. Math.},
      volume={64},
       pages={66\ndash 76},
         url={http://dx.doi.org/10.1515/crll.1865.64.66},
      review={\MR{1579281}},
}

\bib{MR2409557}{incollection}{
      author={Kuwata, Masato},
      author={Shioda, Tetsuji},
       title={Elliptic parameters and defining equations for elliptic
  fibrations on a {K}ummer surface},
        date={2008},
   booktitle={Algebraic geometry in {E}ast {A}sia---{H}anoi 2005},
      series={Adv. Stud. Pure Math.},
      volume={50},
   publisher={Math. Soc. Japan, Tokyo},
       pages={177\ndash 215},
         url={https://doi.org/10.2969/aspm/05010177},
      review={\MR{2409557}},
}

\bib{lombardo}{article}{
      author={Lombardo, Davide},
       title={Computing the geometric endomorphism ring of a genus 2 jacobian},
        date={2016},
      eprint={1610.09674},
         url={https://arxiv.org/abs/1610.09674},
}

\bib{LR}{article}{
      author={Lubicz, David},
      author={Robert, Damien},
       title={Arithmetic on abelian and {K}ummer varieties},
        date={2016},
        ISSN={1071-5797},
     journal={Finite Fields Appl.},
      volume={39},
       pages={130\ndash 158},
         url={https://doi.org/10.1016/j.ffa.2016.01.009},
      review={\MR{3475546}},
}

\bib{2005-1}{article}{
      author={Magaard, Kay},
      author={Shaska, Tanush},
      author={V\"olklein, Helmut},
       title={Genus 2 curves that admit a degree 5 map to an elliptic curve},
        date={2009},
        ISSN={0933-7741,1435-5337},
     journal={Forum Math.},
      volume={21},
      number={3},
       pages={547\ndash 566},
         url={https://doi.org/10.1515/FORUM.2009.027},
      review={\MR{2526800}},
}

\bib{MR3712162}{article}{
      author={Malmendier, A.},
      author={Shaska, T.},
       title={The {S}atake sextic in {F}-theory},
        date={2017},
        ISSN={0393-0440},
     journal={J. Geom. Phys.},
      volume={120},
       pages={290\ndash 305},
         url={https://doi.org/10.1016/j.geomphys.2017.06.010},
      review={\MR{3712162}},
}

\bib{2019-4}{article}{
      author={Malmendier, A.},
      author={Shaska, T.},
       title={From hyperelliptic to superelliptic curves},
        date={2019},
        ISSN={1930-1235},
     journal={Albanian J. Math.},
      volume={13},
      number={1},
       pages={107\ndash 200},
      review={\MR{3978315}},
}

\bib{MR3731039}{article}{
      author={Malmendier, Andreas},
      author={Shaska, Tony},
       title={A universal genus-two curve from {S}iegel modular forms},
        date={2017},
     journal={SIGMA Symmetry Integrability Geom. Methods Appl.},
      volume={13},
       pages={Paper No. 089, 17},
         url={https://doi.org/10.3842/SIGMA.2017.089},
      review={\MR{3731039}},
}

\bib{2025-07}{article}{
      author={Mello, J.},
      author={Salami, S.},
      author={Shaska, E.},
      author={Shaska, T.},
       title={Rational points and zeta functions of {H}umbert surfaces with
  square discriminant},
        date={2025},
      eprint={2504.19268},
         url={https://arxiv.org/abs/2504.19268},
}

\bib{MR728142}{article}{
      author={Morrison, D.~R.},
       title={On {$K3$} surfaces with large {P}icard number},
        date={1984},
        ISSN={0020-9910},
     journal={Invent. Math.},
      volume={75},
      number={1},
       pages={105\ndash 121},
         url={https://mathscinet.ams.org/mathscinet-getitem?mr=728142},
      review={\MR{728142}},
}

\bib{Mum}{book}{
      author={Mumford, David},
       title={Abelian varieties},
      series={Tata Institute of Fundamental Research Studies in Mathematics},
   publisher={Published for the Tata Institute of Fundamental Research, Bombay;
  by Hindustan Book Agency, New Delhi},
        date={2008},
      volume={5},
        ISBN={978-81-85931-86-9; 81-85931-86-0},
        note={With appendices by C. P. Ramanujam and Yuri Manin, Corrected
  reprint of the second (1974) edition},
      review={\MR{2514037}},
}

\bib{MR2514037}{book}{
      author={Mumford, David},
       title={Abelian varieties},
      series={Tata Institute of Fundamental Research Studies in Mathematics},
   publisher={Tata Institute of Fundamental Research, Bombay; by Hindustan Book
  Agency, New Delhi},
        date={2008},
      volume={5},
        ISBN={978-81-85931-86-9; 81-85931-86-0},
        note={With appendices by C. P. Ramanujam and Yuri Manin, Corrected
  reprint of the second (1974) edition},
      review={\MR{2514037}},
}

\bib{MR1013073}{article}{
      author={Oguiso, Keiji},
       title={On {J}acobian fibrations on the {K}ummer surfaces of the product
  of nonisogenous elliptic curves},
        date={1989},
        ISSN={0025-5645},
     journal={J. Math. Soc. Japan},
      volume={41},
      number={4},
       pages={651\ndash 680},
         url={https://doi.org/10.2969/jmsj/04140651},
      review={\MR{1013073}},
}

\bib{Oort}{incollection}{
      author={Oort, Frans},
       title={Endomorphism algebras of abelian varieties},
        date={1988},
   booktitle={Algebraic geometry and commutative algebra, {V}ol.\ {II}},
   publisher={Kinokuniya, Tokyo},
       pages={469\ndash 502},
      review={\MR{977774}},
}

\bib{MR2367218}{article}{
      author={Previato, E.},
      author={Shaska, T.},
      author={Wijesiri, G.~S.},
       title={Thetanulls of cyclic curves of small genus},
        date={2007},
        ISSN={1930-1235},
     journal={Albanian J. Math.},
      volume={1},
      number={4},
       pages={253\ndash 270},
      review={\MR{2367218}},
}

\bib{MR1509868}{article}{
      author={Pringsheim, Alfred},
       title={Zur {T}ransformation zweiten {G}rades der hyperelliptischen
  {F}unctionen erster {O}rdnung},
        date={1876},
        ISSN={0025-5831},
     journal={Math. Ann.},
      volume={9},
      number={4},
       pages={445\ndash 475},
         url={https://doi.org/10.1007/BF01442473},
      review={\MR{1509868}},
}

\bib{MR1578135}{article}{
      author={Richelot, Fried.~Jul.},
       title={De transformatione integralium {A}belianorum primi ordinis
  commentatio. {C}aput secundum. {D}e computatione integralium {A}belianorum
  primi ordinis},
        date={1837},
        ISSN={0075-4102,1435-5345},
     journal={J. Reine Angew. Math.},
      volume={16},
       pages={285\ndash 341},
         url={https://doi.org/10.1515/crll.1837.16.285},
      review={\MR{1578135}},
}

\bib{Serre85}{article}{
      author={Serre, Jean-Pierre},
       title={Rational points on curves over finite fields},
        date={1985},
     journal={notes by F. Gouvea of lectures at Harvard University},
}

\bib{2024-03}{article}{
      author={Shaska, Elira},
      author={Shaska, Tanush},
       title={Machine learning for moduli space of genus two curves and an
  application to isogeny-based cryptography},
        date={2025},
        ISSN={0925-9899,1572-9192},
     journal={J. Algebraic Combin.},
      volume={61},
      number={2},
       pages={23},
         url={https://doi.org/10.1007/s10801-025-01393-8},
      review={\MR{4870337}},
}

\bib{2000-2}{incollection}{
      author={Shaska, Tanush},
      author={V\"olklein, Helmut},
       title={Elliptic subfields and automorphisms of genus 2 function fields},
        date={2004},
   booktitle={Algebra, arithmetic and geometry with applications ({W}est
  {L}afayette, {IN}, 2000)},
   publisher={Springer, Berlin},
       pages={703\ndash 723},
      review={\MR{2037120}},
}

\bib{2000-1}{article}{
      author={Shaska, Tony},
       title={Curves of genus 2 with \(\langle n, n\rangle\) decomposable
  jacobians},
        date={2001},
     journal={Journal of Symbolic Computation},
      volume={31},
      number={5},
       pages={603\ndash 617},
      review={\MR{1828706}},
}

\bib{2001-1}{article}{
      author={Shaska, Tony},
       title={Genus 2 fields with degree 3 elliptic subfields},
        date={2004},
     journal={Forum Mathematicum},
      volume={16},
      number={2},
       pages={263\ndash 280},
      review={\MR{2039100}},
}

\bib{Shimura}{book}{
      author={Shimura, Goro},
       title={Introduction to the arithmetic theory of automorphic functions},
      series={Publications of the Mathematical Society of Japan},
   publisher={Princeton University Press},
     address={Princeton, NJ},
        date={1971},
      volume={11},
        ISBN={978-0-691-08092-5},
        note={Reprinted by Princeton University Press, 1994},
}

\bib{MR2296439}{incollection}{
      author={Shioda, Tetsuji},
       title={Classical {K}ummer surfaces and {M}ordell-{W}eil lattices},
        date={2007},
   booktitle={Algebraic geometry},
      series={Contemp. Math.},
      volume={422},
   publisher={Amer. Math. Soc., Providence, RI},
       pages={213\ndash 221},
         url={https://doi.org/10.1090/conm/422/08062},
      review={\MR{2296439}},
}

\bib{shioda-inose}{article}{
      author={Shioda, Tetsuji},
      author={Inose, Katsumi},
       title={On singular k3 surfaces},
        date={1977},
     journal={Complex Analysis and Algebraic Geometry},
       pages={119\ndash 136},
         url={https://doi.org/10.1017/CBO9780511662638.008},
        note={A collection of papers dedicated to K. Kodaira},
}

\bib{Smith}{incollection}{
      author={Smith, Benjamin},
       title={Computing low-degree isogenies in genus 2 with the
  {D}olgachev-{L}ehavi method},
        date={2012},
   booktitle={Arithmetic, geometry, cryptography and coding theory},
      series={Contemp. Math.},
      volume={574},
   publisher={Amer. Math. Soc., Providence, RI},
       pages={159\ndash 170},
         url={https://doi.org/10.1090/conm/574/11418},
      review={\MR{2961408}},
}

\end{biblist}
\end{bibdiv}

\end{document}